\documentclass[11pt,reqno]{amsart}
\usepackage[margin=1in]{geometry}
\usepackage{amsmath,amssymb,amsthm,graphicx,amsxtra, setspace}
\usepackage[utf8]{inputenc}
\usepackage{mathrsfs}
\usepackage{hyperref}
\usepackage{upgreek}
\usepackage{mathtools}
\usepackage{stmaryrd}
\allowdisplaybreaks
\usepackage{xcolor}

\usepackage[pagewise]{lineno}

\newtheorem{theorem}{Theorem}[section]
\newtheorem{lemma}[theorem]{Lemma}
\newtheorem{proposition}[theorem]{Proposition}
\newtheorem{assumption}[theorem]{Assumption}

\newtheorem{definition}[theorem]{Definition}

\newtheorem{remark}[theorem]{Remark}
\newtheorem{hypothesis}[theorem]{Hypothesis}

\let\originalleft\left
\let\originalright\right
\renewcommand{\left}{\mathopen{}\mathclose\bgroup\originalleft}
\renewcommand{\right}{\aftergroup\egroup\originalright}

\DeclareMathOperator*{\esssup}{ess\,sup}

\DeclareMathAlphabet{\mathpzc}{OT1}{pzc}{m}{it}

\hypersetup{colorlinks=true,%
	citecolor=red,%
	filecolor=blue,%
	linkcolor=blue,%
}
\usepackage{graphicx}

\renewcommand{\d}{\/\mathrm{d}\/}

\def\w{\textbf{W}^{\varepsilon}_{{\theta}^{\varepsilon}}}

\def\e{\varepsilon}

\def\L{\mathbb{L}}

\def\I{\mathrm{I}}
\def\F{\mathrm{F}}
\def\C{\mathrm{C}}
\def\f{\boldsymbol{f}}

\def\B{\mathrm{B}}
\def\D{\mathrm{D}}
\def\y{\boldsymbol{y}}

\def\E{\mathbb{E}}
\def\X{\mathbb{X}}
\def\x{\boldsymbol{x}}

\def\g{\boldsymbol{g}}

\def\h{\boldsymbol{h}}
\def\z{\boldsymbol{z}}
\def\v{\boldsymbol{v}}
\def\w{\boldsymbol{w}}
\def\be{\mathbf{e}}

\def\M{\mathrm{M}}
\def\N{\mathbb{N}}

\def\no{\nonumber}
\def\PA{\mathrm{P}_{1/m}}
\def\V{\mathbb{V}}
\def\wi{\widetilde}

\def\u{\mathrm{U}}
\def\P{\mathrm{P}}
\def\PP{\mathbb{P}}
\def\u{\boldsymbol{u}}
\def\H{\mathbb{H}}

\newcommand{\R}{\mathbb{R}}

\renewcommand{\d}{\/\mathrm{d}\/}

\newcommand{\Addresses}{{
		\footnote{
			\noindent \textsuperscript{1}Montanuniversit\"at Leoben, Department Mathematik und Informationstechnologie, Franz Josef Strasse 18, 8700
			Leoben, Austria.\par\nopagebreak
			\noindent  \textit{e-mail:} \texttt{ankit.kumar@unileoben.ac.at, ankitkumar.2608@gmail.com.}
			
			\noindent \textsuperscript{2}Department of Mathematics, Indian Institute of Technology Roorkee-IIT Roorkee,
			Haridwar Highway, Roorkee, Uttarakhand 247667, INDIA.\par\nopagebreak
			\noindent  \textit{e-mail:} \texttt{maniltmohan@ma.iitr.ac.in, maniltmohan@gmail.com.}
			
			\noindent \textsuperscript{*}Corresponding author.

			\textit{Key words:} convective Brinkman-Forchheimer equations, strong solution, large deviation principle, weak convergence. 
			
			Mathematics Subject Classification (2020): Primary 60H15; Secondary 35R60, 35Q30, 37L55.

}}}

\begin{document}
	
	
	\title[Stochastic damped Ladyzhenskaya-Smagorinsky equations]{Global well-posedness and small time asymptotics of stochastic Ladyzhenskaya-Smagorinsky equations with damping  on unbounded domains
		\Addresses}
	\author[A. Kumar]{Ankit Kumar\textsuperscript{1}}
	\author[M. T. Mohan ]{Manil T. Mohan\textsuperscript{2*}}

	\maketitle
	
	\begin{abstract}
		The Ladyzhenskaya-Smagorinsky equations model turbulence phenomena,  and  are given by 
		$$	\frac{\partial \boldsymbol{u}}{\partial t}-\mu \mathrm{div}\left(\left(1+|\nabla\boldsymbol{u}|^2\right)^{\frac{p-2}{2}}\nabla\boldsymbol{u}\right)+(\boldsymbol{u}\cdot\nabla)\boldsymbol{u}+\nabla p=\boldsymbol{f}, \ \nabla\cdot\boldsymbol{u}=0,$$
		for $p\geq 2.$   In this work, we consider  the stochastic Ladyzhenskaya-Smagorinsky equations with the damping $\alpha\u+\beta|\boldsymbol{u}|^{r-2}\boldsymbol{u},$  for $r\geq 2$ ($\alpha,\beta\geq 0$), subjected to multiplicative Gaussian noise  in a Poincar\'e domain (which may be bounded or unbounded) $\mathcal{O}\subset\R^d$ ($2\leq d\leq 4$). We show the local monotonicity ($p\geq \frac{d}{2}+1,\ r\geq 2$) as well as global monotonicity ($p\geq 2,\ r\geq 4$) properties of the linear and nonlinear operators, which along with an application of a stochastic version of the Minty-Browder technique imply the existence of a unique pathwise strong solution satisfying the energy equality (It\^o formula), which is proved with the help of the methodology developed in  [Krylov, \emph{Probab. Theory Related Fields}, {\bf 147} (2010), 583--605.] Then, we discuss the small time asymptotics by studying the effect of small, highly nonlinear, unbounded drifts (small time large deviation principle) for the stochastic Ladyzhenskaya-Smagorinsky equations with damping.
	\end{abstract}

	\section{Introduction}\label{sec1}\setcounter{equation}{0}
	In this work, we study the well-posedness of the stochastic Ladyzhenskaya-Smagorinsky equations with damping in Poincaré domains $\mathcal{O}\subset \mathbb{R}^d$, where $2 \leq d \leq 4$. We assume that the domain $\mathcal{O}$ satisfies the following condition:
	\begin{assumption}\label{assumpO}
		Let $\mathcal{O}$ be an open, connected and may be unbounded subset of $\R^d$, the boundary of which is uniformly of class $\mathrm{C}^2$. For the domain $\mathcal{O}$ and any fixed $p\in[1,\infty)$, we also assume that, there exists a positive constant $\lambda_{1,p}$ such that the following Poincar\'e inequality  is satisfied:
		\begin{align}\label{2.1}
			\lambda_{1,p}\int_{\mathcal{O}} |\psi(x)|^p \d x \leq \int_{\mathcal{O}} |\nabla \psi(x)|^p \d x,  \ \text{ for all } \  \psi \in \mathbb{W}^{1,p}_0 (\mathcal{O}).
		\end{align}
	\end{assumption}
	A domain in which Poincar\'e inequality is satisfied, we call it as a \emph{Poincar\'e domain} and if  $\mathcal{O}$ is bounded in some direction, then the Poincar\'e inequality \eqref{2.1} holds. For example, one can consider  $\mathcal{O}=\R^{d-1}\times(-L,L),$ $L>0$ (see \cite[p.306]{Te} and
	\cite[p.117]{Robinson3} for more details).

\subsection{The model}
	In the works \cite{OAL1,OAL2,OAL}, etc., Olga Ladyzhenskaya proposed a new set of equations, which describe turbulence phenomena. 	Let $\mathcal{O}\subset\R^d$ ($2\leq d\leq 4$) be a  Poincar\'e domain. Let $\u(t , x) \in \R^d$ be the velocity field at time $t$ and position $x$, $p(t,x)\in\R$ be the pressure, $\f(t,x)\in\R^d$ be an external forcing.  The model described by Ladyzhenskaya is given by (see \cite{HBV})
	\begin{equation}\label{1.1}
		\left\{
		\begin{aligned}
			\frac{\partial\u}{\partial t}+(\u\cdot\nabla)\u-\nabla\cdot\mathbb{T}(\u,p)&=\f,\ \text{ in }\ \mathcal{O}\times(0,T),\\
			\nabla\cdot\u&=0,
		\end{aligned}
		\right.
	\end{equation}
	where $\mathbb{T}$ is the stress tensor
	\begin{align}
		\mathbb{T}&=-p\mathrm{I}+\nu_{\mathbb{T}}(\u)\mathcal{D}\u,\ \mathcal{D}\u=\frac{1}{2}(\nabla\u+(\nabla\u)^{\top}),\label{1.2}\\
		\nu_{\mathbb{T}}(\u)&=\nu_0+\nu_1|\mathcal{D}\u|^{p-2},\label{1.3}
	\end{align}
	and $\nu_0,\nu_1$ are strictly positive constants. The system \eqref{1.1} is supplemented by the following initial  condition:
	\begin{align}
		\u(0)=\u_0 \ \text{ in } \ \mathcal{O}.\label{1.5}
	\end{align}
	The system \eqref{1.1} satisfies the Stokes principle (cf. \cite[Appendix 1]{HBV}). Note that the equation \eqref{1.2} tells us that the stress tensor $\mathbb{T}$ depends on the symmetric part $\mathcal{D}\u$ of the gradient of the velocity  has a polynomial growth of $p$-order, for $p > 2$. Taking $p = 3$, we obtain  the classical Smagorinsky turbulence model introduced  in \cite{JSS} (see \cite{HBV,TJRH,CPa}, etc). J.-L. Lions in  \cite{JLL1} and \cite[Chapter 2, Section 5]{JLL2} considered the case in which $\mathcal{D}\u$ is replaced by $\nabla\u$, but  in this case the Stokes principle is not satisfied. The existence of a global weak solution to the system \eqref{1.1}-\eqref{1.5} is known due to  Ladyzhenskaya for $p\geq 1+\frac{2d}{d+2}$ and its uniqueness for $p\geq 1+\frac{d}{2}$.  For further results on the global solvability results  (the existence and uniqueness of weak as well as strong solutions) on the deterministic Ladyzhenskaya-Smagorinsky equations, the interested readers are referred to see \cite{HBV,OAL1,OAL2,OAL,JLL2,JMJN1,JMJN,JMJN2}, etc. 
	
	In this work, we consider a stochastic counterpart of the following form of Ladyzhenskaya-Smagorinsky equations (cf. \cite{JLL1,PNKT})  with damping: 
	\begin{equation}\label{1}
		\left\{
		\begin{aligned}
			\frac{\partial \u}{\partial t}-\mu \mathrm{div}\Big(\left(1+|\nabla\boldsymbol{u}|^2\right)^{\frac{p-2}{2}}\nabla\boldsymbol{u}\Big)+(\u\cdot\nabla)\u&+\alpha\u+\beta|\u|^{r-2}\u\\+\nabla p&=\boldsymbol{f}, \ \text{ in } \ \mathcal{O}\times(0,T), \\ \nabla\cdot\u&=0, \ \text{ in } \ \mathcal{O}\times(0,T), \\
			\u(0)&=\u_0, \ \text{ in } \ \mathcal{O},
		\end{aligned}
		\right.
	\end{equation}
	where $\alpha,\beta\geq 0$ are damping parameters. The model under our consideration can be described as 
	\begin{equation}\label{31}
		\left\{
		\begin{aligned}
			\d\u(t)+\Big[-\mu \mathrm{div}\Big(\big(1+|\nabla\boldsymbol{u}(t)|^2&\big)^{\frac{p-2}{2}}\nabla\boldsymbol{u}(t)\Big)+(\u(t)\cdot\nabla)\u(t)+\alpha\u(t)+\beta|\u(t)|^{r-2}\u(t)\\+\nabla p(t)\Big]\d t&=\f(t)\d t+\sum_{k=1}^{\infty}\boldsymbol{\sigma}_k(t,\u(t))\d\boldsymbol{W}_k(t), \ \text{ in } \ \mathcal{O}\times(0,T), \\ \nabla\cdot\u(t)&=0, \ \text{ in } \ \mathcal{O}\times(0,T), \\
			\u(0)&=\u_0, \ \text{ in } \ \mathcal{O},
		\end{aligned}
		\right.
	\end{equation} 
	where $\{\boldsymbol{W}_k(t)\}_{k\geq 1}$ is a sequence of independent $\mathscr{F}_t$-adapted Brownian motions defined on a filtered probability space $(\Omega,\mathscr{F},\{\mathscr{F}_t\}_{t\geq 0},\mathbb{P})$. As $\alpha$ does not play a major role in our analysis, we fix $\alpha=0$ in the rest of the paper. 
	 Our first goal in this work is to establish the existence of a unique pathwise strong solution.   Our second goal is to discuss the small time asymptotics by studying the effect of small, highly nonlinear, unbounded drifts (small time large deviation principle) for the system \eqref{31}. More specifically, our focus on the limiting behavior of the strong solution to the Ladyzhenskaya-Smagorinsky equations with damping in a time interval $[0, t]$ as $t$ tends to zero. An inspiration for considering such problems arises from Varadhan's identity 
	\begin{align*}
		\lim_{t\to 0}2t\log\mathbb{P}\left\{\u(0)\in\B,\ \u(t)\in\C\right\}=-d^2(\B,\C),
	\end{align*}
	where $\u$ is the strong solution to the Ladyzhenskaya-Smagorinsky equations with damping and $d $ is an appropriate Riemann distance associated with the diffusion generated by $\u$. 
	
	\subsection{Literature survey}
	The existence of martingale weak solution for an incompressible non-Newtonian fluid equations (stochastic power law fluids)  driven by a Brownian motion is proved in \cite{DBr}. Recently, the authors in \cite{AKMTM10} established the existence and uniqueness of weak solution for the generalized stochastic Navier-Stokes-Voigt equations. The existence and uniqueness of weak solutions for stochastic power law fluids is obtained in \cite{YTNY}.   The existence and uniqueness of strong solutions to coercive and locally monotone stochastic partial differential equations (SPDEs) like stochastic equations of non-Newtonian fluids driven by L\'evy noise is proved in \cite{ZBWL}. The local and global existence and uniqueness of solutions for general nonlinear evolution equations with coefficients satisfying some local monotonicity and generalized coercivity conditions  is established in \cite{WLMR}. The existence of random dynamical systems and random attractors for a large class of locally monotone SPDEs perturbed by additive L\'evy noise, which includes stochastic power law fluids, is obtained in \cite{BGWL}. The existence of global pathwise strong solutions for the two dimensional stochastic non-Newtonian incompressible fluid equations is showed in \cite{MHPN}. The existence and uniqueness of pathwise strong solutions to convective Brinkman-Forchheimer equations perturbed by Gaussian and pure jump noise is obtained in \cite{MTM8,MTM10}, respectively. The well-posedness for a class of fully local monotone SPDEs perturbed by Gaussian noise has been established in \cite{MRSSTZ} and the extension to L\'evy noise is obtained in \cite{AKMTM4}. The existence and uniqueness of strong solutions to stochastic 3D tamed Navier-Stokes equations governed by Gaussian noise and L\'evy noise is proved in \cite{MRTZ,ZDRZ}, respectively. In the articles \cite{GVAZ,RKAKM}, the authors established the well-posedness and LDP for stochastic evolution equation ($p$-Laplace type equation) driven by a multiplicative Gaussian noise, respectively. The global existence of both martingale and pathwise strong solutions of stochastic equations with a monotone operator, of the Ladyzhenskaya-Smagorinsky type, governed by a general L\'evy noise is established in \cite{PNKT}. 
	
	Large deviation theory gained its attention in the past few decades due to its wide range of applications in the areas like mathematical finance, risk management,  fluid mechanics, statistical mechanics, quantum physics, etc (cf. \cite{BD2,Chow1,DZ,FW,AKMTM5,MTM6,MTM12,RWW,SSSP,DWS,Va,SRSV,SRSV1}, etc, and the references therein). The small time large deviation principle (LDP) examines the asymptotic behavior of the tails of a family of probability measures at a given point in space when the time is very small. That is, the limiting behavior of the solution in a time interval $[0, t]$ as $t\to 0$. In this direction, the first and most celebrated work is due to Varadhan \cite{SRSV}, where he considered the small-time asymptotics for finite-dimensional diffusion processes. For the small-time LDP for  infinite dimensional diffusion processes, the interested readers are referred to see \cite{HAb,TAMH,ZQC,MHKM,AKMTM6,MRTZ1,SRSV2,TSZ}, etc. The small time LDP for stochastic  2D  Navier-Stokes equations, stochastic 3D tamed Navier-Stokes equations, stochastic quasi-geostrophic equations in the sub-critical case, stochastic two-dimensional non-Newtonian fluids, 3D stochastic primitive equations, SPDEs with locally monotone coefficients, stochastic convective Brinkman-Forchheimer equations, scalar stochastic conservation laws, are established in \cite{TXTZ,MRTZ,WLMR1,HLCS,ZDRZ1,SLWL,MTM12,ZDRZ2}, respectively. Even though the work \cite{SLWL} covers the case of SPDEs with locally monotone coefficients like stochastic power law fluid equations, it won't cover the system under our consideration for arbitrary values of $r$ (for example $r>\frac{pd}{d-p}$). Therefore, we need a separate analysis for the small noise asymptotics of the  system \eqref{31}. 
	
	\subsection{Novelties, difficulties and approaches}
	This article is dedicated to addressing two primary objectives. The first objective is to establish the existence and uniqueness of solutions to the stochastic Ladyzhenskaya-Smagorinsky equations, denoted as \eqref{31}, which include damping and are driven by a multiplicative Gaussian noise. The second objective focuses on deriving the small-time LDP for these equations with damping. To the best of the authors' knowledge, the analysis and results presented in this article have not been previously explored in the existing literature, thereby contributing novel insights to the field. The novelties of this article are: 
	\begin{enumerate}
		\item The existence and uniqueness of strong solution to the system \eqref{31} (with appropriate assumptions on the exponents and coefficients see  below) with the initial data $\u_0\in\H$ is obtained in the space 
			\begin{align*}
			\u&\in\mathrm{L}^{q}(\Omega;\mathrm{L}^{\infty}(0,T;\H))\cap{\mathrm{L}^{p}(\Omega;\mathrm{L}^p  (0,T;\V_p))}\cap\mathrm{L}^{r}(\Omega;\mathrm{L}^{r}(0,T;\widetilde{\L}^{r})),
		\end{align*} 
		for  $q\in\left\{2, \frac{2(2p-d)}{p-1},{4+\xi}\right\} $, for some $\xi>0$ with $\mathbb{P}$-a.s., continuous trajectories in $\H$.
		
		\item Motivated by the works \cite{CLF,MTM8,NVK}, the energy equality (It\^o formula) for the solution to the system \eqref{31} is derived by  constructing an approximation of the strong solution within a finite-dimensional subspace spanned by the first $m$ eigenfunctions of an appropriately chosen compact operator (defined in Section \ref{ACO}).
		\item The small time LDP for the laws of the solutions to the system \eqref{31} is also obtained.
	\end{enumerate}

		\vspace{2mm}
	\noindent
\textbf{Summary of the results:}	Let us summarize the results obtained in the article to emphasize the dependence on the different parameters appearing in the system \eqref{31}:
	
	\vspace{1mm}
	\noindent
	For our first objective, we build upon the works \cite{chow,CLF,MTM8,NVK,MTM6}, and to establish the existence and uniqueness results, we employ a localized stochastic generalization of the Minty-Browder technique.
	\begin{enumerate}
		\item We first establish the hemicontinuity and the monotonicity properties of the operators $\mathscr{H}(\cdot)=\mu\mathcal{A}(\cdot)+\B(\cdot)+\beta \mathcal{C}(\cdot)$ and $\mathscr{H}(\cdot)+\eta\I(\cdot)$, with the following assumptions on the exponents and coefficients:
		\begin{enumerate}
			\item The operator $\mathscr{H}(\cdot)+\eta\I(\cdot)$ is globally monotone for $p\geq 2$ and $r\geq 4$ ($2\beta\mu\geq 1$ for $r=4$) (see \eqref{mon} and \eqref{2.18});
			\item The operator $\mathscr{H}(\cdot)$ is locally monotone for $p>\frac{d}{2}$ and $r\geq 2$ (see \eqref{mon1}). 
		\end{enumerate}
		\item Later, we demonstrate the existence of global strong solution followed by the uniqueness of the solution to the system \eqref{31} under the following conditions:
		\begin{enumerate}
			\item In the case of global monotone operator, we obtain the existence and uniqueness results for $p\geq 2$ and $r\geq 4$ ($2\beta\mu\geq 1$ for $r=4$) (see Theorem \ref{exis2});
			\item In the case of local monotone operators, we obtain the existence and uniqueness results for $p\geq \frac{d}{2}+1$, and $r\geq2$ (see Section \ref{EUSS}).
		\end{enumerate}
	\end{enumerate} 

	\vspace{1mm}
\noindent
In order to achieve our second objective, we assume that Hypothesis \ref{hyp1} holds (additional assumption on the noise coefficient), and our approach is based on the works \cite{AKMTM6,TSZ} (cf. \cite[Theorem 4.2.13]{DZ}). We demonstrate the small time LDP under the following conditions:
\begin{enumerate}
	\item In the case of global monotone operator, we establish the small time LDP   for $p\geq 2$ and $r\geq 4$ ($2\beta\mu\geq 1$ for $r=4$) (see Theorem \ref{maint},  Remark \ref{rem5.9} and Case 2 in both Lemmas  \ref{lem3.10} and \ref{lem3.11});
		\item In the case of local monotone operators, we establish the small time LDP  for $p\geq \frac{d}{2}+1$, and $r\geq2$ (see Theorem \ref{maint} and Case 1 in both Lemmas \ref{lem3.10} and \ref{lem3.11}).
\end{enumerate}

%
	
	\subsection{Organization of the paper}
	The rest of the paper is organized as follows: In the next section, we provide the essential  functional setting to obtain the global solvability results of the system \eqref{31}. Then we discuss the local monotonicity ($p\geq \frac{d}{2}+1,\ r\geq 2$), global monotonicity ($p\geq 2,\ r\geq 4$,  $2\beta\mu\geq 1$ for $r=4$) and demicontinuity properties of the linear and nonlinear operators (see Lemmas \ref{lem2.2}, \ref{lem2.5} and \ref{lem2.8}). Using these properties and under proper assumptions on the noise coefficient, we establish the existence of a pathwise strong solution to the system \eqref{31} in Sections \ref{sec3} and \ref{EUSS} (see Theorem \ref{exis2}). Under an additional assumption on the noise coefficient, we discuss the small time asymptotics of the  system \eqref{31}, 
	 by studying the effect of small, highly nonlinear, unbounded drifts (small time LDP) in the final section (see Theorem \ref{maint}).

	\section{Mathematical formulation}\label{sec2}\setcounter{equation}{0}
	This section is devoted for the necessary function spaces needed to obtain the global solvability results of the system \eqref{31}.  Furthermore, we discuss important properties of the nonlinear operators.

	\subsection{Function spaces} Let $\mathcal{O}\subset\R^d$ ($2\leq d\leq 4$) be a Poincar\'e domain. Let $\C_0^{\infty}(\mathcal{O};\R^d)$ denote the space of all infinitely differentiable functions  ($\R^d$-valued) with compact support in $\mathcal{O}\subset\R^d$.  The Lebesgue spaces are denoted by $\mathbb{L}^r(\mathcal{O})=\mathrm{L}^r(\mathcal{O};\R^d)$ and Sobolev spaces are represented by $\mathbb{W}^{k,p}(\mathcal{O}):=\mathrm{W}^{k,p}(\mathcal{O};\R^d)$ and $\H^k(\mathcal{O}):=\mathbb{W}^{k,2}$. We define 
	\begin{align*} 
		\mathscr{V}&:=\{\u\in\C_0^{\infty}(\mathcal{O},\R^d):\nabla\cdot\u=0\},\\
		\mathbb{H}&:=\text{the closure of}\ \mathscr{V} \ \text{in the Lebesgue space } \L^2(\mathcal{O}),\\
		\mathbb{V}&:=\text{the closure of}\ \mathscr{V} \ \text{in the Sobolev space } \H^1_0(\mathcal{O}),\\
		\widetilde{\L}^{r}&:=\text{the closure of}\ \mathscr{V} \ \text{in the Lebesgue space } \L^r(\mathcal{O}),\\
		\mathbb{V}_p&:=\text{the closure of}\ \mathscr{V} \ \text{in the Sobolev space } \mathbb{W}_0^{1,p}(\mathcal{O}),\\
		\mathcal{V}^s&:= \text{the closure of}\ \mathscr{V} \ \text{in the Sobolev space } \H^s(\mathcal{O}),
	\end{align*}
	for $p,r\in(2,\infty)$ and $s\in (0,\infty)$. Then under the above smoothness assumptions on the boundary, we characterize the spaces $\H$, $\V$,  $\widetilde{\L}^r$ and $\V_p$ as 
	$
	\H=\{\u\in\L^2(\mathcal{O}):\nabla\cdot\u=0\}$,  with norm  $\|\u\|_{\H}^2:=\int_{\mathcal{O}}|\u(x)|^2\d x,
	$ and
	$
	\V=\{\u\in\H^1(\mathcal{O}):\nabla\cdot\u=0\},$  with norm $ \|\u\|_{\V}^2:=\int_{\mathcal{O}}|\nabla\u(x)|^2\d x$ (using Poincar\'e inequality \eqref{2.1}),
	 $\widetilde{\L}^r=\{\u\in\L^r(\mathcal{O}):\nabla\cdot\u=0\},$ with norm $\|\u\|_{\widetilde{\L}^r}^r=\int_{\mathcal{O}}|\u(x)|^r\d x$, and $\V_p=\big\{\u\in\mathbb{W}^{1,p}_0(\mathcal{O}):\nabla\cdot\u=0\big\},$ with norm $\|\u\|_{\V_p}^p=\int_{\mathcal{O}}|\nabla\u(x)|^p\d x$ (using Poincar\'e inequality \eqref{2.1}), respectively.	Let $(\cdot,\cdot)$ denote the inner product in the Hilbert space $\H$ and $\langle \cdot,\cdot\rangle $ denote the induced duality between the spaces $\V$  and its dual $\V'$, $\V_p$ and its dual $\V_p'$, and $\widetilde{\L}^r$ and its dual $\widetilde{\L}^{r'},$ where $\frac{1}{r}+\frac{1}{r'}=1$. Note that $\H$ can be identified with its own dual $\H'$. The inner product in any Hilbert space $\V$ will be denoted by $(\cdot,\cdot)_{\V}$.   We endow the space $\V_p\cap\widetilde{\L}^{r}$ with the norm $\|\u\|_{\V_p}+\|\u\|_{\widetilde{\L}^{r}},$ for $\u\in\V_p\cap\widetilde{\L}^r$ and its dual $\V_p'+\widetilde{\L}^{r'}$ with the norm $$\inf\left\{\|\u_1\|_{\V'_p}+\|\u_2\|_{\wi\L^{r'}}:\u=\u_1+\u_2,\  \u_1\in\V'_p \ \text{ and } \ \u_2\in\wi\L^{r'}\right\}.$$   We first note that $\mathscr{V}\subset\V_p\cap\widetilde{\L}^{r}\subset\H$ and $\mathscr{V}$ is dense in $\H,\V_p$ and $\widetilde{\L}^{r},$ and hence $\V_p\cap\widetilde{\L}^{r}$ is dense in $\H$. We have the following continuous  embedding also:
	$$\V_p\cap\widetilde{\L}^{r}\hookrightarrow\H\equiv\H'\hookrightarrow\V_p'+\widetilde\L^{\frac{r}{r-1}}.$$ One can define equivalent norms on $\V_p\cap\widetilde\L^{r}$ and $\V_p'+\widetilde\L^{\frac{r}{r-1}}$ as  (see \cite{NAEG})
	\begin{align*}
		\|\u\|_{\V_p\cap\widetilde\L^{r}}=\left(\|\u\|_{\V_p}^2+\|\u\|_{\widetilde\L^{r}}^2\right)^{\frac{1}{2}}\ \text{ and } \ \|\u\|_{\V_p'+\widetilde\L^{\frac{r}{r-1}}}=\inf_{\u=\v+\w}\left(\|\v\|_{\V_p'}^2+\|\w\|_{\widetilde\L^{\frac{r}{r-1}}}^2\right)^{\frac{1}{2}}.
	\end{align*}
	
	\subsection{A compact operator}\label{ACO} The content of this subsection has been taken from \cite[Section 2.3]{ZBEM}. Consider the natural embedding $j:\V\to\H$ and its adjoint $j^*:\H\to\V$. Since the range of $j$ is dense in $\H$, the mapping $j^*$ is one-to-one. Let us define
	\begin{align}\label{2.2.1}\nonumber
		\D(\mathtt{A})&:=j^*(\H)\subset \V,\\
		\mathtt{A}\x&:=(j^*)^{-1}\x, \ \x\in \D(\mathtt{A}).
	\end{align}
	Note that for all $\x\in\D(\mathtt{A})$ and $\y\in\V$
\begin{align*}
	(\mathtt{A}\x,\y)_\H=(\x,\y)_{\V}.
\end{align*}
For any $s>2$,  it is clear that $\mathcal{V}^s$  is dense in $\V$ and the embedding $j_s:\mathcal{V}^s \hookrightarrow \V$ is continuous. Then, there exists a Hilbert space $\mathbb{U}$ (cf.  \cite[Lemma C.1]{ZBEM} or \cite{KHMW}) such that $\mathbb{U}\subset \mathcal{V}^s$, $\mathbb{U}$ is dense in $\mathcal{V}^s$  and 
\begin{align*}
	\text{the natural embedding } \iota_s:\mathbb{U}\hookrightarrow \mathcal{V}^s \text{ is compact.} 
\end{align*}It implies that 
\begin{align*}
	\mathbb{U}\xhookrightarrow[\iota_s]{}\mathcal{V}^s \xhookrightarrow[j_s]{}\V \xhookrightarrow[j]{}\H\cong \H'\xhookrightarrow[j_s']{} (\mathcal{V}^s)' \xhookrightarrow[\iota_s']{} \mathbb{U}'.
\end{align*}Consider the composition 
\begin{align*}
	\iota:j\circ j_s\circ \iota_s:\mathbb{U}\to \H
\end{align*}and its adjoint 
\begin{align*}
	\iota^*:=(j\circ j_s\circ \iota_s)^*=\iota_s^*\circ j_s^*\circ j^*:\H\to \mathbb{U}.
\end{align*}We have that $\iota$ is compact and since its range is dense in $\H$, $\iota^*:\H\to\mathbb{U}$ is one-one. Let us define 
\begin{align}\label{2.2.2}\nonumber
	\D(\mathcal{L})&:=\iota^*(\H)\subset \mathbb{U},\\
	\mathcal{L}\x&:=(\iota^*)^{-1}\x, \  \x\in\D(\mathcal{L}). 
\end{align}Also, we have that $\mathcal{L}:\D(\mathcal{L})\to\H$ is onto, $\D(\mathcal{L})$ is dense in $\H$ and 
\begin{align*}
	(\mathcal{L}\x,\y)_\H=(\x,\y)_\mathbb{U}, \ \x\in\D(\mathcal{L}), \ \y\in\mathbb{U}.
\end{align*}Furthermore, for $\x\in\D(\mathcal{L})$, 
\begin{align*}
	\mathcal{L}\x=(\iota^*)^{-1}\x= (j^*)^{-1}\circ (j_s^*)^{-1}\circ (\iota_s^*)^{-1}\x=\mathtt{A}\circ (j_s^*)^{-1}\circ (\iota_s^*)^{-1}\x,
\end{align*}where $\mathtt{A}$ is defined in \eqref{2.2.1}. Since the operator $\mathcal{L}$ is self-adjoint and $\mathcal{L}^{-1}$ is compact, there exists an orthonormal basis $\{\mathbf{e}_k\}_{k\in\N}$ of $\H$ such that
\begin{align}\label{2.2.3}
	\mathcal{L}\mathbf{e}_k=\mu_k\mathbf{e}_k, \ k\in\N,
\end{align}that is, $\{\mathbf{e}_k\}_{k\in\N}$ are the eigenfunctions and $\{\mu_k\}_{k\in\N}$ are the corresponding eigenvalues of operator $\mathcal{L}$. Note that $\mathbf{e}_k\in\mathbb{U}, \ k\in\N$, because $\D(\mathcal{L})\subset \mathbb{U}$.
	
	Let us fix $n\in\N$ and let $\Pi_n$ be the operator from $\mathbb{U}'$ to $\mathrm{span}\{\mathbf{e}_1,\ldots,\mathbf{e}_n\}=:\H_n$ defined by 
	\begin{align}\label{2.2.4}
		\Pi_n\x^*:=\sum_{k=1}^n\langle \x^*,\mathbf{e}_k\rangle_{\mathbb{U}'\times \mathbb{U}}\mathbf{e}_k, \ \x^*\in\mathbb{U}'.
	\end{align}We will consider the restriction of the operator $\Pi_n$ to the space $\H$ denoted still by the same. In particular, we have $\H\hookrightarrow \mathbb{U}'$, that is, every element $\x\in\H$ induces a functional $\x^*\in\mathbb{U}'$ by 
\begin{align}\label{2.2.5}
	\langle \x^*,\y\rangle_{\mathbb{U}'\times \mathbb{U}}:=(\x,\y), \y\in\mathbb{U}.
	\end{align}Thus the restriction of $\Pi_n$ to $\H$ is given by 
\begin{align}\label{2.2.6}
	\Pi_n\x:=\sum_{k=1}^n(\x,\mathbf{e}_k)\mathbf{e}_k, \ \x\in\H.
\end{align}Hence, in particular, $\Pi_n$ is the orthogonal projection from $\H$ onto $\H_n$. 
\begin{lemma}[{\cite[Lemma 2.4]{ZBEM}}]
	For every $\x\in\mathbb{U}$ and $s>2$, we have 
	\begin{enumerate}
		\item $\lim\limits_{n\to\infty}\|\Pi_n\x-\x\|_\mathbb{U} =0$;
	\item 	$\lim\limits_{n\to\infty}\|\Pi_n\x-\x\|_{\mathcal{V}^s} =0$;
	\item $\lim\limits_{n\to\infty}\|\Pi_n\x-\x\|_\mathbb{V} =0$.
	\end{enumerate}
\end{lemma}

	\subsection{The operator $\mathcal{A}$}\label{OptA} Let $\mathcal{P}_p : \L^p(\mathcal{O}) \to\wi\L^p$ denote the \emph{Helmholtz-Hodge  projection} (\cite{DFHM}). For $p=2$, $\mathcal{P}:=\mathcal{P}_2$ becomes an orthogonal projection and for $2<p<\infty$, it is a bounded linear operator. Moreover, $\mathcal{P}$ maps $\H^{m-1}(\mathcal{O})$ into itself and is bounded if $\mathcal{O}$ is of class $\C^{m}$ (\cite[Remark 1.6]{Te}). Let us define the operator $\mathcal{A}:\V_p\to\V_p'$ by 
	\begin{align}
		\mathcal{A}(\x)=-\mathcal{P}\mathrm{div}\left(\left(1+|\nabla\x|^2\right)^{\frac{p-2}{2}}\nabla\x\right),
	\end{align}
	for all $\x\in\V_p$. For $\y\in\V_p$, it can be easily seen that 
	\begin{align*}
		|\langle\mathcal{A}(\x),\y\rangle|&=|\langle\left(1+|\nabla\x|^2\right)^{\frac{p-2}{2}}\nabla\x,\nabla\y\rangle|\leq 2^{\frac{p-2}{2}}\left(\|\nabla\x\|_{\H}\|\nabla\y\|_{\H}+\|\nabla\x\|_{\wi\L^p}^{p-1}\|\nabla\y\|_{\wi\L^p}\right)\nonumber\\&\leq C\left(1+\|\x\|_{\V_p}^{p-2}\right)\|\x\|_{\V_p}\|\y\|_{\V_p},
	\end{align*}
	for all $\y\in\V_p$, so that $\|\mathcal{A}(\x)\|_{\V_p'}\leq C\left(1+\|\x\|_{\V_p}^{p-2}\right)\|\x\|_{\V_p}$. Since 
	\begin{align*}
		2\langle\left(1+|\nabla\x|^2\right)^{\frac{p-2}{2}}\nabla\x,\nabla\x
		\rangle\geq \|\nabla\x\|_{\H}^2+\|\nabla\x\|_{\wi\L^p}^p,
	\end{align*}
	we have 
	\begin{align}\label{2.2}
		\langle\mathcal{A}(\x),\x\rangle=\langle\left(1+|\nabla\x|^2\right)^{\frac{p-2}{2}}\nabla\x,\nabla\x\rangle\geq\frac{1}{2}\left[\|\nabla\x\|_{\H}^2+\|\nabla\x\|_{\wi\L^p}^p\right]. 
	\end{align}
	Using Taylor's formula and H\"older's inequalities, we find 
	\begin{align}\label{2.3}
		&	|\langle\mathcal{A}(\x)-\mathcal{A}(\y),\z\rangle|\nonumber\\&\leq\left|\left< \left(1+|\nabla\x|^2\right)^{\frac{p-2}{2}}(\nabla\x-\nabla\y),\nabla\z\right>\right|\nonumber\\&\quad+\left|\left< \left(\left(1+|\nabla\x|^2\right)^{\frac{p-2}{2}}-\left(1+|\nabla\y|^2\right)^{\frac{p-2}{2}}\right)\nabla\y,\nabla\z\right>\right|\nonumber\\&\leq \bigg\{2^{\frac{p-2}{2}}\left(1+\|\nabla\x\|_{\wi\L^p}\right)\nonumber\\&\qquad+(p-2)2^{\frac{p-4}{2}}\left[\left(1+\|\nabla\x\|_{\wi\L^p}^{p-4}+\|\nabla\y\|_{\wi\L^p}^{p-4}\right)\left(\|\nabla\x\|_{\wi\L^p}+\|\nabla\y\|_{\wi\L^p}\right)\|\nabla\y\|_{\wi\L^p}\right]\bigg\}\nonumber\\&\quad\times\|\nabla(\x-\y)\|_{\wi\L^p}\|\nabla\z\|_{\wi\L^p},
	\end{align}
	for all $\z\in\V_p$. Thus the operator $\mathcal{A}(\cdot):\V_p\to\V_p'$ is a locally Lipschitz operator. 
	
	\subsection{The bilinear operator $\B(\cdot)$}\label{OptB}
	Let us define the \emph{trilinear form} $b(\cdot,\cdot,\cdot):\V\times\V\times\V\to\R$ by $$b(\x,\y,\z)=\int_{\mathcal{O}}(\x(x)\cdot\nabla)\y(x)\cdot\z(x)\d x=\sum_{i,j=1}^n\int_{\mathcal{O}}\x_i(x)\frac{\partial \y_j(x)}{\partial x_i}\z_j(x)\d x.$$ If $\x, \y$ are such that the linear map $b(\x, \y, \cdot) $ is continuous on $\V$, the corresponding element of $\V'$ is denoted by $\B(\x, \y)$. We also represent  $\B(\x) = \B(\x, \x)=\mathcal{P}[(\x\cdot\nabla)\x]$.	An integration by parts gives 
	\begin{equation*}
		\left\{
		\begin{aligned}
			b(\x,\y,\y) &= 0,\text{ for all }\x,\y \in\V,\\
			b(\x,\y,\z) &=  -b(\x,\z,\y),\text{ for all }\x,\y,\z\in \V.
		\end{aligned}
		\right.\end{equation*}
	Making use of H\"older's and Sobolev's inequalities,
	we find
	\begin{align}
		\left|\langle \B(\x,\x),\y\rangle \right|=\left|b(\x,\y,\x)\right|\leq\|\y\|_{\V}\|\x\|_{\wi\L^4}^2\leq C\|\x\|_{\V}^2\|\y\|_{\V},
	\end{align}
	for all $\x,\y\in\V$ and since $2\leq d\leq 4$, we get $\|\B(\x)\|_{\V'}\leq C\|\x\|_{\V}^2$. One can also show that 
	\begin{align*}
		\|\B(\x)-\B(\y)\|_{\V'}\leq C(\|\x\|_{\V}+\|\y\|_{\V})\|\x-\y\|_{\V},
	\end{align*}
	and hence the operator $\B(\cdot):\V\to\V'$ is locally Lipschitz. 	Using H\"older's and Sobolev's inequalities, we obtain 
	\begin{align*}
		|\langle\B(\x)-\B(\y),\z\rangle|&\leq|\langle\B(\x-\y,\x),\z\rangle|+\langle\B(\y,\x-\y),\z\rangle|\nonumber\\&\leq\left(\|\x-\y\|_{\wi\L^{\frac{pd}{d-p}}}\|\nabla\x\|_{\wi\L^p}+\|\y\|_{\wi\L^{\frac{pd}{d-p}}}\|\nabla(\x-\y)\|_{\wi\L^p}\right)\|\z\|_{\wi\L^{\frac{pd}{pd+p-2d}}}\nonumber\\&\leq C\left(\|\x\|_{\V_p}+\|\y\|_{\V_p}\right)\|\x-\y\|_{\V_p}\|\z\|_{\V_p},
	\end{align*}
	for all $\x,\y,\z\in\V_p$ and $p\geq \frac{3d}{d+2}$. Thus the operator $\B(\cdot):\V_p\to\V_p'$ is a locally Lipschitz operator. Using H\"older's and  interpolation inequalities, we get 
	{	\begin{align}\label{212}
			\left|\langle \B(\x,\x),\y\rangle \right|=\left|b(\x,\y,\x)\right|\leq \|\x\|_{\widetilde{\L}^{r}}\|\x\|_{\widetilde{\L}^{\frac{2r}{r-2}}}\|\y\|_{\V}\leq\|\x\|_{\widetilde{\L}^{r}}^{\frac{r}{r-2}}\|\x\|_{\H}^{\frac{r-4}{r-2}}\|\y\|_{\V},
		\end{align}
		for all $\y\in\V\cap\widetilde{\L}^{r}$ and $r>4$.} Thus, for $r>4$, we have 
	\begin{align}\label{2.9a}
		\|\B(\x)\|_{\V'+\widetilde{\L}^{\frac{r}{r-1}}}\leq\|\x\|_{\widetilde{\L}^{r}}^{\frac{r}{r-2}}\|\x\|_{\H}^{\frac{r-4}{r-2}}.
	\end{align}
	
	For $\x,\y\in\V\cap\widetilde{\L}^{r}$, we also find
	\begin{align}\label{lip}
		\|\B(\x)-\B(\y)\|_{\V'+\widetilde{\L}^{\frac{r}{r-1}}}\leq \left(\|\x\|_{\H}^{\frac{r-4}{r-2}}\|\x\|_{\widetilde{\L}^{r}}^{\frac{2}{r-2}}+\|\y\|_{\H}^{\frac{r-4}{r-2}}\|\y\|_{\widetilde{\L}^{r}}^{\frac{2}{r-2}}\right)\|\x-\y\|_{\widetilde{\L}^{r}},
	\end{align}
	for $r>4$, so that  the operator $\B(\cdot):\V\cap\widetilde{\L}^{r}\to\V'+\widetilde{\L}^{\frac{r}{r-1}}$ is a locally Lipschitz operator. For $r=4$,  a calculation similar to \eqref{lip} yields 
	\begin{align*}
		\|\B(\x)-\B(\y)\|_{\V'+\widetilde{\L}^{\frac{4}{3}}}&\leq \left(\|\x\|_{\widetilde{\L}^{4}}+\|\y\|_{\widetilde{\L}^{4}}\right)\|\x-\y\|_{\widetilde{\L}^{4}},
	\end{align*}
	hence $\B(\cdot):\V\cap\widetilde{\L}^{4}\to\V'+\widetilde{\L}^{\frac{4}{3}}$ is a locally Lipschitz operator. 
	
	\subsection{The nonlinear operator $\mathcal{C}(\cdot)$}\label{OptC}
	Let us define  the operator $\mathcal{C}:\V\cap\wi\L^{r}\to\V'+\wi\L^{\frac{r}{r-1}}$ by  $\mathcal{C}(\x):=\mathcal{P}(|\x|^{r-2}\x).$  It is immediate that $\langle\mathcal{C}(\x),\x\rangle =\|\x\|_{\widetilde{\L}^{r}}^{r},$ for all $\x\in\V\cap\wi\L^{r}$. It has been shown in \cite{MTM8} that 
	\begin{align}\label{213}
		|\langle \mathcal{C}(\x)-\mathcal{C}(\y),\z\rangle|&\leq (r-1)\left(\|\x\|_{\widetilde{\L}^{r}}+\|\y\|_{\widetilde{\L}^{r}}\right)^{r-2}\|\x-\y\|_{\widetilde{\L}^{r}}\|\z\|_{\widetilde{\L}^{r}},
	\end{align}
	for all $\x,\y,\z\in\V\cap\widetilde{\L}^{r}$. 
	Thus the operator $\mathcal{C}(\cdot):\widetilde{\L}^{r}\to\widetilde{\L}^{\frac{r}{r-1}}$ is locally Lipschitz. For any $r\in[1,\infty)$, we have (see \cite{MTM8})
	\begin{align}\label{2.23}
		&\langle\mathcal{C}(\x)-\mathcal{C}(\y),\x-\y\rangle\geq \frac{1}{2}\||\x|^{\frac{r-2}{2}}(\x-\y)\|_{\H}^2+\frac{1}{2}\||\y|^{\frac{r-2}{2}}(\x-\y)\|_{\H}^2\geq 0,
	\end{align}
	for $r\geq 1$. Furthermore 
	\begin{align}\label{a215}
		\|\x-\y\|_{\wi\L^{r}}^{r}&\leq 2^{r-3}\left[\||\x|^{\frac{r-2}{2}}(\x-\y)\|_{\H}^2+\||\y|^{\frac{r-2}{2}}(\x-\y)\|_{\H}^2\right],
	\end{align}
	for $r\geq 1$ (replace $2^{r-3}$ with $1,$ for $2\leq r\leq 3$), so that 
	\begin{align}
		&\langle\mathcal{C}(\x)-\mathcal{C}(\y),\x-\y\rangle\geq\frac{1}{2^{r-2}}\|\x-\y\|_{\wi\L^{r}}^{r},
	\end{align}
	for all $\x,\y\in\wi\L^{r}$. 
	
	\subsection{Monotonicity}
	Let us now discuss  the monotonicity as well as the hemicontinuity properties of the linear and nonlinear operators, which plays a crucial role in the global solvability of the system \eqref{31}. 
	\begin{definition}[\cite{VB}]
		Let $\X$ be a Banach space and let $\X'$ be its topological dual.
		An operator $\mathscr{H}:\mathrm{D}(\mathscr{H})\rightarrow
		\X',$ $\mathrm{D}(\mathscr{H})\subset \X$ is said to be
		\emph{monotone} if
		$$\langle\mathscr{H}(x)-\mathscr{H}(y),x-y\rangle\geq
		0,\ \text{ for all } \ x,y\in \mathrm{D}(\mathscr{H}).$$ 
		The operator $\mathscr{H}(\cdot)$ is said to be \emph{hemicontinuous}, if for all $x, y\in\X$ and $z\in\X',$ $$\lim_{\lambda\to 0}\langle\mathscr{H}(x+\lambda y),z\rangle=\langle\mathscr{H}(x),z\rangle.$$
		The operator $\mathscr{H}(\cdot)$ is called \emph{demicontinuous}, if for all $x\in\mathrm{D}(\mathscr{H})$ and $y\in\X$, the functional $x \mapsto\langle \mathscr{H}(x), y\rangle$  is continuous, or in other words, $x_k\to x$ in $\X$ implies $\mathscr{H}(x_k)\xrightarrow{w}\mathscr{H}(x)$ in $\X'$. Clearly demicontinuity implies hemicontinuity. 
	\end{definition}

	\begin{lemma}\label{lem2.2}
		Let $\x,\y\in\V_p\cap\wi\L^{r}$, for $p\geq 2$ and $r>4$. Then,	for the operator $$\mathscr{H}(\x)=\mu\mathcal{A}(\x)+\B(\x)+\beta\mathcal{C}(\x),$$ we  have 
		\begin{align}\label{mon}
			\langle\mathscr{H}(\x)-\mathscr{H}(\y),\x-\y\rangle+\eta\|\x-\y\|_{\H}^2\geq 0,
		\end{align}
		where \begin{align}\label{eta}
			\eta=\frac{p-4}{2\mu(p-2)}\left(\frac{2}{\beta\mu (p-2)}\right)^{\frac{2}{p-4}}.
		\end{align} 
		That is, the operator $\mathscr{H}+\eta\mathrm{I}$ is a monotone operator from $\V_p$ to $\V_p'$. 
		
		For $r=4$ with $2\beta\mu \geq 1$, the operator $\mathscr{H}(\cdot):\V_p\cap\widetilde{\L}^{4}\to \V_p'+\widetilde{\L}^{\frac{4}{3}}$ is globally monotone, that is, for all $\x,\y\in\V_p\cap\wi\L^4$, we have 
		\begin{align}\label{2.18}\langle\mathscr{H}(\x)-\mathscr{H}(\y),\x-\y\rangle\geq 0.\end{align}
	\end{lemma}
	\begin{proof}
		Let us estimate $-\mu 	\left<\mathcal{A}(\x)-\mathcal{A}(\y),\x-\y\right>$ using an integration by parts as
		\begin{align}\label{ae}
			&\mu 	\left<\mathcal{A}(\x)-\mathcal{A}(\y),\x-\y\right>\nonumber\\&=
			-	\mu 	\left<\mathrm{div}\left(\left(1+|\nabla\x|^2\right)^{\frac{p-2}{2}}\nabla\x\right)- \mathrm{div}\left(\left(1+|\nabla\y|^2\right)^{\frac{p-2}{2}}\nabla\y\right),\x-\y\right>
			\nonumber\\&= \mu\left(\left(1+|\nabla\x|^2\right)^{\frac{p-2}{2}}\nabla\x-\left(1+|\nabla\y|^2\right)^{\frac{p-2}{2}}\nabla\y,\nabla(\x-\y)\right)
			\nonumber\\&= \mu\left<\int_0^1\frac{\d}{\d\theta}\left[\left(1+|\nabla(\theta\x+(1-\theta\y))|^2\right)^{\frac{p-2}{2}}\nabla(\theta\x+(1-\theta)\y)\right]\d\theta ,\nabla(\x-\y)\right>\nonumber\\&= \mu\left<\int_0^1\left(1+|\nabla(\theta\x+(1-\theta\y))|^2\right)^{\frac{p-2}{2}}\d\theta\nabla(\x-\y),\nabla(\x-\y)\right>\nonumber\\&\quad+\mu(p-2)\bigg<\int_0^1\left(1+|\nabla(\theta\x+(1-\theta\y))|^2\right)^{\frac{p-4}{2}}|\nabla(\theta\x+(1-\theta\y))|^2\d\theta\nonumber\\&\qquad\times\nabla(\x-\y),\nabla(\x-\y)\bigg>\nonumber\\&\geq \mu\left<\int_0^1\left(1+|\nabla(\theta\x+(1-\theta\y))|^2\right)^{\frac{p-2}{2}}\d\theta\nabla(\x-\y),\nabla(\x-\y)\right>\nonumber\\&\geq \mu\|\nabla(\x-\y)\|_{\H}^2. 
		\end{align}
		From \eqref{2.23}, we easily have 
		\begin{align}\label{2.27}
			\beta	\langle\mathcal{C}(\x)-\mathcal{C}(\y),\x-\y\rangle \geq \frac{\beta}{2}\||\y|^{\frac{p-2}{2}}(\x-\y)\|_{\H}^2. 
		\end{align}
		Since $\langle\B(\x,\x-\y),\x-\y\rangle=0$, we get $	\langle \B(\x)-\B(\y),\x-\y\rangle =-\langle\B(\x-\y,\x-\y),\y\rangle$, and a calculation similar to the proof of \cite[Theorem 2.2]{MTM8}  yields 
		\begin{align}\label{2.30}
			&|\langle\B(\x-\y,\x-\y),\y\rangle|\nonumber\\&\leq\frac{\mu }{2}\|\x-\y\|_{\V}^2+\frac{\beta}{2}\||\y|^{\frac{r-2}{2}}(\x-\y)\|_{\H}^2+\frac{r-4}{2\mu(r-2)}\left(\frac{2}{\beta\mu (r-2)}\right)^{\frac{2}{r-4}}\|\x-\y\|_{\H}^2.
		\end{align}
		Combining \eqref{ae}, \eqref{2.27} and \eqref{2.30}, we get 
		\begin{align}
			\langle\mathscr{H}(\x)-\mathscr{H}(\y),\x-\y\rangle+\frac{r-4}{2\mu(r-2)}\left(\frac{2}{\beta\mu (r-2)}\right)^{\frac{2}{r-4}}\|\x-\y\|_{\H}^2\geq\frac{\mu }{2}\|\x-\y\|_{\V}^2\geq 0,
		\end{align}
		for $r>4$ and the estimate \eqref{mon} follows.

		From \eqref{2.23}, we infer that 
		\begin{align}\label{231}
			\beta\langle\mathcal{C}(\x)-\mathcal{C}(\y),\x-\y\rangle\geq\frac{\beta}{2}\|\y(\x-\y)\|_{\H}^2. 
		\end{align}
		We estimate $|\langle\B(\x-\y,\x-\y),\y\rangle|$ using H\"older's and Young's inequalities as 
		\begin{align}\label{232}
			|\langle\B(\x-\y,\x-\y),\y\rangle|\leq\|\y(\x-\y)\|_{\H}\|\x-\y\|_{\V} \leq\mu \|\x-\y\|_{\V}^2+\frac{1}{4\mu }\|\y(\x-\y)\|_{\H}^2.
		\end{align}
		Combining \eqref{ae}, \eqref{231} and \eqref{232}, we obtain 
		\begin{align}
			\langle\mathscr{H}(\x)-\mathscr{H}(\y),\x-\y\rangle\geq\frac{1}{2}\left(\beta-\frac{1}{2\mu }\right)\|\y(\x-\y)\|_{\H}^2\geq 0,
		\end{align}
		provided $2\beta\mu \geq 1$. 
	\end{proof}
	
	\begin{remark}
		One can estimate the term $\mu 	\left<\mathcal{A}(\x)-\mathcal{A}(\y),\x-\y\right>$ in the following way also: 
		\begin{align}\label{ae1}
			&\mu 	\left<\mathcal{A}(\x)-\mathcal{A}(\y),\x-\y\right>\nonumber\\&=
			-	\mu 	\left<\mathrm{div}\left(\left(1+|\nabla\x|^2\right)^{\frac{p-2}{2}}\nabla\x\right)- \mathrm{div}\left(\left(1+|\nabla\y|^2\right)^{\frac{p-2}{2}}\nabla\y\right),\x-\y\right>
			\nonumber\\&= \mu\left(\left(1+|\nabla\x|^2\right)^{\frac{p-2}{2}}\nabla\x-\left(1+|\nabla\y|^2\right)^{\frac{p-2}{2}}\nabla\y,\nabla(\x-\y)\right)\nonumber\\&=\mu\left((1+|\nabla\x|^2)^{\frac{p-2}{2}},|\nabla(\x-\y)|^2\right)+\mu\left((1+|\nabla\y|^2)^{\frac{p-2}{2}},|\nabla(\x-\y)|^2\right)\nonumber\\&\quad+\mu\left((1+|\nabla\x|^2)^{\frac{p-2}{2}}\nabla\y,\nabla(\x-\y)\right)-\mu\left((1+|\nabla\y|^2)^{\frac{p-2}{2}}\nabla\x,\nabla(\x-\y)\right)\nonumber\\&=\mu\|(1+|\nabla\x|^2)^{\frac{p-2}{4}}\nabla(\x-\y)\|_{\H}^2+\mu\|(1+|\nabla\y|^2)^{\frac{p-2}{4}}\nabla(\x-\y)\|_{\H}^2\nonumber\\&\quad-\mu\left(\left(1+|\nabla\x|^2\right)^{\frac{p-2}{2}},|\nabla\y|^2\right)-\mu\left(\left(1+|\nabla\y|^2\right)^{\frac{p-2}{2}},|\nabla\x|^2\right)\nonumber\\&\quad+\mu\left(\left(1+|\nabla\x|^2\right)^{\frac{p-2}{2}}+\left(1+|\nabla\y|^2\right)^{\frac{p-2}{2}},(\nabla\x\cdot\nabla\y)\right)\nonumber\\&=\frac{\mu}{2}\|(1+|\nabla\x|^2)^{\frac{p-2}{4}}\nabla(\x-\y)\|_{\H}^2+\frac{\mu}{2}\|(1+|\nabla\y|^2)^{\frac{p-2}{4}}\nabla(\x-\y)\|_{\H}^2\nonumber\\&\quad+\frac{\mu}{2}\left(\left(1+|\nabla\x|^2\right)^{\frac{p-2}{2}}-\left(1+|\nabla\y|^2\right)^{\frac{p-2}{2}},(|\nabla\x|^2-|\nabla\y|^2)\right)\nonumber\\&\geq\frac{\mu}{2}\|(1+|\nabla\x|^2)^{\frac{p-2}{4}}\nabla(\x-\y)\|_{\H}^2+\frac{\mu}{2}\|(1+|\nabla\y|^2)^{\frac{p-2}{4}}\nabla(\x-\y)\|_{\H}^2,
		\end{align}
		where we have used the fact that  $$\left(\left(1+|\nabla\x|^2\right)^{\frac{p-2}{2}}-\left(1+|\nabla\y|^2\right)^{\frac{p-2}{2}},((1+|\nabla\x|^2)-(1+|\nabla\y|^2))\right)\geq 0.$$	It can be easily seen that 
		\begin{align*}
			&2	\int_{\mathcal{O}}(1+|\nabla\x(x)|^2)^{\frac{p-2}{2}}|\nabla(\x(x)-\y(x))|^2\d x\nonumber\\&\geq\int_{\mathcal{O}}|\nabla(\x(x)-\y(x))|^2\d x+\int_{\mathcal{O}}|\nabla\x(x)|^{p-2}|\nabla(\x(x)-\y(x))|^2\d x.
		\end{align*}
		Thus, from \eqref{ae1}, we deduce that 
		\begin{align}\label{2.31}
			&\mu 	\left<\mathcal{A}(\x)-\mathcal{A}(\y),\x-\y\right>\nonumber\\&\geq\frac{\mu}{2}\|\nabla(\x-\y)\|_{\H}^2+\frac{\mu}{4}\||\nabla\x|^{\frac{p-2}{2}}\nabla(\x-\y)\|_{\H}^2+\frac{\mu}{4}\||\nabla\y|^{\frac{p-2}{2}}\nabla(\x-\y)\|_{\H}^2.
		\end{align}
		Since $(a+b)^q\leq 2^{q-1}(a^q+b^q)$, for all $a,b\geq 0$ and $1\leq q<\infty$, and $(a+b)^q\leq (a^q+b^q)$,  for all $a,b\geq 0$ and $0\leq q\leq 1$,  we find 
		\begin{align}\label{2.32}
			\|\x-\y\|_{\V_p}^p&=\int_{\mathcal{O}}|\nabla(\x(x)-\y(x))|^{p-2}|\nabla(\x(x)-\y(x))|^{2}\d x\nonumber\\&\leq 2^{p-3}\int_{\mathcal{O}}|\nabla\x(x)|^{p-2}|\nabla\x(x)-\nabla\y(x)|^2\d x\nonumber\\&\quad+2^{p-3}\int_{\mathcal{O}}|\nabla\y(x)|^{p-2}|\nabla\x(x)-\nabla\y(x)|^2\d x\nonumber\\&= 2^{p-3}[\||\nabla\x|^{\frac{p-2}{2}}\nabla(\x-\y)\|_{\H}^2+\||\nabla\y|^{\frac{p-2}{2}}\nabla(\x-\y)\|_{\H}^2],
		\end{align}
		for $3\leq p<\infty$, where one has to replace $2^{p-3}$ by $1$ for the case $2<p<3$. Thus, from \eqref{2.31}, we finally have 
		\begin{align}\label{233} 
			&\mu 	\left<\mathcal{A}(\x)-\mathcal{A}(\y),\x-\y\right>\geq\frac{\mu}{2}\|\x-\y\|_{\V}^2+\frac{\mu}{2^{p-1}}\|\x-\y\|_{\V_p}^p,
		\end{align}
		for all $\x,\y\in\V_p$, $3\leq p<\infty$. For $2< p< 3$, we have to replace $ \frac{\mu}{2^{p-1}}$ by $\frac{\mu}{4}$. 
	\end{remark}

	\begin{lemma}\label{lem2.5}
		Let {$p>\frac{d}{2}, \ r\geq 2$} and $\x,\y\in\V_p\cap\wi\L^{r}$. Then,	for the operator $\mathscr{H}(\x)=\mu\mathcal{A}(\x)+\B(\x)+\beta\mathcal{C}(\x),$ we  have 
		\begin{align}\label{mon1}
			\langle\mathscr{H}(\x)-\mathscr{H}(\y),\x-\y\rangle+\wi\eta N^{\frac{2p}{2p-d}}\|\x-\y\|_{\H}^2\geq 0,
		\end{align}
		for all $\y\in 
		\widehat{\mathbb{B}}_N:=\big\{\z\in\V_p:\|\z\|_{\V_p}\leq N\big\},
		$	where \begin{align}\label{eta1}
			\wi\eta=C^{\frac{2p}{2p-d}}\left(\frac{2p-d}{2p}\right)\left(\frac{d}{\mu p}\right)^{\frac{d}{2p-d}},
		\end{align} 
		and $C$ is the constant appearing in the Gagliardo-Nirenberg inequality. 
	\end{lemma}
	
	\begin{proof}
		We estimate $|\langle\B(\x-\y,\x-\y),\y\rangle|$  using H\"older's, Gagliardo-Nirenberg's and Young's  inequalities as 
		\begin{align}\label{225}
			|\langle\B(\x-\y,\y),\x-\y\rangle|&\leq\|\nabla\y\|_{\wi\L^p}\|\x-\y\|_{\wi\L^{\frac{2p}{p-1}}}^2\nonumber\\&\leq C\|\y\|_{\V_p}\|\x-\y\|_{\V}^{\frac{d}{p}}\|\x-\y\|_{\H}^{\frac{2p-d}{p}}\nonumber\\&\leq \frac{\mu}{2}\|\x-\y\|_{\V}^2+C^{\frac{2p}{2p-d}}\left(\frac{2p-d}{2p}\right)\left(\frac{d}{\mu p}\right)^{\frac{d}{2p-d}}\|\y\|_{\V_p}^{\frac{2p}{2p-d}}\|\x-\y\|_{\H}^2. 
		\end{align}
		Combining \eqref{ae}, \eqref{2.23} and \eqref{225}, we deduce 
		\begin{align*}
			\langle\mathscr{H}(\x)-\mathscr{H}(\y),\x-\y\rangle+C^{\frac{2p}{2p-d}}\left(\frac{2p-d}{2p}\right)\left(\frac{d}{\mu p}\right)^{\frac{d}{2p-d}}\|\y\|_{\V_p}^{\frac{2p}{2p-d}}\|\x-\y\|_{\H}^2\geq\frac{\mu }{2}\|\x-\y\|_{\V}^2\geq 0.
		\end{align*}
		Let $\widehat{\mathbb{B}}_N$ be an $\V_p$-ball of radius $N$, that is,
		$
		\widehat{\mathbb{B}}_N:=\big\{\z\in\V_p:\|\z\|_{\V_p}\leq N\big\}.
		$ Thus	for all $\y\in\widehat{\mathbb{B}}_N$,  we have 
		\begin{align}
			\langle\mathscr{H}(\x)-\mathscr{H}(\y),\x-\y\rangle+C^{\frac{2p}{2p-d}}\left(\frac{2p-d}{2p}\right)\left(\frac{d}{\mu p}\right)^{\frac{d}{2p-d}}N^{\frac{2p}{2p-d}}\|\x-\y\|_{\H}^2\geq 0,
		\end{align}
		and hence  the operator $\mathscr{H}(\cdot)$ is locally monotone.
	\end{proof}

	\begin{remark}
		1. 	It should be noted that if $\y\in\mathrm{L}^p(0,T;\V_p)$, then 
		\begin{align}\label{2p33}
			\int_0^T\|\y(t)\|_{\V_p}^{\frac{2p}{2p-d}}\d t\leq T^{\frac{2p-(d+2)}{2p-d}} \left(\int_0^T\|\y(t)\|_{\V_p}^p\d t\right)^{\frac{2}{2p-d}}<+\infty,
		\end{align} 
		for $p\geq\frac{d}{2}+1$. 
		
		2. Using  Gagliardo-Nirenberg's and H\"older's inequalities, we find 
		\begin{align*}
			\int_0^T\|\y(t)\|_{\wi\L^r}^{r}\d t&\leq C\int_0^T\|\nabla\y(t)\|_{\wi\L^p}^{\frac{pd(r-2)}{2p+pd-2d}}\|\y(t)\|_{\H}^{\frac{2(pr+pd-dr)}{2p+pd-2d}}\d t\nonumber\\&\leq C T^{\frac{2p+pd-dr}{2p+pd-2d}}\sup_{t\in[0,T]}\|\y(t)\|_{\H}^{\frac{2(pr+pd-dr)}{2p+pd-2d}}\left(\int_0^T\|\y(t)\|_{\V_p}^p\d t\right)^{\frac{d(r-2)}{2p+pd-2d}}<+\infty,
		\end{align*}
		for $2\leq r\leq p\left(\frac{2}{d}+1\right)$, so that we have the following embedding:
		\begin{align*}
			\mathrm{L}^{\infty}(0,T;\H)\cap\mathrm{L}^p(0,T;\V_p)\subset\mathrm{L}^{r}(0,T;\wi\L^{r}). 
		\end{align*}
	\end{remark}

	\begin{remark}
		1. One can estimate $|\langle\B(\x-\y,\x-\y),\y\rangle|$  in the following way also:
		\begin{align*}
			|\langle\B(\x-\y,\x-\y),\y\rangle|&\leq\|\x-\y\|_{\V}\|\x-\y\|_{\H}\|\y\|_{\wi\L^{\infty}}\nonumber\\&\leq\frac{\mu}{2}\|\x-\y\|_{\V}^2+\frac{1}{2\mu}\|\y\|_{\wi\L^{\infty}}^2\|\x-\y\|_{\H}^2\nonumber\\&\leq \frac{\mu}{2}\|\x-\y\|_{\V}^2+\frac{1}{2\mu}\|\y\|_{\V_p}^2\|\x-\y\|_{\H}^2,
		\end{align*}
		for $p>d$. 
		
		2. Using H\"older's, Gagliardo-Nirenberg's and Young's inequalities, one can estimate $|\langle\B(\x-\y,\x-\y),\y\rangle|$  as
		\begin{align}
			|\langle\B(\x-\y,\x-\y),\y\rangle|&\leq\|\x-\y\|_{\V}\|\y\|_{\wi\L^{\frac{dp}{d-p}}}\|\x-\y\|_{\wi\L^{\frac{2pd}{pd-2d+2p}}}\nonumber\\&\leq\|\x-\y\|_{\V}^{\frac{d}{p}}\|\y\|_{\V_p}\|\x-\y\|_{\H}^{\frac{2p-d}{p}}\nonumber\\&\leq\frac{\mu}{2}\|\x-\y\|_{\V}^2+\frac{C}{\mu}\|\y\|_{\V_p}^{\frac{2p}{2p-d}}\|\x-\y\|_{\H}^2,
		\end{align}
		provided $\frac{d}{2}\leq p< d$. From \eqref{2p33}, we infer that $	\int_0^T\|\y(t)\|_{\V_p}^{\frac{2p}{2p-d}}\d t<\infty$, for $1+\frac{d}{2}\leq p< d$. For more estimates on the bilinear operator, the interested readers are referred to see \cite{JMJN}. 
	\end{remark}

	\begin{lemma}\label{lem2.8}
		The operator $\mathscr{H}:\V_p\cap\widetilde{\L}^{r}\to \V_p'+\widetilde{\L}^{\frac{r}{r-1}}$ is demicontinuous. 
	\end{lemma}
	\begin{proof}
		Let us take a sequence $\x^n\to \x$ in $\V_p\cap\widetilde{\L}^{r}$, that is, $\|\x^n-\x\|_{\V_p}+\|\x^n-\x\|_{\wi\L^r}\to 0,$ as $n\to\infty$. For any $\y\in\V_p\cap\widetilde{\L}^{r}$, we consider 
		\begin{align}\label{214}\nonumber
			&	\langle\mathscr{H}(\x^n)-\mathscr{H}(\x),\y\rangle\\ &=\mu \langle \mathcal{A}(\x^n)-\mathcal{A}(\x),\y\rangle+\langle\B(\x^n)-\B(\x),\y\rangle+\beta\langle \mathcal{C}(\x^n)-\mathcal{C}(\x),\y\rangle.
		\end{align} 
		Next, we consider $\langle \mathcal{A}(\x^n)-\mathcal{A}(\x),\y\rangle$ from \eqref{214} and estimate it using \eqref{2.3} as 
		\begin{align}
			&|\langle\mathcal{A}(\x^n)-\mathcal{A}(\x),\y\rangle|\nonumber\\&=|\langle(1+|\nabla\x^n|^2)^{\frac{p-2}{2}}\nabla\x^n-\langle(1+|\nabla\x|^2)^{\frac{p-2}{2}}\nabla\x,\nabla\y\rangle|\nonumber\\&\leq \bigg\{2^{\frac{p-2}{2}}\left(1+\|\nabla\x^n\|_{\wi\L^p}\right)\nonumber\\&\qquad+(p-2)2^{\frac{p-4}{2}}\left[\left(1+\|\nabla\x^n\|_{\wi\L^p}^{p-4}+\|\nabla\x\|_{\wi\L^p}^{p-4}\right)\left(\|\nabla\x^n\|_{\wi\L^p}+\|\nabla\x\|_{\wi\L^p}\right)\|\nabla\x\|_{\wi\L^p}\right]\bigg\}\nonumber\\&\quad\times\|\nabla(\x^n-\x)\|_{\wi\L^p}\|\nabla\y\|_{\wi\L^p}\nonumber\\&\to 0, \ \text{ as } \ n\to\infty, 
		\end{align}
		since $\x^n\to \x$ in $\V_p$. We estimate the term $\langle\B(\x^n)-\B(\x),\y\rangle$ from \eqref{214} using H\"older's and Sobolev's inequalities as 
		\begin{align}
			|\langle\B(\x^n)-\B(\x),\y\rangle|&
			\leq|\langle\B(\x^n,\y),\x^n-\x\rangle|+|\langle\B(\x^n-\x,\y),\x\rangle|\nonumber\\&\leq C\left(\|\x^n\|_{\V}+\|\x\|_{\V}\right)\|\x^n-\x\|_{\V}\|\y\|_{\V}\nonumber\\& \to 0, \ \text{ as } \ n\to\infty, 
		\end{align}
		since $\x^n\to\x$ in $\V$ and $\x^n,\x\in\V$. Finally, we estimate the term $\langle \mathcal{C}(\x^n)-\mathcal{C}(\x),\y\rangle$ from \eqref{214} using Taylor's formula and H\"older's inequality as 
		\begin{align}
			|\langle \mathcal{C}(\x^n)-\mathcal{C}(\x),\y\rangle|&\leq (r-1)\|\x^n-\x\|_{\widetilde{\L}^{r}}\left(\|\x^n\|_{\widetilde{\L}^{r}}+\|\x\|_{\widetilde{\L}^{r}}\right)^{r-2}\|\y\|_{\widetilde{\L}^{r}}\nonumber\\&\to 0, \ \text{ as } \ n\to\infty, 
		\end{align}
		since $\x^n\to\x$ in $\widetilde{\L}^{r}$ and $\x^n,\x\in\V\cap\widetilde{\L}^{r}$. From the above convergences, it is immediate that $\langle\mathscr{H}(\x^n)-\mathscr{H}(\x),\y\rangle \to 0$, for all $\y\in \V_p\cap\widetilde{\L}^{r}$.
		Hence the operator $\mathscr{H}:\V_p\cap\widetilde{\L}^{r}\to \V_p'+\widetilde{\L}^{\frac{r}{r-1}}$ is demicontinuous, which implies that the operator $\mathscr{H}(\cdot)$ is  hemicontinuous. 
	\end{proof}

	\section{Stochastic Ladyzhenskaya-Smagorinsky  equations with damping}\label{sec3}\setcounter{equation}{0} In this section, we consider the stochastic  Ladyzhenskaya-Smagorinsky   equations with damping perturbed by multiplicative Gaussian noise and discuss global solvability results.  Let us first  state the assumptions on the noise coefficient, so that we obtain the existence and uniqueness of strong solution to the system \eqref{31}.

	\begin{hypothesis}\label{hyp}
		The noise coefficients $\boldsymbol{\sigma}_k(\cdot,\cdot)$,\ $k\geq 1$ satisfy the following: 
		\begin{itemize}
			\item[(H.1)]  (Growth condition).	There exists a positive	constant $K$ such that for all $t\in[0,T]$ and $\x\in\H$,
			\begin{equation*}
				\sum_{k=1}^{\infty}\|\boldsymbol{\sigma}_k(t, \x)\|^{2}_{\H} 	\leq K\left(1 +\|\x\|_{\H}^{2}\right);
			\end{equation*}
			
			\item[(H.2)]  (Lipschitz condition).	There exists a positive constant $L$ such that for any $t\in[0,T]$ and all $\x_1,\x_2\in\H$,
			\begin{align}\label{H2.2}
				\sum_{k=1}^{\infty}	\|\boldsymbol{\sigma}_k(t,\x_1) - \boldsymbol{\sigma}_k(t,\x_2)\|^2_{\H}\leq L\|\x_1 -	\x_2\|_{\H}^2.
			\end{align}
		\end{itemize}
	\end{hypothesis}	
	Note that the condition (H.1) implies that for every $\x\in\H$, the linear map $\boldsymbol{\sigma}(\cdot,\x):=\{\boldsymbol{\sigma}_k(\cdot,\x)\}_{k\in\mathbb{N}}:\ell^2\to\H$ defined by \begin{align}\label{3p2}
		\boldsymbol{\sigma}(\cdot,\x)h:=\sum_{k=1}^{\infty}\boldsymbol{\sigma}_k(\cdot,\x)h_k,\ h=\{h_k\}_{k\in\mathbb{N}}
	\end{align}
	is in $\mathcal{L}_2(\ell^2;\H)$, that is, $\boldsymbol{\sigma}(\cdot,\x)$ is a Hilbert-Schmidt operator from $\ell^2$ to $\H$. For the orthonormal basis $\{e_k\}_{k\in\N}\in\ell^2$ with $e_k=(0,\ldots,1,\ldots)$, we have 
	\begin{align}\label{3p3}
		\|\boldsymbol{\sigma}(t,\x)\|_{\mathcal{L}_2(\ell^2;\H)}^2=\sum_{k=1}^{\infty}\|\boldsymbol{\sigma}_k(t,\x)e_k\|_{\H}^2=\sum_{k=1}^{\infty}\|\boldsymbol{\sigma}_k(t,\x)\|_{\H}^2\leq K(1+\|\x\|_{\H}^2)<\infty,
	\end{align}
	for all $t\in[0,T]$.

	Taking the Helmholtz-Hodge orthogonal   projection $\mathcal{P}$ onto the system \eqref{31}, we obtain 
	\begin{equation}\label{32}
		\left\{
		\begin{aligned}
			\d\u(t)+[\mu \mathcal{A}(\u(t))+\B(\u(t))+\beta\mathcal{C}(\u(t))]\d t&=\f(t)\d t+\sum_{k=1}^{\infty}\boldsymbol{\sigma}_k(t,\u(t))\d\boldsymbol{W}_k(t),\\
			\u(0)&=\u_0\in\H,
		\end{aligned}
		\right.
	\end{equation}
	for $t\in(0,T)$. For simplicity, we have  used the notation $\f$ for $\mathcal{P}\f$ and $\boldsymbol{\sigma}_k$ for $\mathcal{P}\boldsymbol{\sigma}_k$.

	Let us now provide the definition of a unique global strong solution in the probabilistic sense to the system (\ref{32}).
	\begin{definition}[Global strong solution]
		Let $\u_0\in\H$, $\f\in\mathrm{L}^{p'}(0,T;\V_p')$ and {$\xi>0$} be given. An $\H$-valued $\{\mathscr{F}_t\}_{t\geq 0}$-adapted stochastic process $\u(\cdot)$ is called a \emph{strong solution} to the system (                                         \ref{32}) if the following conditions are satisfied:  
		\begin{enumerate}
			\item [(i)] for $ q\in\Big\{2,\frac{2(2p-d)}{p-1}, {4+\xi}\Big\}$, the process $\u\in{\mathrm{L}^{q}(\Omega;\mathrm{L}^{\infty}(0,T;\H))\cap\mathrm{L}^p(\Omega;\mathrm{L}^p(0,T;\V_p))}\break\cap\mathrm{L}^{r}(\Omega;\mathrm{L}^{r}(0,T;\widetilde{\L}^{r}))$ and $\u(\cdot)$ has a $\V_p\cap\widetilde{\L}^{r}$-valued  modification, which is progressively measurable with continuous paths in $\H$ and $\u\in\C([0,T];\H)\cap{\mathrm{L}^p(0,T;\V_p)}\cap\mathrm{L}^{r}(0,T;\widetilde{\L}^{r})$, $\mathbb{P}$-a.s.;
			\item [(ii)] the following equality holds for every $t\in [0, T ]$, as an element of $\V_p'+\wi\L^{\frac{r}{r-1}},$ $\mathbb{P}$-a.s.,
			\begin{align}\label{4.4}\nonumber
				\u(t)&=\x-\int_0^t\left[\mu \mathcal{A}(\u(s))+\B(\u(s))+\beta\mathcal{C}(\u(s))-\f(s)\right]\d s\\&\quad +\sum_{k=1}^{\infty}\int_0^t\boldsymbol{\sigma}_k( s,\u(s))\d\boldsymbol{W}_k(s);
			\end{align}
			\item [(iii)] the following It\^o formula holds true, for all $t\in[0,T]$, $\mathbb{P}$-a.s.,
			\begin{align}\label{a34}
				&	\|\u(t)\|_{\H}^2+2\mu \int_0^t\|(1+|\nabla\u(s)|^2)^{\frac{p-2}{4}}\nabla\u(s)\|_{\H}^2\d s+2\beta\int_0^t\|\u(s)\|_{\widetilde{\L}^{r}}^{r}\d s\nonumber\\&=\|\x\|_{\H}^2+2\int_0^t\langle \f(s),\u(s)\rangle\d s+\sum_{k=1}^{\infty}\int_0^t\|\boldsymbol{\sigma}_k(s,\u(s))\|_{\H}^2\d s\nonumber\\&\quad+2\sum_{k=1}^{\infty}\int_0^t(\boldsymbol{\sigma}_k( s,\u(s)),\u(s))\d\boldsymbol{W}_k(s).
			\end{align}
			
		\end{enumerate}
	\end{definition}
	An alternative version of the condition (\ref{4.4}) is to require that for every $t\in[0,T]$, and for any  $\v\in\V_p\cap\widetilde{\L}^{r}$, $\PP$-a.s.,
	\begin{align}\label{4.5}
		(\u(t),\v)&=(\x,\v)-\int_0^t\langle\mu \mathcal{A}(\u(s))+\B(\u(s))+\beta\mathcal{C}(\u(s))-\f(s),\v\rangle\d s\no\\&\quad+\sum_{k=1}^{\infty}\int_0^t\left(\boldsymbol{\sigma}_k(s,\u(s))\d\boldsymbol{W}_k(s),\v\right).
	\end{align}	
	\begin{definition}
		A strong solution $\u(\cdot)$ to (\ref{32}) is called a
		\emph{ pathwise unique    strong solution} if
		$\widetilde{\u}(\cdot)$ is an another strong
		solution, then $$\mathbb{P}\big\{\omega\in\Omega:\u(t)=\widetilde{\u}(t),\ \text{ for all }\ t\in[0,T]\big\}=1.$$ 
	\end{definition}

	\begin{theorem}\label{exis2}
		Let $\u_0\in\H$, $\f\in\mathrm{L}^{p'}(0,T;\V_p')$ and {$\xi>0$} be given.  Then for {$p\geq\frac{d}{2}+1$, $r\geq 2$} or $p\geq 2 $, $r\geq 4$ ($2\beta\mu\geq 1$ for $r=4$), under Hypothesis \ref{hyp},  there exists a \emph{pathwise unique strong solution}
		$\u(\cdot)$ to the system (\ref{32}) such that 
		\begin{align*}
			\u&\in\mathrm{L}^{q}(\Omega;\mathrm{L}^{\infty}(0,T;\H))\cap{\mathrm{L}^{p}(\Omega;\mathrm{L}^p  (0,T;\V_p))}\cap\mathrm{L}^{r}(\Omega;\mathrm{L}^{r}(0,T;\widetilde{\L}^{r})),
		\end{align*} 
		for  $q\in\left\{2, \frac{2(2p-d)}{p-1},{4+\xi}\right\} $, with $\mathbb{P}$-a.s., continuous trajectories in $\H$.
	\end{theorem}
	\noindent
	Let us first provide the methodology for the proof. 
	\begin{itemize}
		\item For $p\geq\frac{d}{2}+1$ and $r\geq 2$, since the operator $\mathscr{H}(\u)=\mu\mathcal{A}(\u)+\B(\u)+\beta\mathcal{C}(\u)$ satisfies the local monotonicity condition \eqref{mon1} as well as demicontinuity condition (see Lemmas \ref{lem2.5} and \ref{lem2.8}, respectively), we establish the  existence of a strong solution by a localized version (stochastic generalization) of  the Minty-Browder technique in Section \ref{EUSS} (for more details one can see \cite{MTM6,SSSP}, etc).
		\item  In order to establish the energy equality (It\^o's formula) \eqref{a34} for $2\leq r\leq \frac{pd}{d-p}$ ($2\leq r<\infty$ for $d=2$), one can use the recent result in \cite[Theorem 2.1]{GK2}. But for $r> \frac{pd}{d-p},$ we need a different technique to establish It\^o's formula. The authors in \cite{CLF} were able to construct functions that can approximate functions defined on smooth bounded domains by elements of eigenspaces of  the Stokes operator in such a way that the approximations are bounded and converge in both Sobolev and Lebesgue spaces simultaneously.
		\item   Note that 
		the eigenfunctions $\{\mathbf{e}_j\}_{j=1}^{\infty}$ of the  operator $\mathcal{L}$ belongs to $\mathcal{V}^s$, for $s> 2$ (see Section \ref{ACO}). In our case, one can choose $s\geq 3$, so that $\{\mathbf{e}_j\}_{j=1}^{\infty}\subset\D(\mathcal{L})\subset \V_p\cap\wi\L^r$, for $p\geq 2$ and $r\geq 2$. For the case $r> \frac{pd}{d-p}$, combining this idea along with the techniques used in \cite[Theorem 3.7]{MTM8} one can establish It\^o's formula.
		\item 	For the case $p\geq 2$ and $r\geq 4$ ($2\beta\mu\geq 1$ for $r=4$), as the operator $\mathscr{H}(\cdot)+\eta\mathrm{I}$ is globally monotone and demicontinuous (see Lemmas \ref{lem2.2} and \ref{lem2.8}, respectively), one can use the similar techniques used in Section \ref{EUSS} to obtain the well-posedness results. 
	\end{itemize}

	\section{Proof of Theorem {\ref{exis2}}}\label{EUSS}\setcounter{equation}{0}
	In this section, we discuss the existence and uniqueness of a strong solution to the system \eqref{32} under the assumption that the operator $\mathscr{H}(\cdot)=\mu\mathcal{A}(\cdot)+\B(\cdot)+\beta\mathcal{C}(\cdot)$ \big(that is, for $p\geq \frac{d}{2}+1$ and $r\geq 2$\big) is locally monotone and demicontinuous (see Lemmas \ref{lem2.5} and \ref{lem2.8}, respectively). Let us start with the energy estimates.
	\subsection{Energy estimates}
	Let $\{\be_1,\be_2,\ldots,\be_n,\dots\}$ be the orthonormal basis in $\H$ of consisting of eigenfunctions of $\mathcal{L}$. Let $\H_n$ be the $n$-dimensional subspace of $\H$. Let $\Pi_n$ denote the orthogonal projection of $\mathbb{U}'$ onto $\H$ defined in \eqref{2.2.4} (see Subsection \ref{ACO}). 
	We define $\mathcal{A}_n(\u_n)=\Pi_n\mathcal{A}(\u_n)$, $\B_n(\u_n)=\Pi_n\B(\u_n)$, $\mathcal{C}_n(\u_n)=\Pi_n\mathcal{C}(\u_n)$, and $ \f_n=\Pi_n\f$. 
	Let us consider the following system of ODEs:
	\begin{equation}\label{EE1}
		\left\{
		\begin{aligned}
			\d(\u_n(t),\v) &=-\langle \mu\mathcal{A}_n(\u_n(t))+\B_n(\u_n(t))+\beta\mathcal{C}_n(\u_n(t))-\f_n(t),\v\rangle\d t\\&\quad +\sum_{k=1}^n\big(\boldsymbol{\sigma}_k(t,\u_n(t))\d\boldsymbol{W}_k(t),\v\big),\\
			\u(0)&=\u_0^n=\Pi_n\u_0,
		\end{aligned}
		\right.
	\end{equation}for all $\v\in\H_n$. Since $\mathcal{A}_n(\cdot)$, $\B_n(\cdot)$, and $\mathcal{C}_n(\cdot)$ are locally Lipschitz (see Subsections \ref{OptA}, \ref{OptB} and \ref{OptC}, respectively), and the operator $\boldsymbol{\sigma}_k(\cdot,\cdot)$, for $k=1,\ldots,n$ are globally Lipschitz (see \eqref{H2.2}), the system \eqref{EE1} has a unique $\H_n$-valued strong solution $\u_n(\cdot)$ and $\u_n\in\mathrm{L}^2(\Omega;\mathrm{L}^\infty(0,T^*;\H_n))$ with $\mathbb{P}$-a.s. continuous sample paths, that is, $\u_n\in\mathrm{C}([0,T^*];\H_n)$.  Now, we discuss the a-priori energy estimates satisfied by the system \eqref{EE1}.
	
	\begin{proposition}[Energy estimates]\label{EE_prop}
		Let $\u_n(\cdot)$ be the unique solution of the system of stochastic ODEs \eqref{EE1} with $\u_0\in\H$, $\f\in \mathrm{L}^{p'}(0,T;\V_p')$ and {$\xi>0$}. Then, 
		for $ q \in\Big\{2,\frac{2(2p-d)}{p-1},{4+\xi}\Big\},$	we have
		\begin{align}\label{EE002}\nonumber
			&\E\bigg[\sup_{t\in[0,T]}	\|\u_n(t)\|_{\H}^q\bigg]+\frac{\mu q}{2}\E\bigg[\int_0^{T}\|\u_n(s)\|_\H^{q-2}\| \u_n(s)\|_\V^2\d s\bigg]\\&\nonumber\qquad +\frac{\mu q}{4}\E\bigg[\int_0^{T}\|\u_n(s)\|_\H^{q-2}\| \u_n(s)\|_{\V_p}^p\d s\bigg]+q\beta \E\bigg[\int_0^{T}\|\u_n(s)\|_\H^{q-2}\|\u_n(s)\|_{\widetilde{\L}^r}^r\d s\bigg]\\&\leq \bigg(\|\u_0\|_{\H}^q+C(\mu,p,q)\int_0^T\|\f(s)\|_{\V_p'}^{p'}\d s+C(K,q)T\bigg)e^{C(\mu,p,q,K)T},
		\end{align}and for $q=\frac{2d}{p-1}$, 
		\begin{align}\label{EE0125}\nonumber
			&	\E\Bigg[\left(\int_0^{T} \|\u_n(s)\|_{\V}^2\d s\right)^{\frac{q}{2}}+\left(\int_0^{T} \|\u_n(s)\|_{\V_p}^p\d s\right)^{\frac{q}{2}}+\left( \int_0^{T} \|\u_n(s)\|_{\widetilde{\L}^r}^r \d s\right)^{\frac{q}{2}}\Bigg]\\&\leq  C(q)
			\bigg(\|\u_0\|_{\H}^q+C(\mu,p,q)\bigg(\int_0^T\|\f(s)\|_{\V_p'}^{p'}\d s\bigg)^{\frac{q}{2}}+C(K,q,T)\bigg)e^{C(\mu,p,q,K,T)}.
		\end{align}
	\end{proposition}
	\begin{proof}
		%
		Let us define a sequence of stopping times $\tau_N^n$ by 
		\begin{align}\label{EE3}
			\tau_N^n:=\inf_{t\geq 0}\bigg\{t: \|\u_n(t)\|_\H^q+ \int_0^t\|\u_n(s)\|_{\H}^{q-2}\|\u_n(s)\|_{\V_p}^p\d s\geq N\bigg\},
		\end{align}for $N\in\N$.  We apply  finite-dimensional It\^o's formula to the process $\|\u_n(\cdot)\|_\H^q$, for large enough $q$, to find for all $t\in[0,T]$, $\PP$-a.s.,
		\begin{align}\label{EE0021}\nonumber
			&	\|\u_n(t\wedge\tau_N^n)\|_{\H}^q \\&\nonumber= \|\u_n(0)\|_\H^q-q\int_0^{t\wedge\tau_N^n} \|\u_n(s)\|_{\H}^{q-2}\langle \mu\mathcal{A}_n (\u_n(s))+\B_n(\u_n(t))+\beta \mathcal{C}_n(\u_n(s))-\f_n(s),\u_n(s)\rangle \d s\\&\nonumber\quad + \frac{q(q-1)}{2}\sum_{k=1}^n\int_0^{t\wedge\tau_N^n}\|\u_n(s)\|_\H^{q-2}\|\boldsymbol{\sigma}_k(\u_n(s))\|_\H^2\d s\\&\nonumber\quad +q\sum_{k=1}^n\int_0^{t\wedge\tau_N^n}\|\u_n(s)\|_\H^{q-2}\big(\boldsymbol{\sigma}_k(s,\u_n(s))\d\boldsymbol{W}_k(s),\u_n(s)\big)\\&\nonumber\leq 
			\|\u_0\|_\H^q-\frac{\mu q}{2}\int_0^{t\wedge\tau_N^n}\|\u_n(s)\|_\H^{q-2}\big[\| \u_n(s)\|_\V^2+\| \u_n(s)\|_{\V_p}^p\big]\d s\\&\nonumber\quad+ C(K,q)\int_0^{t\wedge\tau_N^n}
			\big(
			1+\|\u_n(s)\|_\H^q)\d s-q\beta \int_0^{t\wedge\tau_N^n}\|\u_n(s)\|_\H^{q-2}\|\u_n(s)\|_{\widetilde{\L}^r}^r\d s\\&\nonumber\quad+q\int_0^{t\wedge\tau_N^n}\|\u_n(s)\|_\H^{q-2}\|\f_n(s)\|_{\V_p'}\|\u_n(s)\|_{\V_p}\d s\\&\quad +q\sum_{k=1}^n\int_0^{t\wedge\tau_N^n}\|\u_n(s)\|_\H^{q-2}\big(\boldsymbol{\sigma}_k(s,\u_n(s))\d\boldsymbol{W}_k(s),\u_n(s)\big),
		\end{align}where we have used the fact that $\langle \B(\u),\u\rangle=0$, $\|\u_n(0)\|_\H\leq \|\u_0\|_\H$,  \eqref{2.2} and Hypothesis \ref{hyp} (H.1).

		Using Young's inequality in \eqref{EE0021}, we obtain for all $t\in[0,T]$, $\PP$-a.s.,
		\begin{align}\label{EE0022}\nonumber
			&	\|\u_n(t\wedge\tau_N^n)\|_{\H}^q+\frac{\mu q}{2}\int_0^{t\wedge\tau_N^n}\|\u_n(s)\|_\H^{q-2}\| \u_n(s)\|_\V^2\d s+\frac{\mu q}{4}\int_0^{t\wedge\tau_N^n}\|\u_n(s)\|_\H^{q-2}\| \u_n(s)\|_{\V_p}^p\d s\\&\nonumber\qquad +q\beta \int_0^{t\wedge\tau_N^n}\|\u_n(s)\|_\H^{q-2}\|\u_n(s)\|_{\widetilde{\L}^r}^r\d s \\&\nonumber\leq \|\u_0\|_\H^q +C(\mu,p,q)\int_0^t\|\f(s)\|_{\V_p'}^{p'}\d s+C(K,q)T+C(\mu,p,q,K)\int_0^{t\wedge\tau_N^n}\|\u_n(s)\|_\H^q\d s \\&\quad +q\sum_{k=1}^n\int_0^{t\wedge\tau_N^n}\|\u_n(s)\|_\H^{q-2}\big(\boldsymbol{\sigma}_k(s,\u_n(s))\d\boldsymbol{W}_k(s),\u_n(s)\big).
		\end{align}Let us take expectation on both sides of the inequality \eqref{EE0022}, and the fact that the final term in the right hand side of \eqref{EE0022} is a martingale to find 
		\begin{align}\label{EE0122}\nonumber
			&\E\bigg[	\|\u_n(t\wedge\tau_N^n)\|_{\H}^q+\frac{\mu q}{2}\int_0^{t\wedge\tau_N^n}\|\u_n(s)\|_\H^{q-2}\| \u_n(s)\|_\V^2\d s+\frac{\mu q}{4}\int_0^{t\wedge\tau_N^n}\|\u_n(s)\|_\H^{q-2}\| \u_n(s)\|_{\V_p}^p\d s\bigg]\\&\nonumber\qquad +q\beta\E\bigg[ \int_0^{t\wedge\tau_N^n}\|\u_n(s)\|_\H^{q-2}\|\u_n(s)\|_{\widetilde{\L}^r}^r\d s\bigg] \\&\leq \|\u_0\|_\H^q +C(\mu,p,q)\int_0^t\|\f(s)\|_{\V_p'}^{p'}\d s+C(K,q)T+C(\mu,p,q,K)\E\bigg[\int_0^{t\wedge\tau_N^n}\|\u_n(s)\|_\H^q\d s\bigg].
		\end{align}An application of Gronwall's inequality in \eqref{EE0122} yields
		\begin{align}\label{EE01221}
			\E\big[	\|\u_n(t\wedge\tau_N^n)\|_{\H}^q\big]\leq \bigg(\|\u_0\|_\H^q +C(\mu,p,q)\int_0^t\|\f(s)\|_{\V_p'}^{p'}\d s+C(K,q)T\bigg)e^{C(\mu,p,q,K)T},
		\end{align}for $t\in[0,T]$. It can be shown that (see \cite[Proposition 3.5]{MTM8})
		\begin{align}\label{EE9}
			\lim_{N\to\infty}\mathbb{P} \big\{\omega\in\Omega:\tau_N^n(\omega)<t\big\}=0, \text{ for all } \ t\in[0,T],
		\end{align}and $t\wedge \tau_N^n \to t$, as $N\to\infty$.  Then on passing limit $N\to\infty$ in \eqref{EE01221} and using the Monotone Convergence Theorem, we find 
		\begin{align}\label{EE10}
			\E\big[	\|\u_n(t)\|_{\H}^q\big] \leq \bigg(\|\u_0\|_\H^q +C(\mu,p,q,K)\int_0^t\|\f(s)\|_{\V_p'}^{p'}\d s+C(K,q)T\bigg)e^{C(\mu,p,q,K)T},
		\end{align}for $t\in[0,T]$. Substituting \eqref{EE10} in \eqref{EE0122}, we deduce 
		\begin{align}\label{EE11}\nonumber
			&\E\bigg[	\|\u_n(t)\|_{\H}^q+\frac{\mu q}{2}\int_0^{t}\|\u_n(s)\|_\H^{q-2}\| \u_n(s)\|_\V^2\d s\d s+\frac{\mu q}{4}\int_0^{t}\|\u_n(s)\|_\H^{q-2}\| \u_n(s)\|_{\V_p}^p\d s\bigg]\\&\nonumber\qquad +q\beta\E\bigg[ \int_0^{t\wedge\tau_N^n}\|\u_n(s)\|_\H^{q-2}\|\u_n(s)\|_{\widetilde{\L}^r}^r\d s\bigg] 	 \\&\leq 
			\bigg(\|\u_0\|_\H^q +C(\mu,p,q)\int_0^t\|\f(s)\|_{\V_p'}^{p'}\d s+C(K,q)T\bigg)e^{C(\mu,p,q,K)T},
		\end{align}for $t\in[0,T]$. Note that the right hand side of the above inequality is independent of $n$.

		Taking supremum from $0$ to $T\wedge\tau_N^n$, and then taking expectation in \eqref{EE0022}, we find
		\begin{align}\label{EE0023}\nonumber
			&\E\bigg[\sup_{t\in[0,T\wedge \tau_N^n]}	\|\u_n(t)\|_{\H}^q\bigg]+\frac{\mu q}{2}\E\bigg[\int_0^{T\wedge \tau_N^n}\|\u_n(s)\|_\H^{q-2}\| \u_n(s)\|_\V^2\d s\d s\bigg]\\&\nonumber\qquad +\frac{\mu q}{4}\E\bigg[\int_0^{T\wedge \tau_N^n}\|\u_n(s)\|_\H^{q-2}\| \u_n(s)\|_{\V_p}^p\d s\bigg]+q\beta \E\bigg[\int_0^{T\wedge \tau_N^n}\|\u_n(s)\|_\H^{q-2}\|\u_n(s)\|_{\widetilde{\L}^r}^r\d s\bigg] \\&\nonumber\leq \|\u_0\|_\H^q +C(\mu,p,q)\int_0^T\|\f(s)\|_{\V_p'}^{p'}\d s+C(K,q)T+C(\mu,p,q,K)\E\bigg[\int_0^{T\wedge\tau_N^n}\|\u_n(s)\|_\H^q\d s\bigg] \\&\quad +q\E\bigg[\sup_{t\in[0,T\wedge\tau_N^n]}\bigg|\sum_{k=1}^n\int_0^t\|\u_n(s)\|_\H^{q-2}\big(\boldsymbol{\sigma}_k(s,\u_n(s))\d\boldsymbol{W}_k(s),\u_n(s)\big)\bigg|\bigg].
		\end{align}Let us consider the final term appearing in the right hand side of the inequality \eqref{EE0023}, and estimate it using Burkholder-Davis-Gundy, H\"older's and Young's inequalities and Hypothesis \ref{hyp} (H.1), as
		\begin{align}\label{EE0024}\nonumber
			&q	\E\bigg[\sup_{t\in[0,T\wedge\tau_N^n]}\bigg|\sum_{k=1}^n\int_0^t\|\u_n(s)\|_\H^{q-2}\big(\boldsymbol{\sigma}_k(s,\u_n(s))\d\boldsymbol{W}_k(s),\u_n(s)\big)\bigg|\bigg]\\&\nonumber\leq 
			C(q)\E\bigg[\bigg(\sum_{k=1}^n\int_0^{T\wedge \tau_N^n}\|\u_n(s)\|_\H^{2q-2}\|\boldsymbol{\sigma}_k(s,\u_n(s))\|_\H^2\d s\bigg)^\frac{1}{2}\bigg]
			\\&\nonumber\leq 
			C(q)\E\bigg[\bigg(\|\u_n(s)\|_\H^{q}\sum_{k=1}^n\int_0^{T\wedge \tau_N^n}\|\u_n(s)\|_\H^{q-2}\|\boldsymbol{\sigma}_k(s,\u_n(s))\|_\H^2\d s\bigg)^\frac{1}{2}\bigg]\\&\nonumber\leq 
			\e \E\bigg[\sup_{t\in[0,T\wedge \tau_N^n]}\|\u_n(s)\|_\H^q\bigg]+C(\e,q)\sum_{k=1}^n\E\bigg[\int_0^T\|\u_n(s)\|_\H^{q-2}\|\boldsymbol{\sigma}_k(s,\u_n(s))\|_{\H}^2\d s\bigg]
			\\&\leq 
			\e \E\bigg[\sup_{t\in[0,T\wedge \tau_N^n]}\|\u_n(s)\|_\H^q\bigg]+C(\e,K,q)T+C(\e,K,q)\E\bigg[\int_0^T\|\u_n(s)\|_\H^{q}\d s\bigg],
		\end{align}for some $\e>0$.
		
		Substituting \eqref{EE0024} in \eqref{EE0023} with an appropriate choice of $\e$ and then applying  Gronwall's inequality in the resultant to deduce 
		\begin{align}\label{EE0025}
			\E\bigg[\sup_{t\in[0,T\wedge \tau_N^n]}	\|\u_n(t)\|_{\H}^q\bigg] \leq \bigg(\|\u_0\|_{\H}^q+C(\mu,p,q)\int_0^T\|\f(s)\|_{\V_p'}^{p'}\d s+C(K,q)T\bigg)e^{C(\mu,p,q,K)T}.
		\end{align}Passing $N\to\infty$, using the Monotone Convergence Theorem and then substituting \eqref{EE0025} in \eqref{EE0023} produce the required result \eqref{EE002}.

		One can obtain the estimate \eqref{EE0125} by choosing $q=2$ in \eqref{EE002} and then rising the $\frac{q}{2}$-{th} power on both sides of the resultant and then performing the similar calculations as in the proof of \eqref{EE002}. 
	\end{proof}
\begin{remark}
	One can work with general unbounded domains also with positive $\alpha$ (damping coefficient appearing in the system \eqref{31}). Due to unavailability of the Poincar\'e inequality \eqref{2.1}, one has to work with the full norm of the spaces $\V$ and $\V_p$. In this remark, we are providing the energy estimates and necessary modification in the case of general unbounded domains. 
		
		Let $\u_n(\cdot)$ be the unique solution of the system of stochastic ODEs \eqref{EE1} with $\u_0\in\H$, $\f\in \mathrm{L}^{p'}(0,T;\V_p')$, {$\xi>0$} and {$r=p$}. Then, 
		for $ q \in\Big\{2,\frac{2(2p-d)}{p-1},{4+\xi}\Big\}$,	we have
		\begin{align}\label{R.EE002}\nonumber
			&\E\bigg[\sup_{t\in[0,T]}	\|\u_n(t)\|_{\H}^q\bigg]+qC_{\alpha,\mu/2}\E\bigg[\int_0^{T}\|\u_n(s)\|_\H^{q-2}\| \u_n(s)\|_\V^2\d s\bigg]\\&\nonumber\qquad +\frac{qC_{\beta,\mu/2}}{2}\E\bigg[\int_0^{T}\|\u_n(s)\|_\H^{q-2}\| \u_n(s)\|_{\V_p}^p\d s\bigg]
			\\&\leq \bigg(\|\u_0\|_{\H}^q+C(\mu,\beta,p,q)\int_0^T\|\f(s)\|_{\V_p'}^{p'}\d s+C(K,q)T\bigg)e^{C(\mu,p,q,K)T},
		\end{align}where the constants are given by 
		\begin{align}\label{R.FNorm}
		C_{\alpha,\mu/2}:=\min\big\{\alpha,\mu/2\big\},  \ \  C_{\beta,\mu/2}:=\min\big\{\beta,\mu/2\big\}, \   \text{ and }  \ r=p,
	\end{align}
	and for $q=\frac{2d}{p-1}$, 
		\begin{align}\label{R.EE0125}\nonumber
			&	\E\Bigg[\left(\int_0^{T} \|\u_n(s)\|_{\V}^2\d s\right)^{\frac{q}{2}}+\left(\int_0^{T} \|\u_n(s)\|_{\V_p}^p\d s\right)^{\frac{q}{2}}
			\Bigg]\\&\leq  C(q)
			\bigg(\|\u_0\|_{\H}^q+C(\mu,\beta,p,q)\bigg(\int_0^T\|\f(s)\|_{\V_p'}^{p'}\d s\bigg)^{\frac{q}{2}}+C(K,q,T)\bigg)e^{C(\mu,p,q,K,T)}.
		\end{align}

		%
		Let us define a sequence of stopping times $\tau_N^n$ by 
		\begin{align}\label{R.EE3}
			\tau_N^n:=\inf_{t\geq 0}\bigg\{t: \|\u_n(t)\|_\H^q+ \int_0^t\|\u_n(s)\|_{\H}^{q-2}\|\u_n(s)\|_{\V_p}^p\d s\geq N\bigg\},
		\end{align}for $N\in\N$.  We apply  finite-dimensional It\^o's formula to the process $\|\u_n(\cdot)\|_\H^q$, for large enough $q$, to find for all $t\in[0,T]$, $\PP$-a.s.,
		\begin{align}\label{R.EE0021}\nonumber
			&	\|\u_n(t\wedge\tau_N^n)\|_{\H}^q 
			\\&\nonumber
			\leq 
			\|\u_0\|_\H^q-q\int_0^{t\wedge\tau_N^n}\|\u_n(s)\|_\H^{q-2}\Big[\alpha \|\u_n(s)\|_{\H}^2+\frac{\mu}{2}\|\nabla  \u_n(s)\|_{\H}^2\Big]\d s\\&\nonumber\quad
			-q \int_0^{t\wedge\tau_N^n}	\|\u_n(s)\|_\H^{q-2} \Big[\beta \|\u_n(s)\|_{\widetilde{\L}^r}^r+\frac{\mu}{2}\|\nabla  \u_n(s)\|_{\wi{\L}^p}^p \Big]\d s+ C(K,q)\int_0^{t\wedge\tau_N^n}
			\big(
			1+\|\u_n(s)\|_\H^q)\d s\\&\nonumber\quad+q\int_0^{t\wedge\tau_N^n}\|\u_n(s)\|_\H^{q-2}\|\f_n(s)\|_{\V_p'}\|\u_n(s)\|_{\V_p}\d s\\&\nonumber\quad +q\sum_{k=1}^n\int_0^{t\wedge\tau_N^n}\|\u_n(s)\|_\H^{q-2}\big(\boldsymbol{\sigma}_k(s,\u_n(s))\d\boldsymbol{W}_k(s),\u_n(s)\big)
			\\&\nonumber
			\leq 
			\|\u_0\|_\H^q-qC_{\alpha,\mu/2}\int_0^{t\wedge\tau_N^n}\|\u_n(s)\|_\H^{q-2} \|\u_n(s)\|_{\V}^2\d s\\&\nonumber\quad
			-q C_{\beta,\mu/2}\int_0^{t\wedge\tau_N^n}	\|\u_n(s)\|_\H^{q-2} \|  \u_n(s)\|_{\V_p}^p \d s+ C(K,q)\int_0^{t\wedge\tau_N^n}
			\big(
			1+\|\u_n(s)\|_\H^q)\d s\\&\nonumber\quad+q\int_0^{t\wedge\tau_N^n}\|\u_n(s)\|_\H^{q-2}\|\f_n(s)\|_{\V_p'}\|\u_n(s)\|_{\V_p}\d s\\&\quad +q\sum_{k=1}^n\int_0^{t\wedge\tau_N^n}\|\u_n(s)\|_\H^{q-2}\big(\boldsymbol{\sigma}_k(s,\u_n(s))\d\boldsymbol{W}_k(s),\u_n(s)\big),
		\end{align}	where we have used the fact that $\langle \B(\u),\u\rangle=0$, $\|\u_n(0)\|_\H\leq \|\u_0\|_\H$,  \eqref{2.2} and Hypothesis \ref{hyp} (H.1), 
		and the constants are given in \eqref{R.FNorm}.
%
%
%
	Let us take expectation on both sides of the inequality \eqref{R.EE0021}, and use the fact that the final term in the right hand side of \eqref{R.EE0021} is a martingale to find 
		\begin{align}\label{R.EE0122}\nonumber
			&\E\bigg[	\|\u_n(t\wedge\tau_N^n)\|_{\H}^q+qC_{\alpha,\mu/2}\int_0^{t\wedge\tau_N^n}\|\u_n(s)\|_\H^{q-2}\| \u_n(s)\|_\V^2\d s\bigg]\\&\quad \nonumber+\frac{qC_{\beta,\mu/2}}{2}\E\bigg[\int_0^{t\wedge\tau_N^n}\|\u_n(s)\|_\H^{q-2}\| \u_n(s)\|_{\V_p}^p\d s\bigg]
			\\&\leq \|\u_0\|_\H^q +C(\mu,\beta,p,q)\int_0^t\|\f(s)\|_{\V_p'}^{p'}\d s+C(K,q)T+C(\mu,p,q,K)\E\bigg[\int_0^{t\wedge\tau_N^n}\|\u_n(s)\|_\H^q\d s\bigg].
		\end{align}
The remaining parts of the proofs of estimates \eqref{R.EE002} and \eqref{R.EE0125} follow along similar lines as the proof of Proposition \ref{EE_prop}.
\end{remark}

	\begin{lemma}\label{EU_lem1}
		For $p\geq \frac{d}{2}+1$, $r\geq 2$, $ q\in\Big\{2,\frac{2(2p-d)}{p-1},{4+\xi}\Big\}$, and any functions 
		\begin{align*}
			\u,\v\in\mathrm{L}^q(\Omega;\mathrm{L}^\infty(0,T;\H))\cap \mathrm{L}^p(\Omega;\mathrm{L}^p(0,T;\V_p))\cap \mathrm{L}^r(\Omega;\mathrm{L}^r(0,T;\widetilde{\L}^{r}))
		\end{align*}we have 
		\begin{align}\label{EE17}\nonumber
			&\int_0^Te^{-\widetilde{\eta}t}\Big[\mu\langle \mathcal{A}(\u(t))-\mathcal{A}(\v(t)),\u(t)-\v(t)\rangle+\langle \B(\u(t))-\B(\v(t)),\u(t)-\v(t)\rangle \\&\nonumber\quad+\beta \langle \mathcal{C}(\u(t))-\mathcal{C}(\v(t)),\u(t)-\v(t)\rangle \Big]\d t+\Big(\widetilde{\eta}+\frac{L}{2}\Big)\int_0^Te^{-\widetilde{\eta}t}\|\v(t)\|_{\V_p}^{\frac{2p}{2p-d}}\|\u(t)-\v(t)\|_{\H}^2\d t \\&\geq \frac{1}{2}\sum_{k=1}^\infty\int_0^Te^{-\widetilde{\eta}t}\|\boldsymbol{\sigma}_k(t,\u(t))-\boldsymbol{\sigma}_k(t,\v(t))\|_{\H}^2\d t,
		\end{align}where $\widetilde{\eta}$ is defined in \eqref{eta1}.
	\end{lemma}
	\begin{proof}
		Multiplying \eqref{mon1} with $e^{-\widetilde{\eta}t}$, integrating over time $t\in[0,T]$, and using the Hypothesis \ref{hyp} (H.2), we obtain the required inequality \eqref{EE17}.
	\end{proof}

Let us recall some results from the books \cite{PPB,VIB}, which are essential to establish  the energy equality (see \textbf{Step 4} below).
\begin{definition}[{\cite[Definition 4.5.1]{VIB}}]
	A set of functions $\mathcal{M}\subset \mathrm{L}^1(\Omega,\mathscr{F},\PP)$ is called uniformly integrable if 
	\begin{align}\label{UID1}
		\lim_{N\to\infty}\sup_{g\in\mathcal{M}}\int_{\{|g|>N\}}|g|\d\PP=0,
	\end{align}that is, for every $\e>0$, there exists $N_\e>0$ such that 
\begin{align}\label{UID2}
	\sup_{g\in\mathcal{M}}\int_{\{|g|>N_\e\}}|g|\d\PP<\e.
\end{align}
\end{definition}
\begin{theorem}[Lebesgue-Vitali theorem, {\cite[Theorem 4.5.4]{VIB}}]\label{LVT}
	Let $\{g_n\}_{n\in\N}\subset \mathrm{L}^p(\Omega,\mathscr{F},\PP)$ and $g$ be an $\mathscr{F}$-measurable function. Then, the following are equivalent:
	\begin{enumerate}
		\item $g\in \mathrm{L}^p(\Omega,\mathscr{F},\PP)$ and $\{g_n\}_{n\in\N}$ converges to $g$ in $ \mathrm{L}^p(\Omega,\mathscr{F},\PP)$;
		\item The sequence of functions $\{g_n\}_{n\in\N}$ converges in $\mathbb{P}$-measure to $g$ and $\{|g_n|^p\}_{n\in\N}$ is uniformly integrable.
	\end{enumerate}
\end{theorem}
\begin{remark}\label{rem4.5}
	 From \cite[pp. 31]{PPB}, a simple condition for uniform integrability is that 
	 \begin{align*}
	 	\sup_{n\in\N}\E\big[|g_n|^{1+\e}\big]<\infty,\ \text{ for some }\ \e>0,
	 \end{align*}in this case, an application of H\"older's and Markov's inequalities yield
 \begin{align*}
 	\int_{\{|g_n|>N\}}|g_n|\d\PP \leq \frac{1}{N_\e}\E\big[|g_n|^{1+\e}\big].
 \end{align*} 
\end{remark}

	\subsection{Existence and uniqueness of strong solution}
	In this subsection, we discuss the existence and uniqueness of strong solution (probabilistic sense) of the system \eqref{32} by exploiting the local monotonicity property and a stochastic generalization of the Minty-Browder technique. Similar existence results for different types of stochastic partial differential equations have been explored in the works \cite{ICAM,MTM8,MJSS,MTM6,SSSP}, etc.
	
	\begin{proof}[Proof of Theroem \ref{exis2}]
		The global solvability results to the system \eqref{32} is divided into the following steps:
		
		\vskip2mm 
		\noindent
		\textbf{Step 1:} \emph{Finite-dimensional (Galerkin) approximation to the system \eqref{32}}:  Let us first consider the following Galerkin approximated It\^o stochastic differential equation satisfied by $\u_n(\cdot)$:
		\begin{equation}\label{EE18}
			\left\{
			\begin{aligned}
				\d \u_n(t)&= -\mathscr{H}_n(\u_n(t))\d t+\sum_{k=1}^n\boldsymbol{\sigma}_k(t,\u_n(t))\d \boldsymbol{W}_k(t), \\
				\u_n(0)&=\u_0^n,
			\end{aligned}
			\right.
		\end{equation}where $$\mathscr{H}_n(\u_n(\cdot))=\mu\mathcal{A}(\u_n(\cdot))+\B_n(\u_n(\cdot))+\beta\mathcal{C}_n(\u_n(\cdot))-\f_n(\cdot).$$  Applying finite-dimensional It\^o's formula to the process $e^{-\rho(\cdot)}\|\u_n(\cdot)\|_{\H}^2$, we obtain for all $t\in[0,T]$, $\mathbb{P}$-a.s.,
		\begin{align}\label{EE19}\nonumber
			&e^{-\rho(t)}\|\u_n(t)\|_{\H}^2\\&\nonumber=\|\u_0^n\|_{\H}^2-\int_0^t e^{-\rho(s)}\langle2 \mathscr{H}_n(\u_n(s))+\rho'(s)\u_n(s),\u_n(s)\rangle \d s\\&\quad +\sum_{k=1}^n\int_0^te^{-\rho(s)} \|\boldsymbol{\sigma}_k(s,\u_n(s))\|_{\H}^2\d s +2\sum_{k=1}^n\int_0^te^{-\rho(s)}\big(\boldsymbol{\sigma}_k(s,\u_n(s))\d \boldsymbol{W}_k(s),\u_n(s)\big),
		\end{align}where $\rho(\cdot)$ is defined later (see \eqref{EE69} below).
		
		Note that the final term in the right hand side of the above equality is a martingale and on taking expectation on both sides, we find for all $t\in[0,T]$
		\begin{align}\label{EE20}\nonumber
			\E\Big[e^{-\rho(t)}\|\u_n(t)\|_{\H}^2\Big]&=\|\u_0^n\|_{\H}^2-\E\bigg[\int_0^t e^{-2\rho(s)}\langle 2\mathscr{H}_n(\u_n(s))+\rho'(s)\u_n(s),\u_n(s)\rangle \d s\bigg]\\&\quad +\E\bigg[\sum_{k=1}^n\int_0^te^{-\rho(s)} \|\boldsymbol{\sigma}_k(s,\u_n(s))\|_{\H}^2\d s\bigg].
		\end{align}
		
		\vskip2mm 
		\noindent
		\textbf{Step 2:} \emph{Weak convergence of the sequences $\u_n(\cdot),\ \mathscr{H}_n(\u_n(\cdot))$ and $\boldsymbol{\sigma}^n(\cdot,\cdot)\  (=:\Pi_n\boldsymbol{\sigma}(\cdot,\cdot))$:} We know that $\mathrm{L}^2(\Omega;\mathrm{L}^\infty(0,T;\H))\cong \mathrm{L}^2(\Omega;\mathrm{L}^1(0,T;\H))'$ and the space $\mathrm{L}^2(\Omega;\mathrm{L}^1(0,T;\H))$ is separable. Moreover, the spaces $\mathrm{L}^p(\Omega;\mathrm{L}^p(0,T;\V_p))$ and $\mathrm{L}^r(\Omega;\mathrm{L}^r(0,T;\widetilde{\L}^r))$ are reflexive ($\X''\cong \X$). From the energy estimate \eqref{EE002} (for $q=2$) given in Proposition \ref{EE_prop}, we know that the sequence $\{\u_n(\cdot)\}$ is uniformly bounded in the spaces $\mathrm{L}^2(\Omega;\mathrm{L}^\infty(0,T;\H))$, $\mathrm{L}^p(\Omega;\mathrm{L}^p(0,T;\V_p))$ and $\mathrm{L}^r(\Omega;\mathrm{L}^r(0,T;\widetilde{\L}^r))$. Using the Banach-Alaoglu theorem, we can find a subsequence $\{\u_{n_k}(\cdot)\}$ of $\{\u_n(\cdot)\}$ such that (we still denote the index $n_k$ of subsequence by $n$):
		\begin{equation}\label{EE21}
			\left\{
			\begin{aligned}
				\u_n &\xrightarrow{w^*} \u, \ \text{ in}  &&\mathrm{L}^2(\Omega;\mathrm{L}^\infty(0,T;\H),\\
				\u_n &\xrightarrow{w} \u, \ \text{ in}  &&\mathrm{L}^p(\Omega;\mathrm{L}^p(0,T;\V_p)),\\
				\u_n &\xrightarrow{w} \u,	 \ \text{ in}  &&\mathrm{L}^r(\Omega;\mathrm{L}^r(0,T;\widetilde{\L}^r)),\\
				\mathscr{H}_n(\u_n) &\xrightarrow{w} \mathscr{H}_0, \ \text{ in} && \mathrm{L}^{p'}(\Omega;\mathrm{L}^{p'}(0,T;\V_p'))+ \mathrm{L}^{\frac{r}{r-1}}(\Omega;\mathrm{L}^{\frac{r}{r-1}}(0,T;\widetilde{\L}^{\frac{r}{r-1}})).
			\end{aligned}
			\right.
		\end{equation}The final convergence in \eqref{EE20} can be justified as follows:
		\begin{align}\label{EE021}\nonumber
			&\E\bigg[\bigg|\int_0^T \langle \mathscr{H}_n(\u_n(t)),\v(t)\rangle\d t \bigg|\bigg]\\&\nonumber  \leq 
			\mu\E\bigg[\bigg|\int_0^T\langle \mathcal{A}_n(\u_n(t)),\v(t)\rangle \d t\bigg]+\E\bigg[\bigg|\int_0^T \langle \B_n(\u_n(t),\u_n(t)),\v(t)\rangle \d t \bigg|\bigg]\\&\nonumber\quad + \beta \E\bigg[\bigg|\int_0^T \langle \mathcal{C}_n(\u_n(t)),\v(t)\rangle \d t\bigg]+ \E\bigg[\bigg|\int_0^T \langle \f_n(t), \v(t)\rangle\d t  \bigg|\bigg]\\&\nonumber  \leq
			C(p,\mu)\E\bigg[\int_0^T \big(1+\|\u_n(t)\|_{\V_p}^{p-1}\big)\|\v(t)\|_{\V_p}\d t \bigg]+\E\bigg[\int_0^T\|\u_n(t)\|_{\widetilde{\L}^{\frac{2p}{p-1}}}^2\|\v(t)\|_{\V}\d t\bigg] \\&\nonumber\quad +
			C(\beta)\E\bigg[\int_0^T \|\u_n(t)\|_{\widetilde{\L}^r}^{r-1}\|\v(t)\|_{\widetilde{\L}^r}\d t\bigg]+\E\bigg[\int_0^T\|\f_n(t)\|_{\V_p'}\|\v(t)\|_{\V_p}\d t\bigg]\\&\nonumber  \leq
			C(p,\mu)\E\bigg[T^{\frac{p-1}{p}}\bigg(\int_0^T\|\v(t)\|_{\V_p}^p\d t\bigg)^\frac{1}{p}+ \bigg(\int_0^T\|\u_n(t)\|_{\V_p}^p\d t\bigg)^\frac{p-1}{p}\bigg(\int_0^T\|\v(t)\|_{\V_p}^p\d t\bigg)^\frac{1}{p}\bigg]\\&\nonumber\quad +\E\bigg[\bigg(\int_0^T\|\u_n(t)\|_{\V}^{\frac{d}{p-1}}\|\u_n(t)\|_{\H}^{\frac{2p-d}{p-1}}\d t\bigg)^{\frac{p-1}{p}} \bigg(\int_0^T\|\v(t)\|_{\V_p}^p\d t\bigg)^\frac{1}{p}\bigg]\\&\nonumber\quad + 
			C(\beta) \E\bigg[\bigg(\int_0^T\|\u_n(t)\|_{\widetilde{\L}^r}^r\d t\bigg)^\frac{r-1}{r}\bigg(\int_0^T\|\v(t)\|_{\widetilde{\L}^r}^r\d t\bigg)^\frac{1}{r}\bigg]\\&\nonumber\quad + 
			\E\bigg[\bigg(\int_0^T\|\f_n(t)\|_{\V_p'}^{p'}\d t\bigg)^\frac{1}{p'}\bigg(\int_0^T\|\v(t)\|_{\V_p}^p\d t\bigg)^\frac{1}{p}\bigg]
			\\&\nonumber  \leq
			C(p,\mu)T^{\frac{p-1}{p}}\bigg\{\E\bigg[\int_0^T\|\v(t)\|_{\V_p}^p\d t\bigg]\bigg\}^\frac{1}{p}\\&\nonumber\quad+C(p,\mu) \bigg\{\E\bigg[\int_0^T\|\u_n(t)\|_{\V_p}^p\d t\bigg]\bigg\}^\frac{p-1}{p}\bigg\{\E\bigg[\int_0^T\|\v(t)\|_{\V_p}^p\d t\bigg]\bigg\}^\frac{1}{p}\\&\nonumber\quad +\bigg\{\E\bigg[\int_0^T\|\u_n(t)\|_{\V}^{\frac{d}{p-1}}\|\u_n(t)\|_{\H}^{\frac{2p-d}{p-1}}\d t\bigg]\bigg\}^{\frac{p-1}{p}}\bigg\{\E\bigg[ \int_0^T\|\v(t)\|_{\V_p}^p\d t\bigg]\bigg\}^\frac{1}{p}\\&\nonumber\quad + 
			C(\beta) \bigg\{\E\bigg[\int_0^T\|\u_n(t)\|_{\widetilde{\L}^r}^r\d t\bigg]\bigg\}^\frac{r-1}{r}\bigg\{\E\bigg[\int_0^T\|\v(t)\|_{\widetilde{\L}^r}^r\d t\bigg]\bigg\}^\frac{1}{r}\\&\nonumber\quad + 
			\bigg\{\E\bigg[\int_0^T\|\f_n(t)\|_{\V_p'}^{p'}\d t\bigg]\bigg\}^\frac{1}{p'}\bigg\{\E\bigg[\int_0^T\|\v(t)\|_{\V_p}^p\d t\bigg]\bigg\}^\frac{1}{p}
			\\&\nonumber  \leq
			C(p,\mu)T^{\frac{p-1}{p}}\bigg\{\E\bigg[\int_0^T\|\v(t)\|_{\V_p}^p\d t\bigg]\bigg\}^\frac{1}{p}\\&\nonumber\quad+C(p,\mu) \bigg\{\E\bigg[\int_0^T\|\u_n(t)\|_{\V_p}^p\d t\bigg]\bigg\}^\frac{p-1}{p}\bigg\{\E\bigg[\int_0^T\|\v(t)\|_{\V_p}^p\d t\bigg]\bigg\}^\frac{1}{p}\\&\nonumber\quad +T^{\frac{2p-2-d}{2p}}\bigg\{\E\bigg[\sup_{t\in[0,T]}\|\u_n(t)\|_{\H}^{\frac{2p-d}{p-1}}\bigg(\int_0^T\|\u_n(t)\|_{\V}^{2}\d t\bigg)^\frac{d}{2(p-1)}\bigg]\bigg\}^{\frac{p-1}{p}}\bigg\{\E\bigg[ \int_0^T\|\v(t)\|_{\V_p}^p\d t\bigg]\bigg\}^\frac{1}{p}\\&\nonumber\quad + 
			C(\beta) \bigg\{\E\bigg[\int_0^T\|\u_n(t)\|_{\widetilde{\L}^r}^r\d t\bigg]\bigg\}^\frac{r-1}{r}\bigg\{\E\bigg[\int_0^T\|\v(t)\|_{\widetilde{\L}^r}^r\d t\bigg]\bigg\}^\frac{1}{r}\\&\nonumber\quad + 
			\bigg\{\E\bigg[\int_0^T\|\f_n(t)\|_{\V_p'}^{p'}\d t\bigg]\bigg\}^\frac{1}{p'}\bigg\{\E\bigg[\int_0^T\|\v(t)\|_{\V_p}^p\d t\bigg]\bigg\}^\frac{1}{p}
			\\&\nonumber  \leq
			C(p,\mu)T^{\frac{p-1}{p}}\bigg\{\E\bigg[\int_0^T\|\v(t)\|_{\V_p}^p\d t\bigg]\bigg\}^\frac{1}{p}\\&\nonumber\quad+C(p,\mu) \bigg\{\underbrace{\E\bigg[\int_0^T\|\u_n(t)\|_{\V_p}^p\d t\bigg]}_{I_1}\bigg\}^\frac{p-1}{p}\bigg\{\E\bigg[\int_0^T\|\v(t)\|_{\V_p}^p\d t\bigg]\bigg\}^\frac{1}{p}\\&\nonumber\quad +T^{\frac{2p-2-d}{2p}}\bigg\{\underbrace{\E\bigg[\sup_{t\in[0,T]}\|\u_n(t)\|_{\H}^{\frac{2(2p-d)}{p-1}}\bigg]}_{I_2}\bigg\}^\frac{p-1}{2p}\bigg\{\underbrace{\E\bigg[\bigg(\int_0^T\|\u_n(t)\|_{\V}^{2}\d t\bigg)^\frac{d}{(p-1)}\bigg]}_{I_3}\bigg\}^{\frac{p-1}{2p}}\\&\nonumber\qquad\times\bigg\{\E\bigg[ \int_0^T\|\v(t)\|_{\V_p}^p\d t\bigg]\bigg\}^\frac{1}{p}\\&\nonumber\quad + 
			C(\beta) \bigg\{\underbrace{\E\bigg[\int_0^T\|\u_n(t)\|_{\widetilde{\L}^r}^r\d t\bigg]}_{I_4}\bigg\}^\frac{r-1}{r}\bigg\{\E\bigg[\int_0^T\|\v(t)\|_{\widetilde{\L}^r}^r\d t\bigg]\bigg\}^\frac{1}{r}\\&\quad + 
			\bigg\{\E\bigg[\int_0^T\|\f_n(t)\|_{\V_p'}^{p'}\d t\bigg]\bigg\}^\frac{1}{p'}\bigg\{\E\bigg[\int_0^T\|\v(t)\|_{\V_p}^p\d t\bigg]\bigg\}^\frac{1}{p},
		\end{align}for all $\v\in \mathrm{L}^p(\Omega;\mathrm{L}^p(0,T;\V_p))\cap\mathrm{L}^r(\Omega;\mathrm{L}^r(0,T;\widetilde{\L}^r))$, for $p\geq \frac{d}{2}+1$ and $r\geq 2$, where we have used \eqref{2.2}, Gagliardo-Nirenberg's, H\"older's inequalities, the estimates \eqref{EE002} and \eqref{EE0125} for the permissible $q$ given in Proposition \ref{EE_prop}. That is, for $I_1,I_4$ and $I_3$ we used the estimate \eqref{EE002} with $q=2$, and $q= \frac{2(2p-d)}{p-1}$, respectively, finally for the remaining term $I_3$ we used the estimate \eqref{EE0125} with $q=\frac{2d}{p-1}$.
		
		Using the Hypothesis \ref{hyp} and the energy estimates \eqref{EE002} (with $q=2,4$), we obtain 
		\begin{align}\label{EE22}
			\E\bigg[\sum_{k=1}^n\int_0^T\|\boldsymbol{\sigma}_k^n(t,\u_ n(t))\|_{\H}^2\d t\bigg] \leq K\E\bigg[\int_0^T\big(1+\|\u_n(t)\|_{\H}^2\big)\d t\bigg]
			<\infty,
		\end{align}and 
	\begin{align}\label{EE220}\nonumber
		\E\bigg[\bigg(\sum_{k=1}^n\int_0^T\|\boldsymbol{\sigma}_k^n(t,\u_n(t))\|_\H^2\d t\bigg)^2\bigg]&\nonumber \leq K^2 \E\bigg[\bigg(\int_0^T\big(1+\|\u_n(t)\|_\H^2\big)\d t\bigg)^2\bigg]\\&\nonumber
		\leq 
		K^2T \E\bigg[\int_0^T \big(1+\|\u_n(t)\|_\H^2\big)^2\d t\bigg]\\&\nonumber
		\leq 
		2	K^2T \E\bigg[\int_0^T \big(1+\|\u_n(t)\|_\H^4\big)\d t\bigg]\\&\nonumber
		\leq 
		2	K^2T^2 \bigg\{1+\E\bigg[\sup_{t\in[0,T]}\|\u_n(t)\|_\H^4\bigg]\bigg\}
		\\&\nonumber\leq
			2K^2T^2 \bigg\{1+C(\mu,p,K,T)\Big(\|\u_0\|_\H^4+C\|\f\|_{\mathrm{L}^{p'}(0,T;\V_p')}^{p'}\Big)\bigg\}\\&<\infty.
		\end{align}

	One should note that the right hand side of the inequalities \eqref{EE22} and \eqref{EE220} are independent of $n$ and thus we can extract a sequence $\{\boldsymbol{\sigma}_k^{n_j}(\cdot,\u_{n_j}(\cdot))\}$ of $\{\boldsymbol{\sigma}_k^{n}(\cdot,\u_{n}(\cdot))\}$ such that (for convenience we still denoting by the same symbol $\{\boldsymbol{\sigma}_k^{n}(\cdot,\u_{n}(\cdot))\}$)
		\begin{align}\label{EE23}
			\boldsymbol{\sigma}^n(\cdot,\u_n) \xrightarrow{w} \boldsymbol{\sigma}(\cdot), \quad \text{in} \quad \mathrm{L}^4(\Omega;\mathrm{L}^2(0,T;\mathcal{L}_2(\ell^2;\H))).
		\end{align}

		\vspace{2mm}
		\noindent
		\textbf{Step 3:} \emph{It\^o stochastic differential satisfied by $\u(\cdot)$:} In this step, we obtain the It\^o stochastic differential satisfied by $\u(\cdot)$ using \cite[Theorem 7.5]{chow}. We cannot directly use \cite[Thereom 7.5]{chow} as we have the weak convergence of the nonlinear term in the following space $$\mathrm{L}^{p'}(\Omega;\mathrm{L}^{p'}(0,T;\V_p'))+ \mathrm{L}^{\frac{r}{r-1}}(\Omega;\mathrm{L}^{\frac{r}{r-1}}(0,T;\widetilde{\L}^{\frac{r}{r-1}})).$$ Due to these technical difficulties, we extend our time interval $[0,T]$ to the open interval $(\e,T+\e)$ with $\e>0$, and set the terms in the system \eqref{EE18} equal to zero outside the interval $[0,T]$. Choose $\phi\in\mathrm{H}^1(-\mu,T+\mu)$ such that $\phi(0)=1$. For $\v_m\in\V_p\cap\widetilde{\L}^{r+1}$, we define $\v_m(t)=\phi(t)\v_m$. Applying the It\^o's formula to the process $\big(\u_n(t),\v_m(t)\big)$, we find
		\begin{align}\label{EE24}\nonumber
			\big(\u_n(T),\v_m(T)\big)&=\big(\u_n(0),\v_m\big)+\int_0^T\big(\u_n(t),\dot{\v}_m(t)\big)\d t-\int_0^T\big(\mathscr{H}_n(\u_n(t),\v_m(t))\big)\d t\\&\quad +\sum_{k=1}^n\int_0^T\big(\boldsymbol{\sigma}^n_k(t,\u_n(t))\d\boldsymbol{W}_k(t),\v_m(t)\big),
		\end{align}where $\dot{\v}_m(t)=\frac{\d\phi(t)}{\d t}\v_m$. We can pass the limit $n\to\infty$, using the weak convergences given in \eqref{EE21} and \eqref{EE23}. For instance, we consider the stochastic integral present in the equality \eqref{EE24}, with $m$ fixed. Let $\mathscr{P}_T$ denote the class of predictable process with the values in the space $\mathrm{L}^2(\Omega;\mathrm{L}^2(0,T;\mathcal{L}_2(\ell^2;\H)))$ with the inner product defined by 
		\begin{align*}
			\big(\boldsymbol{\sigma},\boldsymbol{\xi}\big)_{\mathscr{P}_T} =\E\bigg[\sum_{k=1}^\infty\int_0^T\big(\boldsymbol{\sigma}_k(s),\boldsymbol{\zeta}_k(s)\big)\mathbf{e}_k\d s\bigg], \text{ for all } \boldsymbol{\sigma},\boldsymbol{\xi}\in\mathscr{P}_T.
		\end{align*}Define a map $\Gamma:\mathscr{P}_T\to \mathrm{L}^2(\Omega;\mathrm{L}^2(0,T))$ by $\displaystyle \Gamma(\boldsymbol{\sigma})=\sum\limits_{k=1}^\infty\int_0^t\big(\boldsymbol{\sigma}_k(s)\d \boldsymbol{W}_k(s),\v_m(s)\big)$, for all $t\in[0,T]$ and each $m\in\N$. Note that the mapping $\Gamma(\cdot)$ is linear and continuous. The weak convergence from \eqref{EE23} implies that \begin{align*}
			\sum_{k=1}^n\big(\boldsymbol{\sigma}_k^n(t,\u_n)\d\boldsymbol{W}_k(s),\boldsymbol{\xi}\big)_{\mathscr{P}_T}\to  \sum_{k=1}^\infty\big(\boldsymbol{\sigma}_k(t)\d\boldsymbol{W}_k(s),\boldsymbol{\xi}\big)_{\mathscr{P}_T}, \text{ for all } \boldsymbol{\xi}\in\mathcal{P}_T, \text{  as }\ n\to\infty.
		\end{align*} From this, we obtain 
		\begin{align*}
			\Gamma(\boldsymbol{\sigma}_k^n(t,\u_n(t)))=\sum_{k=1}^n\int_0^t\big(\boldsymbol{\sigma}_k^n(t,\u_n(t))\d\boldsymbol{W}_k(s),\v_m(s)\big)\to \sum_{k=1}^\infty\int_0^t\big(\boldsymbol{\sigma}_k(s)\d\boldsymbol{W}_k(s),\v_m(s)\big),
		\end{align*}as $n\to\infty$, for all $t\in[0,T]$ and for each $m$. Passing the limit $n\to\infty$ term wise in \eqref{EE24}, we find 
		\begin{align*}
			\big(\boldsymbol{\xi},\v_m\big)\phi(T)&=\big(\u_0,\v_m\big)+\int_0^T\big(\u(t),\dot{\v}_m\big)\d t-\int_0^T\phi(t)\big\langle\mathscr{H}_0(t),\v_m\big\rangle\d t\\&\quad+ \sum_{k=1}^\infty\int_0^T\phi(t)\big(\boldsymbol{\sigma}_k(t)\d\boldsymbol{W}_k(t),\v_m\big).
		\end{align*}
	
	Let us consider a sequence $\{\phi_k\}_{k\in\N}\in\mathrm{H}^1(-\mu,T+\mu)$ with $\phi_k(0)=1$, for all $k\in\N$ such that $\phi\to\chi_t$ and the derivative of $\phi_k$ with respect to time converges to $\delta_t$ (Dirac $\delta$-distribution), where $\chi_t(s)=1$, if $s\leq t$, and $0$, otherwise. Substituting $\phi_k$ in the place of $\phi$ in \eqref{EE24} and then passing $k\to\infty$, we obtain 
		\begin{align}\label{EE25}
			\big(\u(t),\v_m\big)	&=\big(\u_0,\v_m\big)-\int_0^T\big\langle\mathscr{H}_0(t),\v_m\big\rangle\d t+  \sum_{k=1}^\infty\int_0^T\big(\boldsymbol{\sigma}_k(t)\d\boldsymbol{W}_k(t),\v_m\big),
		\end{align}for all $t\in(0,T)$ with $\big(\u(T),\v_m\big)=\big(\boldsymbol{\xi},\v_m\big)$ and for any $\v_m\in\V_p\cap \widetilde{\L}^r$. Note that $\mathscr{V}\subset \V_p\cap \widetilde{\L}^r\subset\H$ and $\mathscr{V}$ is dense in $\H$. Therefore, $\mathscr{V}\subset \V_p\cap \widetilde{\L}^r$ is dense in $\H$ and the equation \eqref{EE25} holds for any $\v\in \mathscr{V}\subset \V_p\cap \widetilde{\L}^r$. Thus, we have 
		\begin{align}\label{EE26}
			\big(\u(t),\v\big)	&=\big(\u_0,\v\big)-\int_0^T\big\langle\mathscr{H}_0(t),\v\big\rangle\d t+ \sum_{k=1}^\infty\int_0^T\big(\boldsymbol{\sigma}_k(t)\d\boldsymbol{W}_k(t),\v\big),
		\end{align}$\mathbb{P}$-a.s., for all $t\in(0,T)$ with $	\big(\u(T),\v\big)=\big(\boldsymbol{\xi},\v\big)$, for all $\v \in\V_p\cap \widetilde{\L}^r$. Hence, $\u(\cdot)$ satisfies the following differential equation: 
		\begin{equation}\label{EE27}
			\left\{
			\begin{aligned}
				\d \u(t)&= -\mathscr{H}_0(t)\d t+\sum_{k=1}^\infty\boldsymbol{\sigma}_k(t)\d \boldsymbol{W}_k(t), \\
				\u_n(0)&=\u_0,
			\end{aligned}
			\right.
		\end{equation}in $\V_p'+\wi{\L}^{\frac{r}{r-1}}$, for $\u_0\in\H$.
		
		\vspace{2mm}
		\noindent
		\textbf{Step 4:} \emph{Energy equality satisfied by $\u(\cdot)$:} In this step, we establish the energy equality satisfied by $\u(\cdot)$. One should note that such an energy equality is not straight forward due to the final convergence in \eqref{EE21} and we cannot use the infinite-dimensional It\^o's formula available in the literature for semimartingales (see \cite[Theorem 1]{GK1}, \cite[Theorem 6.1]{Me}). We borrow the idea from \cite{CLF} (cf. \cite{MTM8}) to obtain an approximation of $\u(\cdot)$ such that the approximations are bounded and converge in both Sobolev and Lebesgue spaces simultaneously (for periodic domains one can see \cite{KWH}). Now, we approximate $\u(t)$, for each $t\in(0,T)$ and $\PP$-a.s., by using the finite-dimensional space spanned by the first $m$-eigenfunctions of the  operator $\mathcal{L}$ (see Section \ref{ACO}) (for more details one can see \cite{CLF,KWH})
		\begin{align}\label{EE28}
			\boldsymbol{u}^m(t):=\mathrm{P}_{1/m}\u(t)=\sum_{\lambda_j<m^2}e^{-\lambda_j/m}\langle \u(t),\be_j\rangle \be_j.
		\end{align}One should note that, the operator $\PA(\cdot)$ is a self-adjoint operator, but not a projection operator. Here, we use $	\boldsymbol{u}^m$ for approximation of the type \eqref{EE28} and $\u_n$ stands for Galerkin approximations in previous steps. Observe that 
		\begin{align}\label{EE29}
			\|	\boldsymbol{u}^m\|_\H^2 =\|\mathrm{P}_{1/m}\u\|_\H^2=\sum_{\lambda_j<m^2}e^{-2\lambda_j /m}\big|\langle \u,\be_j\rangle\big|^2\leq \sum_{j=1}^\infty\big|\langle \u,\be_j\rangle\big|^2=\|\u\|_\H^2<\infty,
		\end{align}for all $\u\in\H$. Moreover, we have 
		\begin{align}\label{EE30}\nonumber
			\|(\I-\P_{1/m})\u\|_\H^2&=\|\u\|_{\H}^2-2\langle \u,	\boldsymbol{u}^m\rangle +\|\boldsymbol{u}^m\|_\H^2\\&\nonumber = \sum_{j=1}^\infty|\langle \u,\be_j\rangle|^2-2\sum_{\lambda_j<m^2}e^{-\lambda_j/m}|\langle \u,\be_j\rangle|^2+\sum_{\lambda_j<m^2}e^{-2\lambda_j/m}|\langle \u,\be_j\rangle|^2\\&=\sum_{\lambda_j<m^2}\big(1-e^{-\lambda_j/m}\big)^2|\langle \u,\be_j\rangle|^2+\sum_{\lambda_j\geq m^2}|\langle \u,\be_j\rangle|^2,
		\end{align}for all $\u\in\H$. Being a tails end of convergent series $\sum\limits_{j=1}^\infty |\langle \u,\be_j\rangle|^2$, the final term in the right hand side of the above inequality goes to $0$, as $m$ approaches to $\infty$. We can bound the first term in the right hand side of the equality \eqref{EE30} in the following manner:
		\begin{align*}
			\sum_{j=1}^\infty\big(1-e^{-\lambda_j/m}\big)^2|\langle \u,\be_j\rangle|^2\leq 4\sum_{j=1}^\infty |\langle \u,\be_j\rangle |^2=4\|\u\|_\H^2<\infty.
		\end{align*}Using the Dominated Convergence Theorem, we can interchange the summation and limit and hence we obtain 
		\begin{align*}
			\lim_{m\to\infty} \sum_{j=1}^\infty \big(1-e^{-\lambda_j/m})^2|\langle \u,\be_j\rangle |^2=\sum_{j=1}^\infty\lim_{m\to\infty}\big(1-e^{-\lambda_j/m})^2|\langle \u,\be_j\rangle |^2=0.
		\end{align*}Using the above convergence in \eqref{EE30}, we procure
	\begin{align}\label{CT}	
		\big\|(\I-\P_{1/m})\u\big\|_\H\to 0, \text{ as } m\to\infty.
	\end{align} 
Let us discuss the properties of the approximation given in \eqref{EE28}.
	The authors in the article \cite{CLF} established that such an approximation satisfies:
		\begin{equation}\label{cond}
			\left\{ \begin{aligned}
			(1) \ &\boldsymbol{u}^m(t)\to\u(t) \text{ in } \H^1 \text{ with } \|\boldsymbol{u}^m(t)\|_{\H^1}\leq C_1\|\u(t)\|_{\H^1},  \text{ for a.e. } t\in[0,T] \\ &\text{ and  } \mathbb{P}\text{-a.s.;}\\
			 (2)\ &\boldsymbol{u}^m(t)\to\u(t) \text{ in } \L^p(\mathcal{O}) \text{ with } \|\boldsymbol{u}^m(t)\|_{\L^p}\leq C_1\|\u(t)\|_{\L^p}, \text{ for any } p\in(1,\infty), \\ &\text{ for a.e. } t\in[0,T] \text{ and } \mathbb{P}\text{-a.s.;}\\
		 (3) \ &\boldsymbol{u}^m(t) \text{ is divergence free and zero on }\partial \mathcal{O}, \text{ for a.e. }t\in[0,T] \text{ and } \mathbb{P}\text{-a.s.}
		\end{aligned}
			\right.
		\end{equation}In the points (1) and (2) of \eqref{cond}, the constant $C_1$ is an absolute constant. Note that for $2\leq d\leq 4$ and $s>2$, $\D(\mathcal{L})\subset\mathcal{V}^s\subset  \H^2\subset \L^p$, for all $p\in(1,\infty)$ (cf. \cite{CLF}). We know that $\be_j$'s are the eigenvectors of the operator $\mathcal{L}(\cdot)$, therefore $\be_j\in \D(\mathcal{L})\subset \V$ and $\be_j\in\D(\mathcal{L})\subset \widetilde{\L}^r$.	It remains to show the following strong convergence 
	\begin{align}\label{EE39}
		\|	\boldsymbol{u}^m-\u\|_{\mathrm{L}^r(\Omega;\mathrm{L}^r(0,T;\widetilde{\L}^r))}\to 0,\ \text{ as }\ m\to\infty,
	\end{align}which follows from (3) of \eqref{cond}. Since $\u\in \mathrm{L}^r(\Omega;\mathrm{L}^r(0,T;\widetilde{\L}^r)$ and the fact that $\|\u^m(t,\omega)-\u(t,\omega)\|_{\widetilde{\L}^r}\to 0$, for a.e., $t\in[0,T]$ and a.e. $\omega\in\Omega$, one can establish the convergence \eqref{EE39} with the help of Dominated Convergence Theorem. Moreover using the fact that $\u\in \mathrm{L}^p(\Omega;\mathrm{L}^p(0,T;\V_p))$, we have the following convergence: 
\begin{align}\label{EE390}
	\|\boldsymbol{u}^m-\u\|_{\mathrm{L}^p(0,T;\V_p)}\to 0, \text{ as } m\to\infty, \ \PP\text{-a.s.,}
\end{align}
	
	The proof of the energy equality is carried out following the methodology outlined in \cite{NVK}. 
	Fixing $\v=\be_j$ in \eqref{EE26}, multiplying $e^{-\lambda_j/m}\be_j$ and then taking sum over all $j$ such that $\lambda_j<m^2$, we observe that $\boldsymbol{u}^m(\cdot)$ satisfies the following It\^o stochastic differential, for all $t\in [0,\infty)$, $\PP$-a.s,
		\begin{align}\label{EE32}
		\boldsymbol{u}^m(t,x)=\boldsymbol{u}^m_0(x)-\int_0^t\PA\mathscr{H}_{0}(s,x)\d s+\sum_{k=1}^\infty\int_0^t\PA\boldsymbol{\sigma}_{k}(s,x)\d\boldsymbol{W}_k(s),
		\end{align}for all $k\in\N$, so that $\boldsymbol{u}^m(\cdot,\cdot)$ is a smooth function in the second variable.   The equation \eqref{EE32} has a unique solution $\boldsymbol{u}^m(\cdot)$ (see \cite{DaZ}). Using the finite-dimensional It\^o formula, we obtain for each $x\in \mathcal{O}$
		\begin{align}\label{EE33}\nonumber
			|\boldsymbol{u}^m(t,x)|^2&=|\boldsymbol{u}^m_0(x)|^2-2\int_0^t\PA\mathscr{H}_0(s,x)\cdot \boldsymbol{u}^m(s,x)\d s\\&\quad +2\sum_{k=1}^\infty\int_0^t\PA\boldsymbol{\sigma}_{k}(s,x)\d\boldsymbol{W}_k(s)\cdot\boldsymbol{u}^m(s,x)\\&\quad +\sum_{k=1}^\infty\int_0^t\Bigg[\sum_{\lambda_j<m^2}e^{-\lambda_j/m}\langle\boldsymbol{\sigma}_k(s),\be_j\rangle_{\mathbb{U}'\times\mathbb{U}} \be_j(x)\Bigg]^2\d s. 
		\end{align}
	Let us integrate the above expression over $\mathcal{O}$ to find
	\begin{align}\label{EE330}
		\nonumber
		\|\boldsymbol{u}^m(t)\|_\H^2&=\|\boldsymbol{u}^m_0\|_\H^2-2\int_{\mathcal{O}}\int_0^t\PA\mathscr{H}_0(s,x)\cdot \boldsymbol{u}^m(s,x)\d s\d x\\&\nonumber\quad +2\int_\mathcal{O}\sum_{k=1}^\infty\int_0^t\PA\boldsymbol{\sigma}_{k}(s,x)\d\boldsymbol{W}_k(s)\cdot\boldsymbol{u}^m(s,x)\d s\d x\\&\quad +\int_\mathcal{O}\sum_{k=1}^\infty\int_0^t\Bigg[\sum_{\lambda_j<m^2}e^{-\lambda_j/m}\langle \boldsymbol{\sigma}_k(s),\be_j\rangle_{\mathbb{U}'\times\mathbb{U}}\be_j(x)\Bigg]^2\d s\d x, \quad \PP\text{-a.s.}
	\end{align}
Our aim is to use both deterministic and stochastic versions of Fubini's theorem. There are no issue concerning the intergral with respect to $\d s$. In order to use stochastic Fubini's theorem, first we have to make sure that the stochastic integral is well-defined. To be more specific, we need to show that the following integral finite $\PP$-a.s.,
\begin{align}\label{EE331}
\int_0^t\sum_{k=1}^\infty\bigg[\int_\mathcal{O}\PA\boldsymbol{\sigma}_{k}(s,x)\cdot\boldsymbol{u}^m(s,x)\d x\bigg]^2\d s<\infty.
\end{align}The computation below show that 
\begin{align}\label{EE332}
	\E\Bigg[\int_0^t\sum_{k=1}^\infty\bigg(\int_\mathcal{O}\PA\boldsymbol{\sigma}_{k}(s,x)\cdot\boldsymbol{u}^m(s,x)\d x\bigg)^2\d s\Bigg]<\infty,
\end{align}and with the help of \eqref{EE332}, we are in position to employ the stochastic Fubini's theorem (see \cite[Lemma 3.3]{GKCTS} or \cite{JVN}). Moreover, $\boldsymbol{u}^m(t,x)$ is infinitely differential in $x$ for any $(\omega,t)$. Therefore, it is $\mathscr{F}_t\otimes \mathcal{B}(\R^+)$-measurable. Since it is also continuous in $t$ for each $(\omega, x)$, the function $\boldsymbol{u}^m(t,x)$ is $\mathscr{F}_t\otimes \mathcal{B}(\R^+)$-measurable and there is no measurability obstructions in applying Fubini's theorems.
	
	We know that $\mathscr{H}_0\in\mathrm{L}^{p'}(\Omega;\mathrm{L}^{p'}(0,T;\V_p'))+ \mathrm{L}^{\frac{r}{r-1}}(\Omega;\mathrm{L}^{\frac{r}{r-1}}(0,T;\widetilde{\L}^{\frac{r}{r-1}}))$, therefore we can write $\mathscr{H}_0$ as $\mathscr{H}_0=\mathscr{H}_0^1+\mathscr{H}_0^2$, where $\mathscr{H}_0^1=  \mathrm{L}^{p'}(\Omega;\mathrm{L}^{p'}(0,T;\V_p'))$, and $\mathscr{H}_0^2\in \mathrm{L}^{\frac{r}{r-1}}(\Omega;\mathrm{L}^{\frac{r}{r-1}}(0,T;\widetilde{\L}^{\frac{r}{r-1}}))$. Now, consider
	\begin{align}\label{EE333}	\nonumber
	&	\int_0^t\int_\mathcal{O} \PA \mathscr{H}_0(s,x)\cdot \boldsymbol{u}^m(s,x)\d x\d s\\&	\nonumber=	\int_0^t\langle  \mathscr{H}_0(s),\PA \boldsymbol{u}^m(s)\rangle \d s\\&\nonumber
	= \int_0^t\langle  \mathscr{H}_0^1(s),\PA \boldsymbol{u}^m(s)\rangle \d s+\int_0^t\langle  \mathscr{H}_0^2(s),\PA \boldsymbol{u}^m(s)\rangle \d s\\&		\nonumber
	\leq  \int_0^t\| \mathscr{H}_0^1(s)\|_{\V_p'}\|\PA \boldsymbol{u}^m(s)\|_{\V_p}\d s+\int_0^t\|  \mathscr{H}_0^2(s)\|_{\wi{\L}^{\frac{r}{r-1}}}\|\PA \boldsymbol{u}^m(s)\|_{\wi{\L}^r}\d s
	\\&		\nonumber
	\leq  C_1 \int_0^t\| \mathscr{H}_0^1(s)\|_{\V_p'}\| \boldsymbol{u}^m(s)\|_{\V_p}\d s+C_1\int_0^t\|  \mathscr{H}_0^2(s)\|_{\wi{\L}^{\frac{r}{r-1}}}\| \boldsymbol{u}^m(s)\|_{\wi{\L}^r}\d s
	\\&		\nonumber
	\leq  C_1^2 \int_0^t\| \mathscr{H}_0^1(s)\|_{\V_p'}\| \u(s)\|_{\V_p}\d s+C_1^2\int_0^t\|  \mathscr{H}_0^2(s)\|_{\wi{\L}^{\frac{r}{r-1}}}\| \u(s)\|_{\wi{\L}^r}\d s
		\\&		\nonumber
	\leq  C_1^2 \bigg(\int_0^t\| \mathscr{H}_0^1(s)\|_{\V_p'}^{p'}\d s\bigg)^{\frac{1}{p'}}\bigg(\int_0^t\| \u(s)\|_{\V_p}^p\d s\bigg)^{\frac{1}{p}}\\&\quad +C_1^2\bigg(\int_0^t\|  \mathscr{H}_0^2(s)\|_{\wi{\L}^{\frac{r}{r-1}}}^{\frac{r}{r-1}}\d s\bigg)^\frac{r-1}{r}\bigg(\int_0^t\| \u(s)\|_{\wi{\L}^r}^r\d s\bigg)^\frac{1}{r}, \quad \PP\text{-a.s.}
	\end{align} Thus, we arrive at
\begin{align}\label{EE334}
	\E\bigg[	\int_0^t\int_\mathcal{O} \PA \mathscr{H}_0(s,x)\cdot \boldsymbol{u}^m(s,x)\d x\d s\bigg]<\infty.
\end{align}Now, we consider 
	\begin{align}\label{EE335}\nonumber
	&\sum_{k=1}^\infty\int_0^t	\int_\mathcal{O}\Bigg[\sum_{\lambda_j<m^2}e^{-\lambda_j/m}\langle \boldsymbol{\sigma}_k(s),\be_j\rangle_{\mathbb{U}'\times\mathbb{U}}\be_j(x)\Bigg]^2\d x\d s\\&\nonumber=
	\sum_{k=1}^\infty\int_0^t	\Bigg[\sum_{\lambda_j<m^2}e^{-2\lambda_j/m}\big|\langle \boldsymbol{\sigma}_k(s),\be_j\rangle_{\mathbb{U}'\times\mathbb{U}}\big|^2\d s\Bigg]
	\\&\nonumber
	\leq 
	\sum_{k=1}^\infty\int_0^t \sum_{j=1}^\infty	\big|\langle\boldsymbol{\sigma}_k(s),\be_j\rangle_{\mathbb{U}'\times\mathbb{U}}\big|^2\d s 
	\\&=\sum_{k=1}^\infty\int_0^t \|\boldsymbol{\sigma}_k(s)\|_\H^2\d s=\int_0^t\|\boldsymbol{\sigma}(s)\|_{\mathcal{L}_2(\ell^2;\H)}^2\d s,
	\end{align}and using the fact \eqref{3p3}, implies
\begin{align}\label{EE336}
	\E\Bigg[\sum_{k=1}^\infty \int_0^t\int_\mathcal{O}\Bigg[\sum_{\lambda_j<m^2}e^{-\lambda_j/m}\langle \boldsymbol{\sigma}_k(s),\be_j\rangle_{\mathbb{U}'\times\mathbb{U}}\be_j(x)\Bigg]^2\d x\d s\Bigg]<\infty.
\end{align}An application of the Cauchy-Schwarz and H\"older's inequalities yield
\begin{align}\label{EE337}\nonumber
&	\E\Bigg[\int_0^t \sum_{k=1}^\infty\bigg(\int_\mathcal{O} \PA \boldsymbol{\sigma}_k(s,x)\cdot\boldsymbol{u}^m(s,x)\d x \bigg)^2\d s\Bigg]
\\&\nonumber =\E\Bigg[\int_0^t\sum_{k=1}^\infty \Big[\big(\boldsymbol{\sigma}_k(s),\PA\boldsymbol{u}^m(s)\big)\Big]^2\d s \Bigg]
\\&\nonumber \leq \E\Bigg[\int_0^t\sum_{k=1}^\infty \|\boldsymbol{\sigma}_k(s)\|_\H^2\|\PA\boldsymbol{u}^m(s)\|_\H^2\d s \Bigg] \\&\nonumber\leq  C_1^2
\E\Bigg[\sup_{s\in[0,t]}\|\u(s)\|_\H^2\int_0^t\sum_{k=1}^\infty \|\boldsymbol{\sigma}_k(s)\|_\H^2\d s \Bigg] 
\\&\nonumber\leq C_1^2
\bigg(
\E\bigg[\sup_{s\in[0,t]}\|\u(s)\|_\H^4\bigg]\bigg)^\frac{1}{2}\bigg(\E\bigg[\bigg(\int_0^t\sum_{k=1}^\infty \|\boldsymbol{\sigma}_k(s)\|_\H^2\d s \bigg)^2\bigg] \bigg)^\frac{1}{2}
\\&\nonumber\leq  C_1^2
\bigg(
\E\bigg[\sup_{s\in[0,t]}\|\u(s)\|_\H^4\bigg]\bigg)^\frac{1}{2}\bigg(\E\bigg[\bigg(\int_0^t\|\boldsymbol{\sigma}(s)\|_{\mathcal{L}_2(\ell^2;\H)}^2\d s \bigg)^2\bigg] \bigg)^\frac{1}{2}
\\&<\infty,
\end{align}where we have used the fact that $\u\in \mathrm{L}^4(\Omega;\mathrm{L}^\infty(0,T;\H))$ and $\boldsymbol{\sigma}\in \mathrm{L}^4(\Omega;\mathrm{L}^2(0,T;\mathcal{L}_2(\ell^2;\H)))$.

Using deterministic and stochastic versions of Fubini's theorem, we procure  from \eqref{EE330} that for all $t\in [0,\infty)$, $\PP$-a.s.,

		\begin{align}\label{EE338}
		\nonumber
		\|\boldsymbol{u}^m(t)\|_\H^2&=\|\boldsymbol{u}^m_0\|_\H^2-2\int_{\mathcal{O}}\int_0^t\big(\PA\mathscr{H}_0(s,x)\cdot \boldsymbol{u}^m(s,x)\big)\d s\d x\\&\nonumber\quad +2\int_\mathcal{O}\sum_{k=1}^\infty\int_0^t\PA\boldsymbol{\sigma}_{k}(s,x)\d\boldsymbol{W}_k(s)\cdot\boldsymbol{u}^m(s,x)\d s\d x\\&\nonumber\quad +\sum_{k=1}^\infty\int_\mathcal{O}\int_0^t\Bigg[\sum_{\lambda_j<m^2}e^{-\lambda_j/m}\langle \boldsymbol{\sigma}_k(s),\be_j\rangle_{\mathbb{U}'\times\mathbb{U}}\be_j(x)\Bigg]^2\d s\d x
		\\&\nonumber=\|\boldsymbol{u}^m_0\|_\H^2-\int_0^t\big(\PA\mathscr{H}_0(s), \boldsymbol{u}^m(s)\big)\d s +\sum_{k=1}^\infty\int_0^t\big(\PA\boldsymbol{\sigma}_{k}(s)\d\boldsymbol{W}_k(s),\boldsymbol{u}^m(s)\big)\d s\\&\nonumber\quad +\sum_{k=1}^\infty\int_0^t\sum_{\lambda_j<m^2}e^{-2\lambda_j/m}\big|\langle\boldsymbol{\sigma}_k(s),\be_j\rangle_{\mathbb{U}'\times\mathbb{U}}\big|^2\d s
		\\&\nonumber=\|\boldsymbol{u}^m_0\|_\H^2-\int_0^t\big(\mathscr{H}_0(s), \PA\boldsymbol{u}^m(s)\big)\d s +\sum_{k=1}^\infty\int_0^t\big(\boldsymbol{\sigma}_{k}(s)\d\boldsymbol{W}_k(s),\PA\boldsymbol{u}^m(s)\big)\d s\\&\quad +\int_0^t\sum_{k=1}^\infty\sum_{\lambda_j<m^2}e^{-2\lambda_j/m}\big|\langle \boldsymbol{\sigma}_k(s),\be_j\rangle_{\mathbb{U}'\times\mathbb{U}}\big|^2\d s.
	\end{align} Since the convergence  \eqref{CT} holds true for all $t\in[0,T]$, we find 
\begin{align}\label{EE339}
	\lim_{m\to \infty} \|\boldsymbol{u}^m(t)\|_\H^2=\|\u(t)\|_\H^2, \ \text{ and } \ 	\lim_{m\to \infty} \|\boldsymbol{u}^m_0\|_\H^2=\|\u_0\|_\H^2, \quad \PP\text{-a.s.}
\end{align}

For $\mathscr{H}_0\in\mathrm{L}^{p'}(\Omega;\mathrm{L}^{p'}(0,T;\V_p'))+ \mathrm{L}^{\frac{r}{r-1}}(\Omega;\mathrm{L}^{\frac{r}{r-1}}(0,T;\widetilde{\L}^{\frac{r}{r-1}}))$, where $\mathscr{H}_0^1=  \mathrm{L}^{p'}(\Omega;\mathrm{L}^{p'}(0,T;\V_p'))$, and $\mathscr{H}_0^2\in \mathrm{L}^{\frac{r}{r-1}}(\Omega;\mathrm{L}^{\frac{r}{r-1}}(0,T;\widetilde{\L}^{\frac{r}{r-1}}))$. First, we consider
\begin{align}\label{EE340}\nonumber
&	\lim_{m\to\infty} \bigg|\int_0^t \langle \mathscr{H}_0(s), \PA\boldsymbol{u}^m(s)\rangle \d s-\int_0^t\langle \mathscr{H}_0(s),\u(s)\rangle \d s\bigg|\\& \nonumber
	=	\lim_{m\to\infty} \bigg|\int_0^t \langle \mathscr{H}_0^1(s)+\mathscr{H}_0^2(s), \PA(\boldsymbol{u}^m(s)-\u(s))+\boldsymbol{u}^m(s)-\u(s)\rangle \d s\bigg|
	\\& \nonumber
	\leq 	\lim_{m\to\infty} \bigg|\int_0^t \langle \mathscr{H}_0^1(s), \PA(\boldsymbol{u}^m(s)-\u(s))+\boldsymbol{u}^m(s)-\u(s)\rangle \d s\bigg|\\&\nonumber\quad +\lim_{m\to\infty} \bigg|\int_0^t \langle \mathscr{H}_0^2(s), \PA(\boldsymbol{u}^m(s)-\u(s))+\boldsymbol{u}^m-\u(s)\rangle \d s\bigg| \\&\nonumber
	\leq 
	\lim_{m\to\infty}  \|\mathscr{H}_0^1\|_{\mathrm{L}^{p'}(0,T;\V_{p}')}\big(\|\PA(\boldsymbol{u}^m-\u)\|_{\mathrm{L}^p(0,T;\V_p)}+\|\boldsymbol{u}^m-\u\|_{\mathrm{L}^p(0,T;\V_p)}\big)\\&\nonumber\quad +
	\lim_{m\to\infty}  \|\mathscr{H}_0^2\|_{\mathrm{L}^{\frac{r}{r-1}}(0,T;\wi{\L}^{\frac{r}{r-1}})}\big(\|\PA(\boldsymbol{u}^m-\u)\|_{\mathrm{L}^r(0,T;\wi{\L}^r)}+\|\boldsymbol{u}^m-\u\|_{\mathrm{L}^r(0,T;\wi{\L}^r)}\big)
	 \\&\nonumber
	\leq 
	\lim_{m\to\infty}  \|\mathscr{H}_0^1\|_{\mathrm{L}^{p'}(0,T;\V_{p}')}\big(C_1\|\boldsymbol{u}^m-\u\|_{\mathrm{L}^p(0,T;\V_p)}+\|\boldsymbol{u}^m-\u\|_{\mathrm{L}^p(0,T;\V_p)}\big)\\&\nonumber\quad +
	\lim_{m\to\infty}  \|\mathscr{H}_0^2\|_{\mathrm{L}^{\frac{r}{r-1}}(0,T;\wi{\L}^{\frac{r}{r-1}})}\big(C_1\|\boldsymbol{u}^m-\u\|_{\mathrm{L}^r(0,T;\wi{\L}^r)}+\|\boldsymbol{u}^m-\u\|_{\mathrm{L}^r(0,T;\wi{\L}^r)}\big)\\&
	=0, \ \PP \text{-a.s.,}
\end{align}where we have used the convergences given in \eqref{EE39} and \eqref{EE390}. Now, we consider
\begin{align}\label{EE341}\nonumber
&\int_0^t\|\boldsymbol{\sigma}(s)\|_{\mathcal{L}_2(\ell^2;\H)}^2\d s- \int_0^t \sum_{k=1}^\infty \sum_{\lambda_j<m^2}e^{-2\lambda_j/m}\big|\langle\boldsymbol{\sigma}_k(s),\be_j\rangle_{\mathbb{U}'\times\mathbb{U}}\big|^2\d s
\\&\nonumber =
\int_0^t\sum_{k=1}^\infty \sum_{j=1}^\infty \big|\langle\boldsymbol{\sigma}_k(s), \be_j\rangle_{\mathbb{U}'\times\mathbb{U}}\big|^2\d s-\int_0^t \sum_{k=1}^\infty \sum_{\lambda_j<m^2}e^{-2\lambda_j/m}\big|\langle\boldsymbol{\sigma}_k(s),\be_j\rangle_{\mathbb{U}'\times\mathbb{U}}\big|^2\d s
\\&\nonumber =
\int_0^t\sum_{k=1}^\infty \sum_{\lambda_j<m^2}(1-e^{-2\lambda_j/m}) \big|\langle\boldsymbol{\sigma}_k(s), \be_j\rangle_{\mathbb{U}'\times\mathbb{U}}\big|^2\d s+\int_0^t \sum_{k=1}^\infty \sum_{\lambda_j\geq m^2}e^{-2\lambda_j/m}\big|\langle\boldsymbol{\sigma}_k(s),\be_j\rangle_{\mathbb{U}'\times\mathbb{U}}\big|^2\d s
\\& \leq 
\int_0^t\sum_{k=1}^\infty \sum_{j=1}^\infty(1-e^{-2\lambda_j/m}) \big|\langle\boldsymbol{\sigma}_k(s), \be_j\rangle_{\mathbb{U}'\times\mathbb{U}}\big|^2\d s+\int_0^t \sum_{k=1}^\infty \sum_{\lambda_j\geq m^2}e^{-2\lambda_j/m}\big|\langle\boldsymbol{\sigma}_k(s),\be_j\rangle_{\mathbb{U}'\times\mathbb{U}}\big|^2\d s, 
\end{align}$\PP$-a.s. We observe that 
	\begin{align}\label{EE342}\nonumber
	&	\int_0^t\sum_{k=1}^\infty \sum_{j=1}^\infty(1-e^{-2\lambda_j/m}) \big|\langle\boldsymbol{\sigma}_k(s), \be_j\rangle_{\mathbb{U}'\times\mathbb{U}}\big|^2\d s+\int_0^t \sum_{k=1}^\infty \sum_{\lambda_j\geq m^2}e^{-2\lambda_j/m}\big|\langle\boldsymbol{\sigma}_k(s),\be_j\rangle_{\mathbb{U}'\times\mathbb{U}}\big|^2\d s\\& \leq 
	3\int_0^t\sum_{k=1}^\infty \sum_{j=1}^\infty \big|\langle \boldsymbol{\sigma}_k(s),\be_j\rangle_{\mathbb{U}'\times\mathbb{U}}\big|^2\d s=3\int_0^t\|\boldsymbol{\sigma}(s)\|_{\mathcal{L}_2(\ell^2;\H)}^2\d s<\infty, \ \PP\text{-a.s.}
	\end{align}Using the Lebesgue Dominated Convergence Theorem, we procure
\begin{align}\label{EE343}
	\lim_{m\to\infty} \int_0^t \sum_{k=1}^\infty \sum_{\lambda_j\geq n^2}e^{-2\lambda_j/m}\big|\langle\boldsymbol{\sigma}_k(s),\be_j\rangle_{\mathbb{U}'\times\mathbb{U}}\big|^2\d s=\int_0^t\|\boldsymbol{\sigma}(s)\|_{\mathcal{L}_2(\ell^2;\H)}^2\d s, \ \PP\text{-a.s.}
\end{align}Finally, we consider 
\begin{align}\label{EE344}\nonumber
&	\E\bigg[\bigg(\int_0^T\bigg(\sum_{k=1}^\infty \boldsymbol{\sigma}_k(s)\d\boldsymbol{W}_k(s),\PA\boldsymbol{u}^m(s)-\u(s)\bigg)\bigg)^2\bigg]
\\& \nonumber=\E \bigg[\int_0^T\sum_{k=1}^\infty\big[\big( \boldsymbol{\sigma}_k(s),\PA\boldsymbol{u}^m(s)-\u(s)\big)\big]^2\d s\bigg]
\\& \nonumber\leq \E \bigg[\int_0^T\sum_{k=1}^\infty\| \boldsymbol{\sigma}_k(s)\|_{\H}^2\|\PA\boldsymbol{u}^m(s)-\u(s)\|_\H^2\d s\bigg]
\\&\nonumber =\E \bigg[\int_0^T\sum_{k=1}^\infty\| \boldsymbol{\sigma}_k(s)\|_{\H}^2\|\PA\boldsymbol{u}^m(s)-\boldsymbol{u}^m(s)+\boldsymbol{u}^m(s)-\u(s)\|_\H^2\d s\bigg]
\\&\nonumber \leq 2C_1\E \bigg[\int_0^T\sum_{k=1}^\infty\| \boldsymbol{\sigma}_k(s)\|_{\H}^2\|\boldsymbol{u}^m(s)-\u(s)\|_\H^2\d s\bigg]
\\& \nonumber\leq 2C_1\E \bigg[\int_0^T\| \boldsymbol{\sigma}(s)\|_{\mathcal{L}_2(\ell^2;\H)}^2\|\boldsymbol{u}^m(s)-\u(s)\|_\H^2\d s\bigg]
\\&\nonumber \leq 2C_1\E \bigg[\sup_{s\in[0,t]}\|\boldsymbol{u}^m(s)-\u(s)\|_\H^2\int_0^T\| \boldsymbol{\sigma}(s)\|_{\mathcal{L}_2(\ell^2;\H)}^2\d s\bigg]\\&
\leq 
2C_1 \bigg(\E\bigg[\sup_{s\in[0,t]}\|\boldsymbol{u}^m(s)-\u(s)\|_\H^4\bigg]\bigg)^{\frac{1}{2}}\bigg(\E\bigg[\bigg(\int_0^t\|\boldsymbol{\sigma}(s)\|_{\mathcal{L}_2(\ell^2;\H)}^2\d s\bigg)^2\bigg]\bigg)^\frac{1}{2}.
\end{align}From \eqref{EE23}, we know that the second term $\displaystyle\bigg(\E\bigg[\bigg(\int_0^t\|\boldsymbol{\sigma}(s)\|_{\mathcal{L}_2(\ell^2;\H)}^2\d s\bigg)^2\bigg]\bigg)^\frac{1}{2}$ is finite. 
Now, our goal is to prove that $\displaystyle \E\bigg[\sup_{s\in[0,t]}\|\boldsymbol{u}^m(s)-\u(s)\|_\H^4\bigg]\to0$, as $m\to\infty$. In order to establish this, we employ the Lebesgue-Vitali theorem (Theorem \ref{LVT}). 

Using the convergence  \eqref{CT},   we deduce 
\begin{align*}
	\sup_{s\in[0,t]}\|\boldsymbol{u}^m(s)-\u(s)\|_\H^4\to 0,\ \text{ as }\ m\to\infty, \ \mathbb{P}\text{-a.s.}
\end{align*} From the energy estimate \eqref{EE002} (with $q=4+\xi$ in Proposition \ref{EE_prop}), we know that for $\u_0\in\H$, for some $\xi>0$
\begin{align}\label{EE345}\nonumber
	\E\bigg[\sup_{s\in[0,t]}\|\boldsymbol{u}^m(s)-\u(s)\|_\H^{4+\xi}\bigg]&\leq C \E\bigg[\sup_{s\in[0,t]}\Big\{\|\boldsymbol{u}^m(s)\|_\H^{4+\xi}+\|\u(s)\|_\H^{4+\xi}\Big\}\bigg]\\& \leq C\E\bigg[\sup_{s\in[0,t]}\|\u(s)\|_\H^{4+\xi}\bigg]<\infty.
\end{align} Using Remark \ref{rem4.5}, it is straightforward to see that the uniform integrability condition holds. Thus, we employ the Lebesgue-Vitali theorem (Theorem \ref{LVT}) to ensure the following convergence:
\begin{align*}
\E\bigg[\sup_{s\in[0,t]}\|\boldsymbol{u}^m(s)-\u(s)\|_\H^4\bigg]\to0, \text{ as } m\to\infty,
\end{align*}and hence from \eqref{EE344}, we infer
\begin{align}\label{EE346}
	\E\bigg[\bigg(\int_0^T\bigg(\sum_{k=1}^\infty \boldsymbol{\sigma}_k(s)\d\boldsymbol{W}_k(s),\PA\boldsymbol{u}^m(s)-\u(s)\bigg)\bigg)^2\bigg] \to 0,\ \text{ as }\ m\to\infty,
\end{align}which implies 
\begin{align}\label{EE347}
\lim_{m\to \infty}	\int_0^T\bigg(\sum_{k=1}^\infty\PA \boldsymbol{\sigma}_k(s)\d\boldsymbol{W}_k(s),\boldsymbol{u}^m(s)\bigg)=\int_0^t\bigg(\sum_{k=1}^\infty \boldsymbol{\sigma}_k(s)\d\boldsymbol{W}_k(s),\u(s)\bigg)\bigg], \ \PP\text{-a.s.}
\end{align}Using the convergences \eqref{EE339}, \eqref{EE340}, \eqref{EE343} and \eqref{EE347} in \eqref{EE338}, we procure for all $t\in[0,T]$, $\PP$-a.s.,
 \begin{align}\label{EE348}\nonumber
 	\|\u(t)\|_\H^2&=\|\u_0\|_\H^2-2\int_0^t\langle \mathscr{H}_0(s),\u(s)\rangle \d s+2\sum_{k=1}^\infty \int_0^t \big(\boldsymbol{\sigma}_k(s)\d\boldsymbol{W}_k(s),\u(s)\big)\\&\quad +\sum_{k=1}^\infty\int_0^t\|\boldsymbol{\sigma}_k(s)\|_{\H}^2\d s.
 \end{align}Taking expectation on the both sides of above equality and using the fact that the penultimate term appearing in the right hand side of the above equality is a martingale, we observe
\begin{align}\label{EE349}
	\E\big[	\|\u(t)\|_\H^2\big]&=\|\u_0\|_\H^2-2\E\bigg[\int_0^t\langle \mathscr{H}_0(s),\u(s)\rangle \d s\bigg] +\E\bigg[\sum_{k=1}^\infty\int_0^t\|\boldsymbol{\sigma}_k(s)\|_{\H}^2\d s\bigg].
\end{align}
Finally,	applying the infinite-dimensional It\^o formula to the process $e^{-\rho(\cdot)}\|\u(\cdot)\|_\H^2$, we find for all $t\in[0,T]$
\begin{align}\label{EE68}\nonumber
	\E\big[e^{-\rho(t)}	\|\u(t)\|_\H^2\big]&=\|\u_0\|_\H^2-\E\bigg[\int_0^te^{-\rho(s)}\langle 2\mathscr{H}_0(s)+\rho'(s)\u(s),\u(s)\rangle\d s\bigg]\\&\quad +\E\bigg[\sum_{k=1}^\infty \int_0^te^{-\rho(s)}\|\boldsymbol{\sigma}_k(s)\|_{\H}^2\d s\bigg].
\end{align}  We know that the initial data $\u_n(0)$ converges to $\u_0$ strongly in $\H$, that is,
\begin{align}\label{EE068}
	\lim_{n\to\infty}\|\u_n(0)-\u_0\|_\H^2=0.
\end{align}

		\vspace{2mm}
		\noindent
		\textbf{Step 5:} \emph{Minty-Browder technique and the existence of global strong solution:} In the part, our aim is to identify the limits, that is, we need to prove to the following:
		\begin{align*}
			\mathscr{H}(\u(\cdot))=\mathscr{H}_0(\cdot),\ \text{ and } \ \boldsymbol{\sigma}(\cdot,\u(\cdot))=\boldsymbol{\sigma}(\cdot). 
		\end{align*}For any $\v\in\mathrm{L}^2(\Omega;\mathrm{L}^\infty(0,T;\H_m))$ with $m<n$, we define 
		\begin{align}\label{EE69}
			\rho(t)=\big(2\widetilde{\eta}+L\big)\int_0^t\|\v(s)\|_{\V_p}^{\frac{2p}{2p-d}}\d s, \text{ so that } \rho'(t)=\big(2\widetilde{\eta}+L\big)\|\v(t)\|_{\V_p}^{\frac{2p}{2p-d}}, \text{ a.e.}
		\end{align}From the local monotonicity result \eqref{EE17} established in Lemma \ref{EU_lem1}, we deduce 
		\begin{align}\label{EE70}\nonumber
			&	\E\bigg[\int_0^T e^{-\rho(t)}\Big(2\big\langle \mathscr{H}(\v(t))-\mathscr{H}_n(\u_n(t)),\v(t)-\u_n(t)\big\rangle +\rho'(t)\big(\v(t)-\u_n(t),\v(t)-\u_n(t)\big)\Big)\d t\bigg]\\& \geq 
			\E\bigg[\sum_{k=1}^n\int_0^T e^{-\rho(t)}\|\boldsymbol{\sigma}_k(t,\v(t))-\boldsymbol{\sigma}_k(t,\u_n(t))\|_{\H}^2\d t\bigg],
		\end{align}for $p\geq \frac{d}{2}+1$, $r\geq 2$. Rearranging the terms in \eqref{EE70} and using the energy equality \eqref{EE20}, we obtain
		\begin{align}\label{EE71}\nonumber
			&	\E\bigg[\int_0^T e^{-\rho(t)}\big\langle 2\mathscr{H}(\v(t))+\rho'(t)\v(t),\v(t)-\u_n(t)\big\rangle \d t\bigg]\\&\nonumber \quad- 
			\E\bigg[\sum_{k=1}^n\int_0^Te^{-\rho(t)} \|\boldsymbol{\sigma}_k(t,\v(t))\|_{\H}^2\d t\bigg]+2\E\bigg[\sum_{k=1}^n\int_0^Te^{-\rho(t)}\big(\boldsymbol{\sigma}_k(t,\v(t)),\boldsymbol{\sigma}_k(t,\u_n(t))\big)\d t\bigg]\\&\nonumber\geq 
			\E\bigg[\int_0^Te^{-\rho(t)}\big\langle 2\mathscr{H}_n(\u_n(t))+\rho'(t)\u_n(t),\v(t)\big\rangle \d t\bigg]\\&\nonumber\quad -\E\bigg[\int_0^Te^{-\rho(t)}\big\langle 2\mathscr{H}_n(\u_n(t))+\rho'(t)\u_n(t),\u_n(t)\big\rangle \d t\bigg]+
			\E\bigg[\sum_{k=1}^n\int_0^Te^{-\rho(t)} \|\boldsymbol{\sigma}_k(t,\u_n(t))\|_{\H}^2\d t\bigg]\\&= 
			\E\bigg[\int_0^Te^{-\rho(t)}\big\langle 2\mathscr{H}_n(\u_n(t))+\rho'(t)\u_n(t),\v(t)\big\rangle \d t\bigg]+\E\Big[e^{-\rho(t)}\|\u_n(t)\|_\H^2-\|\u_n(0)\|_\H^2\Big]
			.
		\end{align}Let us now discuss the convergence of the terms related to the noise coefficient. From the weak convergence given in \eqref{EE23} and the Lebesgue Dominated Convergence Theorem, we get
		\begin{align}\label{EE72}\nonumber
			&	\E\bigg[\sum_{k=1}^n\int_0^Te^{-\rho(t)}\Big(2\big(\boldsymbol{\sigma}_k(t,\v(t)),\boldsymbol{\sigma}_k(t,\u_n(t))\big)-\|\boldsymbol{\sigma}_k(t,\v(t))\|_{\H}^2\Big)\d t\bigg]\\&\to
			\E\bigg[\sum_{k=1}^\infty\int_0^Te^{-\rho(t)}\Big(2\big(\boldsymbol{\sigma}_k(t,\v(t)),\boldsymbol{\sigma}_k(t)\big)-\|\boldsymbol{\sigma}_k(t,\v(t))\|_{\H}^2\Big)\d t\bigg], \ \text{ as }\ n\to\infty.
		\end{align}Taking liminf on both sides of \eqref{EE71} and using \eqref{EE72} in the resultant, we find 
		\begin{align}\label{EE73}\nonumber
			&	\E\bigg[\int_0^T e^{-\rho(t)}\big\langle 2\mathscr{H}(\v(t))+\rho'(t)\v(t),\v(t)-\u(t)\big\rangle \d t\bigg]\\&\nonumber\nonumber \qquad- 
			\E\bigg[\sum_{k=1}^\infty\int_0^Te^{-\rho(t)} \|\boldsymbol{\sigma}_k(t,\v(t))\|_{\H}^2\d t\bigg]+2\E\bigg[\sum_{k=1}^\infty\int_0^Te^{-\rho(t)}\big(\boldsymbol{\sigma}_k(t,\v(t)),\boldsymbol{\sigma}_k(t)\big)\d t\bigg]\\&\geq 
			\E\bigg[\int_0^Te^{-\rho(t)}\big\langle 2\mathscr{H}_0(t)+\rho'(t)\u(t),\v(t)\big\rangle \d t\bigg]+\lim\inf_{n\to\infty}\E\Big[e^{-\rho(T)}\|\u_n(T)\|_\H^2-\|\u_n(0)\|_\H^2\Big]. 
		\end{align}Using the lower semicontinuity property of the $\H$-norm and the strong convergence given in \eqref{EE068}, the second term on the right hand side of the above inequality satisfies:
		\begin{align}\label{EE74}
			\liminf_{n\to\infty}\E\Big[e^{-\rho(T)}\|\u_n(T)\|_\H^2-\|\u_n(0)\|_\H^2\Big]\geq \E\Big[e^{-\rho(T)}\|\u(T)\|_\H^2-\|\u(0)\|_\H^2	\Big].
		\end{align}Using the energy equality \eqref{EE68} and \eqref{EE74} in \eqref{EE73}, we obtain 
		\begin{align}\label{EE75}\nonumber
			&	\E\bigg[\int_0^T e^{-\rho(t)}\big\langle 2\mathscr{H}(\v(t))+\rho'(t)\v(t),\v(t)-\u(t)\big\rangle \d t\bigg]\\&\nonumber\nonumber \geq 
			\E\bigg[\sum_{k=1}^\infty\int_0^Te^{-\rho(t)} \|\boldsymbol{\sigma}_k(t,\v(t))\|_{\H}^2\bigg]-2\E\bigg[\sum_{k=1}^\infty\int_0^Te^{-\rho(t)}\big(\boldsymbol{\sigma}_k(t,\v(t)),\boldsymbol{\sigma}_k(t)\big)\d t\bigg]\\&\quad + 
			\E\bigg[\sum_{k=1}^\infty\int_0^Te^{-\rho(t)} \|\boldsymbol{\sigma}_k(t)\|_{\H}^2 \d t\bigg]+	\E\bigg[\int_0^Te^{-\rho(t)}\big\langle 2\mathscr{H}_0(t)+\rho'(t)\u(t),\v(t)-\u(t)\big\rangle \d t\bigg]. 
		\end{align}Rearranging the terms in  \eqref{EE75}, we find
		\begin{align}\label{EE76}\nonumber
			&	\E\bigg[\int_0^T e^{-\rho(t)}\big\langle 2\big(\mathscr{H}(\v(t))-\mathscr{H}_0(t)\big)+\rho'(t)\big(\v(t)-\u(t)\big),\v(t)-\u(t)\big\rangle \d t\bigg] \\&\geq 
			\E\bigg[\sum_{k=1}^\infty\int_0^T e^{-\rho(t)}\|\boldsymbol{\sigma}_k(t,\v(t))-\boldsymbol{\sigma}_k(t)\|_{\H}^2\d t\bigg]\geq 0. 
		\end{align}The above estimate is valid for all $\v\in\mathrm{L}^2(\Omega;\mathrm{L}^\infty(0,T;\H_m))$ and for each $m\in\N$, since the above estimate is independent of $m$ and $n$. Using a density argument, the estimate \eqref{EE76} remains valid for any 
		\begin{align*}
			\v\in \mathrm{L}^2(\Omega;\mathrm{L}^\infty(0,T;\H))\cap \mathrm{L}^p(\Omega;\mathrm{L}^p(0,T;\V_p))\cap \mathrm{L}^{r}(\Omega;\mathrm{L}^r(0,T;\widetilde{\L}^r))=:\mathscr{G}.
		\end{align*}Indeed for any $\v\in\mathscr{G}$, there exists a strongly convergent subsequence $\v_m\in\mathscr{G}$, which satisfies  \eqref{EE76}. Considering $\v(\cdot)=\u(\cdot)$ in \eqref{EE76} immediately gives $\boldsymbol{\sigma}_k(\cdot,\v(\cdot))=\boldsymbol{\sigma}_k(\cdot)$, for all $k\in\N$. Next, we choose $\v(\cdot)=\u(\cdot)+\lambda\w(\cdot), \ \lambda>0$, where $\w\in \mathrm{L}^\infty(\Omega;\mathrm{L}^\infty(0,T;\H))$ and substituting it in \eqref{EE76}, we get 
		\begin{align}\label{EE77}
			\E\bigg[\int_0^T e^{-\rho(t)}\big\langle 2\mathscr{H}(\u(t)+\lambda\w(t))-2\mathscr{H}_0(t)+\rho'(t)\lambda\w(t),\lambda\w(t)\big\rangle \d t \bigg]\geq 0.
		\end{align}Dividing the above inequality by $\lambda$ and using the hemicontinuity property $\mathscr{H}(\cdot)$, and letting $\lambda\to0$, we get 
		\begin{align}\label{EE78}
			\E\bigg[\int_0^T e^{-\rho(t)}\big\langle 2\mathscr{H}(\u(t))-2\mathscr{H}_0(t),\w(t)\big\rangle \d t \bigg]\geq 0.
		\end{align}The final term in \eqref{EE77} goes to $0$ as $\lambda \to\infty$, since 
		\begin{align*}
			\E\bigg[\int_0^Te^{-\rho(t)}\rho'(t)\big(\w(t),\w(t)\big)\d t\bigg]&=\big(2\widetilde{\eta}+L\big)\E\bigg[\int_0^Te^{-\rho(t)}\|\v(t)\|_{\V_p}^{\frac{2p}{2p-d}}\|\w(t)\|_\H^2\d t\bigg] \\&\leq 
			\big(2\widetilde{\eta}+L\big)\esssup_{\omega\in\Omega}\sup_{t\in[0,T]}\|\w(t)\|_\H^2\E\bigg[\int_0^Te^{-\rho(t)}\|\v(t)\|_{\V_p}^{\frac{2p}{2p-d}}\d t\bigg] \\&<\infty,
		\end{align*}for $p\geq \frac{d}{2}+1$. Thus, from \eqref{EE78}, we obtain $\mathscr{H}(\u(t))=\mathscr{H}_0(t)$ and hence $\u(\cdot)$ is a strong solution to the system \eqref{32} and $\u\in\mathscr{G}$.  It is clear from \eqref{EE68} that the $\mathscr{F}_t$-adapted paths of $\u(\cdot)$ are continuous with trajectories in $\C([0,T];\H),\ \PP$-a.s.

		\vspace{2mm}
		\noindent
		\textbf{Step 5:} \emph{Uniqueness:}
		Let us prove the uniqueness of the strong solution to the system \eqref{32}. Let $\u_1(\cdot)$ and $\u_2(\cdot)$ be ant two strong solutions to the system \eqref{32}. For $N\in\N$, we define
		\begin{align*}
			\tau_N^1=\inf_{t\in[0,T]}\big\{t:\|\u_1(t)\|_\H\geq N \big\}, \ \tau_N^2=\inf_{t\in[0,T]}\big\{t:\|\u_2(t)\|_\H\geq N \big\}, \text{ and } \tau_N:=\tau_N^1\wedge \tau_N^2.
		\end{align*}One can prove that $\tau_N\to T$, as $N\to\infty, \ \PP$-a.s. Set $\w(\cdot)=\u_1(\cdot)-\u_2(\cdot)$ and $\boldsymbol{\widetilde{\sigma}}_k(\cdot)=\boldsymbol{\sigma}_k(\cdot,\u_1(\cdot))-\boldsymbol{\sigma}_k(\cdot,\u_2(\cdot))$, for $k\in\N$. Then, $\w(\cdot)$ satisfies the following system:
		\begin{equation}\label{EE79}
			\left\{
			\begin{aligned}
				\d\w(t)&= -\big[\mu\big(\mathcal{A}(\u_1(t))-\mathcal{A}(\u_2(t))\big)+\B(\u_1(t))-\B(\u_2(t))+\beta\big(\mathcal{C}(\u_1(t))-\mathcal{C}(\u_2(t))\big)\big]\d t\\&\quad + \sum_{k=1}^\infty\boldsymbol{\widetilde{\sigma}}_k(t)\d\boldsymbol{W}_k(t),
				\\ 
				\w(0)&=\w_0.
			\end{aligned}
			\right.
		\end{equation}An application of the infinite-dimensional It\^o's formula to the process $e^{-\rho(\cdot)}\|\w(\cdot)\|_\H^2$, where 
		\begin{align*}
			\rho(t)=\big(2\widetilde{\eta}+L\big)\int_0^t\|\u_2(s)\|_{\V_p}^{\frac{2p}{2p-d}}\d s, \text{ so that } \rho'(t)=\big(2\widetilde{\eta}+L\big)\|\u_2(t)\|_{\V_p}^{\frac{2p}{2p-d}}, \text{ a.e.,}
		\end{align*}yields for all $t\in[0,T]$, $\PP$-a.s.,
		\begin{align}\label{EE80}\nonumber
			e^{-\rho(t\wedge \tau_N)}\|\w(t\wedge \tau_N)\|_\H^2&=\|\w_0\|_\H^2-2\int_0^{t\wedge \tau_N}e^{-\rho(s)}\big\langle \mu\big(\mathcal{A}(\u_1(t))-\mathcal{A}(\u_2(t))\big)\\&\nonumber\qquad +\B(\u_1(t))-\B(\u_2(t))+\beta\big(\mathcal{C}(\u_1(t))-\mathcal{C}(\u_2(t))\big),\w(s)\big\rangle\d s\\&\nonumber\quad -\int_0^{t\wedge \tau_N}e^{-\rho(s)}\rho'(s)\|\w(s)\|_\H^2\d s+\sum_{k=1}^\infty\int_0^{t\wedge \tau_N}e^{-\rho(s)}\|\boldsymbol{\widetilde{\sigma}}_k(s)\|_{\H}^2\d s\\&\quad + 2\sum_{k=1}^\infty\int_0^{t\wedge \tau_N}e^{-\rho(s)}\big(\boldsymbol{\widetilde{\sigma}}_k(s)\d\boldsymbol{W}_k(s),\w(s)\big).
		\end{align}Using \eqref{mon1} in \eqref{EE80}, we obtain for all $t\in[0,T]$, $\PP$-a.s.,
		\begin{align}\label{EE81}
			\nonumber
			e^{-\rho(t\wedge \tau_N)}\|\w(t\wedge \tau_N)\|_\H^2&=\|\w_0\|_\H^2+\sum_{k=1}^\infty\int_0^{t\wedge \tau_N}e^{-\rho(s)}\|\boldsymbol{\widetilde{\sigma}}_k(s)\|_{\H}^2\d s\\&\quad + 2\sum_{k=1}^\infty\int_0^{t\wedge \tau_N}e^{-\rho(s)}\big(\boldsymbol{\widetilde{\sigma}}_k(s)\d\boldsymbol{W}_k(s),\w(s)\big).
		\end{align}Taking expectation \eqref{EE81}, using the fact that the final term in the right hand side of the inequality \eqref{EE81} is a martingale and using the Hypothesis  \ref{hyp}, we arrive at 
		\begin{align}\label{EE82}
			\E\Big[e^{-\rho(t\wedge \tau_N)}\|\w(t\wedge \tau_N)\|_\H^2\Big]&=\|\w_0\|_\H^2+L \E\bigg[\int_0^{t\wedge \tau_N}e^{-\rho(s)}\|\w(s)\|_{\H}^2\d s\bigg].
		\end{align}An application of Gronwall's inequality in \eqref{EE82} yields 
		\begin{align}\label{EE83}
			\E\Big[e^{-\rho(t\wedge \tau_N)}\|\w(t\wedge \tau_N)\|_\H^2\Big]\leq \|\w_0\|_\H^2e^{LT}. 
		\end{align}Thus, the initial data $\u_1(0)=\u_2(0)=\u_0$, leads to $\w(t\wedge \tau_N)=0,\ \PP$-a.s. Using the fact that $\tau_N\to T, \ \mathbb{P}$-a.s., implies $\w(t)=0$ and hence $\u_1(t)=\u_2(t), \ \PP$-a.s., for all $t\in[0,T]$, and hence the pathwise uniqueness follows.
	\end{proof}
	
	\begin{remark}
		It is natural to ask whether the operator $\mathrm{P}_{1/m}$ defined in \eqref{EE28} could be used in place of $\Pi_n$ from \eqref{2.2.4}, which appears in the existence proof, when establishing the energy equality. However, since $\mathrm{P}_{1/m}$ is self-adjoint but not a projection, $\Pi_n$ must be retained for the existence results.
	\end{remark}
	
	\section{Small time asymptotics} 
	In this section, we discuss the small time asymptotics of the  Ladyzhenskaya-Smagorinsky equations  with damping by studying the effect of small, highly nonlinear, unbounded drifts (small time large deviation principle (LDP)).
	We start this part with the basic definitions. 
	
	Let us recall some basics of LDP. 	Let $(\mathcal{E},\rho)$ be a complete separable metric space. We are given a family of probability measures $\{\mu_{\e}\}_{\e> 0}$ on $\mathcal{E}$ and a lower semicontinuous function $\mathrm{R} : \mathcal{E}\to[0,\infty]$, not identically equal to $+\infty$ and such that its level sets,
	\begin{align*}
		\mathrm{K}_M:=\left\{\psi\in\mathcal{E}: \mathrm{R}(\psi)\leq M\right\}, \ M>0,
	\end{align*}
	are compact. The family $\{\mu_{\e}\}_{\e> 0}$ is said to satisfy the LDP  or to have the large deviation property with respect to the rate function $\mathrm{R}$ if
	\begin{enumerate}
		\item [(i)] for all closed sets $\F\subset\mathcal{E}$,  we have
		\begin{align*}
			\limsup_{\e\to 0}\e\log\mu_{\e}(\F)\leq -\inf_{\psi\in\F}\mathrm{R}(\psi);
		\end{align*}
		\item [(ii)] for all open sets $\mathrm{O}\subset\mathcal{E}$,  we have
		\begin{align*}
			\liminf_{\e\to 0}\e\log\mu_{\e}(\mathrm{O})\leq -\sup_{\psi\in\mathrm{O}}\mathrm{R}(\psi).
		\end{align*}
	\end{enumerate}
	\begin{definition}[Rate function]
		Let $\mathrm{R}(\cdot)$ be a rate function on the space $\mathcal{E}$. A family $\{\u^\e\}_{\e>0}$ of $\mathcal{E}$-valued random variables is said to satisfy the Laplace principle on $\mathcal{E}$ with the rate function $\mathrm{R}(\cdot)$, for each $f\in\C_b(\mathcal{E})$, 
		\begin{align*}
			\lim_{\e\to0} \e\log\E\bigg[\exp\bigg(\frac{-f(\u^\e)}{\e}\bigg)\bigg]=-\inf_{\psi\in\mathcal{E}}\big\{f(\u^\e)+\mathrm{R}(\psi)\}.
		\end{align*}
	\end{definition}
	
	From previous sections, we know that under Hypothesis \ref{hyp}, the system \eqref{32} has a unique strong solution with paths in $\C([0,T];\H)\cap\mathrm{L}^p(0,T;\V_p)\cap\mathrm{L}^{r}(0,T;\wi\L^{r})$, $\mathbb{P}$-a.s., that is, $\u(\cdot)$ satisfies the following equation for all $t\in[0,T]$, $\mathbb{P}$-a.s., 
	\begin{align}
		\u(t)&=\u_0-\int_0^t\left[\mu \mathcal{A}(\u(s))+\B(\u(s))+\beta\mathcal{C}(\u(s))\right]\d s+\int_0^t\f(s)\d s\nonumber\\&\quad+\sum_{k=1}^{\infty}\int_0^t\boldsymbol{\sigma}_k(s,\u(s))\d\boldsymbol{W}_k(s).
	\end{align}
	Let us take $\e>0$. Then  by the scaling property of the Brownian motion, it can be  easily verified that $\u(\e t) $ coincides in law with the solution of the following equation for all $t\in[0,T]$, $\mathbb{P}$-a.s., 
	\begin{align}\label{37}
		\u^{\e}(t)&=\u_0-\e\int_0^t\left[\mu \mathcal{A}(\u^{\e}(s))+\B(\u^{\e}(s))+\beta\mathcal{C}(\u^{\e}(s))-\f(\e s)\right]\d s\nonumber\\&\quad+\sqrt{\e}\sum_{k=1}^{\infty}\int_0^t\boldsymbol{\sigma}_k(\e s,\u^{\e}(s))\d\boldsymbol{W}_k(s).
	\end{align}
	In order to prove small time LDP, we  need the following additional assumption on the noise coefficient: 
	\begin{hypothesis}\label{hyp1}
		The noise coefficient $\boldsymbol{\sigma}_k(\cdot,\cdot),k\geq 1$ satisfies: 
		\begin{itemize}
			\item[(H.3)]  (Additional growth condition).	There exists positive	constants $\widehat{K},\wi K$ and $\overline{K}$ such that for all $t\in[0,T]$ and $\u\in\H$,
			\begin{align}
				\sum_{k=1}^{\infty}\|\boldsymbol{\sigma}_k(t, \u)\|^{2}_{\V} 	&\leq \widehat{K}\left(1 +\|\u\|_{\V}^{2}\right), \label{4p8}\\ \sum_{k=1}^{\infty}\|\boldsymbol{\sigma}_k(t, \u)\|^{2}_{\V_p} 	&\leq \widehat{K}\left(1 +\|\u\|_{\V_p}^{2}\right),\ \text{ for }\ p>2, \label{4p9}\\  \sum_{k=1}^{\infty}\|\boldsymbol{\sigma}_k(t, \u)\|^{2}_{\wi\L^r} 	&\leq \overline{K}\left(1 +\|\u\|_{\wi\L^r}^{2}\right),\ \text{ for }\ r\geq 2. \label{4p10}
			\end{align}
			\item[(H.4)] (Continuity at $t=0$). For each $k\in\N$, the operator $\boldsymbol{\sigma}_k(\cdot,\u)$ is continuous at $0$ for any \textbf{$\u\in\H$}, that is, for any sequence  $\{t_n\}_{n\in\N}$ such that $t_n\to0$, as $n\to\infty$ and $\u\in\H$,
			\begin{align}\label{4p010}
			\|\boldsymbol{\sigma}_k(t_n,\u)-\boldsymbol{\sigma}_k(0,\u)\|_\H \to 0, \text{ for each } k\in\N.
			\end{align}Furthermore, as $\boldsymbol{\sigma}_k(\cdot,\u)$ is continuous at $0$, for all $\u\in\H$, it is measurable at $0$, for each $k\in\N$. 
		\end{itemize}
	\end{hypothesis}	
	Note that the condition \eqref{4p8} implies that for each $\u\in\V$, the linear map $\boldsymbol{\sigma}(\cdot,\u):=\{\boldsymbol{\sigma}_k(\cdot,\u)\}_{k\in\mathbb{N}}\break :\ell^2\to\V$ defined by \eqref{3p2} is in $\mathcal{L}_2(\ell^2;\V)$. Note that  every Hilbert space, $\L^p$ spaces with $p \geq 2$ and Sobolev spaces $\mathbb{W}^{s,p}_0$
	with $p \geq 2$ and $s \geq 1$ are $2$-smooth Banach spaces. For a Hilbert space $\mathbb{U}$ and a $2$-smooth Banach space $\mathbb{V}$, we denote by $\gamma(\mathbb{U};\mathbb{V})$ for the space of all $\gamma$-radonifying operators from $\mathbb{U}$ to $\mathbb{V}$. We remember that an operator $\mathrm{R}\in\gamma(\mathbb{U};\V)$ if the series $$\sum_{k=1}^{\infty}\gamma_k\mathrm{R}(e_k)$$ converges in $\mathrm{L}^2(\wi\Omega;\mathbb{V})$ for any sequence $\{\gamma_k\}_{k\geq 0}$ of independent Gaussian real-valued random variables on a probability space $(\wi{\Omega},\wi{\mathscr{F}},\wi{\mathbb{P}})$ and any orthonormal basis $\{e_k\}_{k\geq 0}$ of  $\mathbb{U}$. Then, the space $\gamma(\mathbb{U};\mathbb{V})$ is endowed with the norm 
	\begin{align*}
		\|\mathrm{R}\|_{\gamma(\mathbb{U};\mathbb{V})}=\left(\wi{\E}\left\|\sum_{k=1}^{\infty}\gamma_k\mathrm{R}(e_k)\right\|_{\V}^2\right)^{\frac{1}{2}},
	\end{align*}
	which does not depend $\{\gamma_k\}_{k\geq 0}$  and $\{e_k\}_{k\geq 0},$  and is a Banach space. If $\mathbb{V}$ is a separable Hilbert space,  then $\gamma(\mathbb{U};\mathbb{V})$) consists of all Hilbert-Schmidt operators of mapping $\mathbb{U}$ into $\mathbb{V}$. The condition \eqref{4p9} implies that for every $\u\in\V_p$, for some $p\geq 2$, the linear map $\boldsymbol{\sigma}(\cdot,\u):=\{\boldsymbol{\sigma}_k(\cdot,\u)\}_{k\in\mathbb{N}}:\ell^2\to\V_p$ defined by \eqref{3p2} is in $\gamma(\ell^2;\V_p)$, that is,  $\boldsymbol{\sigma}(\cdot,\u)$ is a $\gamma$-radonifying operator from $\ell^2$ to $\V_p$. For the orthonormal basis $\{e_k\}_{k\in\N}\in\ell^2$ with $e_k=(0,\ldots,1,\ldots)$ and  an independent standard Gaussian sequence $\{\gamma_k\}_{k\in\N}$  on some probability space $(\widetilde{\Omega},\widetilde{\mathcal{F}},\widetilde{\mathbb{P}})$,  using H\"older's inequality, stochastic Fubini's theorem, Gaussian formula (\cite[Appendix A.1]{DNu}) and Minkowski's  inequality, 
	we have
	\begin{align*} 
		\|\boldsymbol{\sigma}(t,\u)\|_{\gamma(\ell^2;\V_p)}^2&=\lim_{N\to\infty}\wi\E\left[\left\|\sum_{k=1}^{N}\gamma_k\boldsymbol{\sigma}_k(t,\u)e_k\right\|_{\V_p}^2\right]\nonumber\\&=\lim_{N\to\infty}\wi\E\left[\left(\int_{\mathcal{O}}\left|\sum_{k=1}^{N}\gamma_k\nabla\boldsymbol{\sigma}_k(t,\u(x))\right|^p\d x\right)^{\frac{2}{p}}\right]\nonumber\\&\leq\lim_{N\to\infty} \left\{\int_{\mathcal{O}}\wi\E\left|\sum_{k=1}^{N}\gamma_k\nabla\boldsymbol{\sigma}_k(t,\u(x))\right|^p\d x\right\}^{\frac{2}{p}}\nonumber\\&=\lim_{N\to\infty}A_p\left\{\int_{\mathcal{O}}\left(\wi\E\left[\left|\sum_{k=1}^{N}\gamma_k\nabla\boldsymbol{\sigma}_k(t,\u(x))\right|^2\right]\right)^{\frac{p}{2}}\d x\right\}^{\frac{2}{p}}\nonumber\\&=\lim_{N\to\infty} A_p\left\{\int_{\mathcal{O}}\left(\sum_{k=1}^{N}\left|\nabla\boldsymbol{\sigma}_k(t,\u(x))\right|^2\right)^{\frac{p}{2}}\d x\right\}^{\frac{2}{p}}\nonumber\\&\leq \lim_{N\to\infty}A_p\sum_{k=1}^{N}\left(\int_{\mathcal{O}}\left|\nabla\boldsymbol{\sigma}_k(t,\u(x))\right|^{p}\d x\right)^{\frac{2}{p}}=A_p\sum_{k=1}^{\infty}\|\boldsymbol{\sigma}_k(t,\u)\|_{\V_p}^2\nonumber\\&\leq A_p\widehat{K}\left(1+\|\u\|_{\V_p}^2\right)<\infty,
	\end{align*}
	where $A_p=\int_{\mathbb{R}}\frac{|\xi|^p}{\sqrt{2\pi}}e^{-\frac{\xi^2}{2}}\d\xi.$ Similarly, one can show that the condition \eqref{4p10} implies that for every $\u\in\wi\L^r$,  for  $r\geq 2$, the linear map $\boldsymbol{\sigma}(\cdot,\u):=\{\boldsymbol{\sigma}_k(\cdot,\u)\}_{k\in\mathbb{N}}:\ell^2\to\wi\L^r$ defined by \eqref{3p2} is in $\gamma(\ell^2;\wi\L^r)$. 
	
	Let $\mu_{\u_0}^{\e}$ be the law of $\u^{\e}(\cdot)$ on $\C([0,T];\H)$. 
	Let us set 
	\begin{align}\label{4.999}\nonumber
		\mathcal{H}:=&\bigg\{\h=(h_1,\ldots,h_k,\ldots);\ \h(\cdot):[0,T]\to\ell^2\ \text{ such that }\nonumber\\&\qquad \h\ \text{ is absolutely continuous and }\ \sum_{k=1}^{\infty}\int_0^T\dot{h}_k^2(t)\d t <\infty\bigg\}.
	\end{align}
	For $\h\in\mathcal{H}$, let $\u^{\h}(\cdot)$ be the unique solution of the following deterministic equation: 
	\begin{equation}\label{499}
		\left\{
		\begin{aligned}
			\d\u^{\h}(t)&=\sum_{k=1}^{\infty}\boldsymbol{\sigma}_k(0,\u^{\h}(t))\dot{h}_k(t)\d t,\\
			\u^{\h}(0)&=\u_0. 
		\end{aligned}
		\right.
	\end{equation}The well-posedness of the system \eqref{499} is straight forward from \cite[Step 2, Appendix B]{AKMTM6}.

	For $\h(t)=\sum\limits_{k=1}^{\infty}h_k(t)e_k\in\mathcal{H}$, we define 
	\begin{align*}
		\mathrm{I}(\h)=\frac{1}{2}\sum_{k=1}^{\infty}\int_0^T\left[\dot{h}_k(t)\right]^2\d t.
	\end{align*}
	For $\g\in\C([0,T];\H)$, we also define 
	\begin{align*}
		\Gamma_{\g}:=\bigg\{\h\in\mathcal{H}: \g(t)=\u_0+\sum_{k=1}^{\infty}\int_0^t\boldsymbol{\sigma}_k(0,\g(s))\dot{h}_k(s)\d s, \ 0\leq t\leq T\bigg\}.
	\end{align*}
	Furthermore, we define 
	\begin{align}\label{4100}
		\mathrm{R}(\g)=\left\{\begin{array}{cc}\inf\limits_{\h\in\Gamma_{\g}}\mathrm{I}(\h)& \text{ if }\ \Gamma_{\g}\neq\varnothing,\\
			+\infty&\text{ if } \ \Gamma_{\g}=\varnothing. \end{array}\right.
	\end{align}
	
	Then we have the following result: 
	\begin{theorem}\label{maint}
		For $p\geq \frac{d}{2}+1,\ r\geq 2$ or $p\geq 2,\ r\geq 4$  ($\beta\mu> 1$ for $r=4$), under Hypotheses \ref{hyp} and \ref{hyp1},	$\mu_{\u_0}^{\e}$ satisfies a large deviation principle with the rate function $\mathrm{R}(\cdot)$, that is,
		\begin{enumerate}
			\item [(i)] For any closed set $\F\subset\C([0,T];\H)$,
			\begin{align*}
				\limsup_{\e\to 0}\e\log\mu_{\u_0}^{\e}(\F)\leq -\inf_{\g\in \F}\mathrm{R}(\g);
			\end{align*}
			\item [(ii)] For any open set $\mathrm{O}\subset\C([0,T];\H)$,
			\begin{align*}
				\liminf_{\e\to 0}\e\log\mu_{\u_0}^{\e}(\mathrm{O})\geq -\sup_{\g\in \mathrm{O}}\mathrm{R}(\g).
			\end{align*}
		\end{enumerate}
	\end{theorem}
	
	\begin{proof}
		Let $\v^{\e}(\cdot)$ be the solution of the stochastic equation
		\begin{align}\label{3.7}
			\v^{\e}(t)=\u_0+\sqrt{\e}\sum_{k=1}^{\infty}\int_0^t\boldsymbol{\sigma}_k(\e s,\v^{\e}(s))\d\boldsymbol{W}_k(s),
		\end{align}
		and $\tilde{\mu}_{\u_0}^{\e}$ be the law of $\v^{\e}(\cdot)$ on the $\C([0,T];\H)$. Then from \cite[Lemma 2.1]{AKMTM6}, we infer that $\v^{\e}(\cdot)$   satisfies a LDP with the rate function $\mathrm{R}(\cdot)$. Our main task is to show that two families of the probability measures $\mu^{\e}_{\u_0}$ and $\tilde{\mu}_{\u_0}^{\e}$  are exponentially equivalent, that is, for any $\delta>0$,
		\begin{align}\label{3p8}
			\lim_{\e\to 0}\e\log \mathbb{P}\left\{\sup_{t\in[0,T]}\|\u^{\e}(t)-\v^{\e}(t)\|_{\H}^2>\delta\right\}=-\infty. 
		\end{align}
	From	\cite[Theorem 4.2.13]{DZ}, we know that if one of the two exponentially equivalent families satisfies a LDP, so does the other. Then our Theorem \ref{maint} follows from \eqref{3p8} and \cite[Theorem 4.2.13]{DZ}.
	\end{proof}
	\begin{lemma}\label{lemLDP}
		Assume that Hypothesis \ref{hyp1} holds. Then, there exists a stochastic process $\{\v^\e\}_{\e>0}$ (solution to the system \eqref{3.7}) satisfying an LDP in $\C([0,T];\H)$ with the rate function $\mathrm{R}(\cdot)$ defined in \eqref{4100}.
	\end{lemma}
	\begin{proof}
		For the proof we are referring to \cite[Appendix B]{AKMTM6}, where the authors obtained LDP with more general diffusion coefficient.
	\end{proof}

	Now, it is left to prove \eqref{3p8} only. Due to the presence of  the nonlinear operators $\mathcal{A},\B$ and $\mathcal{C}$, the proof of \eqref{3p8} not easy and we divide the proof into several lemmas. The following result is an estimate of the probability that the solution of \eqref{37} leaves an energy ball.
	\begin{lemma}\label{lem3.7}
		For $\f\in\mathrm{L}^{p'}(0,T;\V_p')$, under Hypothesis \ref{hyp}, we have 
		\begin{align}\label{39}
			\lim_{M\to\infty}\sup_{0<\e\leq 1}\e\log\mathbb{P}\left\{\left(|\u^{\e}|_{\V_p,\wi\L^{r}}^{\H}(T)\right)^2>M\right\}=-\infty,
		\end{align}
		where 
		\begin{align*}
			&	\left(|\u^{\e}|_{\V_p,\wi\L^{r}}^{\H}(T)\right)^2\nonumber\\&=\sup_{0\leq t\leq T}\|\u^{\e}(t)\|_{\H}^2+\e\mu\int_0^T\|\u^{\e}(t)\|_{\V}^2\d t+\frac{\e\mu}{2}\int_0^T\|\u^{\e}(t)\|_{\V_p}^p\d t+\e\beta\int_0^T\|\u^{\e}(t)\|_{\wi\L^{r}}^{r}\d t.
		\end{align*}
	\end{lemma}
	\begin{proof}
		Applying the infinite-dimensional It\^o's formula to the process $\|\u^{\e}(\cdot)\|_{\H}^2$ (\cite{GK1,MTM8,Me}), we obtain for all $t\in[0,T]$,	$\mathbb{P}$-a.s., 
		\begin{align}\label{310}
			&	\|\u^{\e}(t)\|_{\H}^2+2\mu \e\int_0^t\left(\left(1+|\nabla\u(s)|^2\right)^{\frac{p-2}{2}}\nabla\u(s),\nabla\u(s)
			\right)\d s+2\beta\int_0^t\|\u^{\e}(s)\|_{\widetilde{\L}^{r}}^{r}\d s\nonumber\\&=\|\x\|_{\H}^2+2\e\int_0^t\langle\f(\e s),\u^{\e}(s)\rangle \d s+\e\sum_{k=1}^{\infty}\int_0^t\|\boldsymbol{\sigma}_k(\e s,\u^{\e}(s))\|_{\H}^2\d s\nonumber\\&\quad+2\sqrt{\e}\sum_{k=1}^{\infty}\int_0^t\big(\boldsymbol{\sigma}_k(\e s,\u^{\e}(s))\d\boldsymbol{W}_k(s),\u^{\e}(s)\big)\nonumber\\&=:\|\x\|_{\H}^2+I_1+I_2+I_3.
		\end{align}
		It should be noted that 
		\begin{align*}
			2\left(\left(1+|\nabla\u|^2\right)^{\frac{p-2}{2}}\nabla\u,\nabla\u
			\right)\geq \|\nabla\u\|_{\H}^2+\|\nabla\u\|_{\wi\L^p}^p. 
		\end{align*}
		We estimate $I_1$ using Cauchy-Schwarz's and Young's inequalities, and $I_2$ using Hypothesis \ref{hyp} (H.1) as 
		\begin{align}
			|I_1|&\leq 2\e\int_0^t\|\u^{\e}(s)\|_{\V_p}\|\f(
			\e s)\|_{\V_p'}\d s\nonumber\\&\leq\frac{\e\mu}{2}\int_0^t\|\u^{\e}(s)\|_{\V_p}^p\d s+\frac{2^{\frac{p+1}{p-1}}(p-1)}{ p}\left(\frac{1}{\mu p}\right)^{\frac{1}{p-1}}\int_0^{t\e}\|\f(s)\|_{\V_p'}^{p'}\d s,\label{311} \\
			|I_2|&\leq\e K\int_0^t\left(1+\|\u^{\e}(s)\|_{\H}^2\right)\d s\leq \e Kt+\e K\int_0^t\|\u^{\e}(s)\|_{\H}^2\d s.\label{312}
		\end{align}
		Substituting \eqref{311} and \eqref{312} in \eqref{310}, we deduce that  for all $t\in[0,T]$,	$\mathbb{P}$-a.s., 
		\begin{align*}
			&	\|\u^{\e}(t)\|_{\H}^2+\e\mu\int_0^t\|\u(s)\|_{\V}^2\d s+\frac{\e\mu}{2}\int_0^t\|\u(s)\|_{\V_p}^p\d s+2\beta\e\int_0^t\|\u(s)\|_{\widetilde{\L}^{r}}^{r}\d s\nonumber\\&\leq\left(\|\x\|_{\H}^2+\e Kt+C(p,\mu)\int_0^{t\e}\|\f(s)\|_{\V_p'}^{p'}\d s\right)+\e K\int_0^t\|\u^{\e}(s)\|_{\H}^2\d s\nonumber\\&\quad+2\sqrt{\e}\left|\sum_{k=1}^{\infty}\int_0^t\big(\boldsymbol{\sigma}_k(\e s,\u^{\e}(s))\d\boldsymbol{W}_k(s),\u^{\e}(s)\big)\right|,
		\end{align*}
		where $C(p,\mu)=\frac{2^{\frac{p+1}{p-1}}(p-1)}{ p}\left(\frac{1}{\mu p}\right)^{\frac{1}{p-1}}$. Therefore, taking supremum over time $t\in[0,T]$, we get  for all $t\in[0,T]$,	$\mathbb{P}$-a.s., 
		\begin{align*}
			\left(|\u^{\e}|_{\V_p,\wi\L^{r}}^{\H}(T)\right)^2 &\leq\left(\|\x\|_{\H}^2+\e KT+C(p,\mu)\int_0^{T\e}\|\f(t)\|_{\V_p'}^{p'}\d t\right)+\e K\int_0^T	\left(|\u^{\e}|_{\V_p,\wi\L^{r}}^{\H}(t)\right)^2 \d t\nonumber\\&\quad+2\sqrt{\e}\sup_{t\in[0,T]}\left|\sum_{k=1}^{\infty}\int_0^t\big(\boldsymbol{\sigma}_k(\e s,\u^{\e}(s))\d\boldsymbol{W}_k(s),\u^{\e}(s)\big)\right|.
		\end{align*}
		Thus, for $q\geq 2$, we obtain 
		\begin{align}\label{313}
			&	\left\{\E\left[	\left(|\u^{\e}|_{\V_p,\wi\L^{r}}^{\H}(T)\right)^{2q}\right]\right\}^{\frac{1}{q}}\nonumber\\&\leq 2\left(\|\x\|_{\H}^2+\e KT+C(p,\mu)\int_0^{T\e}\|\f(t)\|_{\V_p'}^{p'}\d t\right)+2\e K\left\{\E\left[\left(\int_0^T	\left(|\u^{\e}|_{\V_p,\wi\L^{r}}^{\H}(t)\right)^2 \d t\right)^{q}\right]\right\}^{\frac{1}{q}}\nonumber\\&\quad+4\sqrt{\e}\left\{\E\left[\sup_{t\in[0,T]}\left|\sum_{k=1}^{\infty}\int_0^t\big(\boldsymbol{\sigma}_k(\e s,\u^{\e}(s))\d\boldsymbol{W}_k(s),\u^{\e}(s)\big)\right|^q\right]\right\}^{\frac{1}{q}}.
		\end{align}
		In order to estimate the stochastic integral term, we use the following important result (cf. \cite{MTB,BD1}). There exists
		a universal constant $C>0$ such that, for any $q \geq 2$ and for any continuous martingale $\M=\{\M_t\}_{t\geq 0}$ with $\M_0 = 0$, we have 
		\begin{align}\label{3.15}
			\|\M^*_t\|_{\mathrm{L}^q}\leq Cq^{\frac{1}{2}}\|\langle\M_t\rangle\|_{\mathrm{L}^q},
		\end{align}
		where $\M^*_t=\sup\limits_{s\in[0,t]}\|\M_s\|_{\H}$ and $\{\langle\M_t\rangle\}_{t\geq 0}$  is the quadratic variation process of $\M$. Making use of this result and Minkowski's integral inequality, we get 
		\begin{align}\label{315}
			&4\sqrt{\e}\left\{\E\left[\sup_{t\in[0,T]}\left|\sum_{k=1}^{\infty}\int_0^t\big(\boldsymbol{\sigma}_k(\e s,\u^{\e}(s))\d\boldsymbol{W}_k(s),\u^{\e}(s)\big)\right|^q\right]\right\}^{\frac{1}{q}}\nonumber\\&\leq 4C\sqrt{q\e}\left\{\E\left[\left(\int_0^T\sum_{k=1}^{\infty}\|\boldsymbol{\sigma}_k(\e t,\u^{\e}(t))\|_{\H}^2\|\u^{\e}(t)\|_{\H}^2\d t\right)^{\frac{q}{2}}\right]\right\}^{\frac{1}{q}}\nonumber\\&\leq 4CK\sqrt{q\e}\left\{\E\left[\left(\int_0^T\left(1+\|\u^{\e}(t)\|_{\H}^2\right)\|\u^{\e}(t)\|_{\H}^2\d t\right)^{\frac{q}{2}}\right]\right\}^{\frac{1}{q}}\nonumber\\&\leq 4CK\sqrt{q\e}\left[\left\{\E\left[\left(\int_0^T\left(1+\|\u^{\e}(t)\|_{\H}^2\right)^2\d t\right)^{\frac{q}{2}}\right]\right\}^{\frac{2}{q}}\right]^{\frac{1}{2}}\nonumber\\&\leq 4CK\sqrt{2q\e}\left[\left\{\E\left[\left(\int_0^T\left(1+\|\u^{\e}(t)\|_{\H}^4\right)\d t\right)^{\frac{q}{2}}\right]\right\}^{\frac{2}{q}}\right]^{\frac{1}{2}}\nonumber\\&\leq  8CK\sqrt{q\e}\left\{\int_0^T\left(1+\left[\E\left(\|\u^{\e}(t)\|_{\H}^{2q}\right)\right]^{\frac{2}{q}}\right)\d t\right\}^{\frac{1}{2}}.
		\end{align}
		Moreover, by using Minkowski's integral inequality, we have 
		\begin{align}\label{3.16}
			&2\e K\left\{\E\left[\left(\int_0^T	\left(|\u^{\e}|_{\V_p,\wi\L^{r}}^{\H}(t)\right)^2 \d t\right)^{q}\right]\right\}^{\frac{1}{q}}\leq 2\e K\int_0^T\left[\E\left(|\u^{\e}|_{\V_p,\wi\L^{r}}^{\H}(t)\right)^{2q}\right]^{\frac{1}{q}}\d t.
		\end{align}
		Substituting \eqref{315} and \eqref{3.16}  in \eqref{313}, we deduce that 
		\begin{align}
			&	\left\{\E\left[	\left(|\u^{\e}|_{\V_p,\wi\L^{r}}^{\H}(T)\right)^{2q}\right]\right\}^{\frac{2}{q}}\nonumber\\&\leq 8\left(\|\x\|_{\H}^2+\e KT+C(p,\mu)\int_0^{T\e}\|\f(t)\|_{\V_p'}^{p'}\d t\right)^2+8\e^2 K^2T\int_0^T\left[\E\left(|\u^{\e}|_{\V_p,\wi\L^{r}}^{\H}(t)\right)^{2q}\right]^{\frac{2}{q}}\d t\nonumber\\&\quad+128C^2K^2{q\e} T+128C^2K^2{q\e}\int_0^T\left[\E\left(|\u^{\e}|_{\V_p,\wi\L^{r}}^{\H}(t)\right)^{2q}\right]^{\frac{2}{q}}\d t.
		\end{align}
		An application of Gronwall's inequality yields 
		\begin{align}\label{318}
			&	\left\{\E\left[	\left(|\u^{\e}|_{\V_p,\wi\L^{r}}^{\H}(T)\right)^{2q}\right]\right\}^{\frac{2}{q}}\nonumber\\&\leq\left\{8\left(\|\x\|_{\H}^2+\e KT+C(p,\mu)\int_0^{T\e}\|\f(t)\|_{\V_p'}^{p'}\d t\right)^2+128C^2K^2{q\e} T\right\}\nonumber\\&\quad\times\exp\left\{8\e^2 K^2T^2+128C^2K^2{q\e}T\right\}.
		\end{align}
		By Markov's inequality, we have 
		\begin{align*}
			\mathbb{P}\left\{\left(|\u^{\e}|_{\V_p,\wi\L^{r}}^{\H}(T)\right)^{2}>M\right\}\leq M^{-q}\E\left[	\left(|\u^{\e}|_{\V_p,\wi\L^{r}}^{\H}(T)\right)^{2q}\right].
		\end{align*}
		Letting $q=\frac{1}{\e}$ in \eqref{318}, we get 
		\begin{align}
			&	\e\log \mathbb{P}\left\{\left(|\u^{\e}|_{\V_p,\wi\L^{r}}^{\H}(T)\right)^{2}>M\right\}\nonumber\\&\leq -\log M+ \log 	\left\{\E\left[	\left(|\u^{\e}|_{\V_p,\wi\L^{r}}^{\H}(T)\right)^{2q}\right]\right\}^{\frac{1}{q}}\nonumber\\&\leq -\log M+\log\sqrt{\left\{8\left(\|\x\|_{\H}^2+\e KT+C(p,\mu)\int_0^{T\e}\|\f(t)\|_{\V_p'}^{p'}\d t\right)^2+128C^2K^2 T\right\}}\nonumber\\&\quad+8\e^2K^2T^2+128C^2K^2T. 
		\end{align}
		Therefore, it is immediate that 
		\begin{align*}
			&\sup_{0<\e\leq 1}\e \mathbb{P}\left\{\left(|\u^{\e}|_{\V_p,\wi\L^{r}}^{\H}(T)\right)^{2}>M\right\}\nonumber\\&\leq -\log M+\log\sqrt{\left\{8\left(\|\x\|_{\H}^2+ KT+C(p,\mu)\int_0^{T}\|\f(t)\|_{\V_p'}^{p'}\d t\right)^2+128C^2K^2 T\right\}}\nonumber\\&\quad+8K^2T^2+128C^2K^2T. 
		\end{align*}
		Letting $M\to\infty$ in the above expression leads to the estimate \eqref{39}. 
	\end{proof}
	
	Since $\mathscr{V}\subset\V_p\cap\wi\L^{r}\subset\H$ and $\mathscr{V}$ is dense in $\H$ implies that $\V_p\cap\wi\L^{r}$ is also dense in $\H$.  Let $\{\u_0^n\}$ be a sequence in $\V_p\cap\wi\L^{r}$ such that $\|\u_0^n-\u_0\|_{\H}\to 0$, as $n\to\infty$. Let $\u_n^{\e}(\cdot)$ be the solution of \eqref{37} with the initial value $\u_0^n$. From the proof of Lemma \ref{lem3.7}, it follows that
	\begin{align}\label{320}
		\lim_{M\to\infty}\sup_{n}\sup_{0<\e\leq 1}\e\log\mathbb{P}\left\{\left(|\u^{\e}_n|_{\V_p,\wi\L^{r}}^{\H}(T)\right)^2>M\right\}=-\infty.
	\end{align}
	Let $\v_n^{\e}(\cdot)$ be the solution of \eqref{3.7} with the initial value $\u_0^n$. Then, we have the following result: 
	\begin{lemma}\label{lem3.8}
		Under hypothesis \ref{hyp1}, 	for any $n\in\mathbb{Z}^+$, we have 
		\begin{align}
			\lim_{M\to\infty}\sup_{0<\e\leq 1}\e\log\mathbb{P}\left\{\sup_{0\leq t\leq T}\|\v_n^{\e}(t)\|_{\V}^2>M\right\}&=-\infty,\label{3p22}\\
			\lim_{M\to\infty}\sup_{0<\e\leq 1}\e\log\mathbb{P}\left\{\sup_{0\leq t\leq T}\|\v_n^{\e}(t)\|_{\V_p}^p>M\right\}&=-\infty\label{3p23},
		\end{align}and
		\begin{align}
			\lim_{M\to\infty}\sup_{0<\e\leq 1}\e\log\mathbb{P}\left\{\sup_{0\leq t\leq T}\|\v_n^{\e}(t)\|_{\wi\L^{r}}^{r}>M\right\}=-\infty,\label{3p24}
		\end{align}
		for $p\geq 2$ and $r\geq 2$. 
	\end{lemma}
	\begin{proof}
		Applying infinite-dimensional It\^o's formula to the process $\|\v_n^{\e}(\cdot)\|_{\V}^2$, we get for all $t\in[0,T]$,	$\mathbb{P}$-a.s., 
		\begin{align}
			\|\nabla\v_n^{\e}(t)\|_{\H}^2&=\|\nabla\u_0^n\|_{\H}^2+\e\sum_{k=1}^{\infty}\int_0^t\|\nabla\boldsymbol{\sigma}_k(\e s,\v_n^{\e}(s))\|_{\H}^2\d s\nonumber\\&\quad+2\sqrt{\e}\sum_{k=1}^{\infty}\int_0^t\big(\nabla\boldsymbol{\sigma}_k(\e s,\v_n^{\e}(s))\d\boldsymbol{W}_k(s),\nabla\v_n^{\e}(s)\big).
		\end{align}
		By using Hypothesis \ref{hyp} (H.1) and \eqref{3.15}, we obtain 
		\begin{align}\label{3.25}
			&\left\{\E\left[\sup_{0\leq t\leq T}\|\v_n^{\e}(t)\|_{\V}^{2q}\right]\right\}^{\frac{2}{q}}\nonumber\\&\leq 2\|\u_0^n\|_{\V}^4+2\e^2\left\{\mathbb{E}\left[\left(\sum_{k=1}^{\infty}\int_0^T\|\nabla\boldsymbol{\sigma}_k(\e s,\v_n^{\e}(s))\|_{\H}^2\d s\right)^{q}\right]\right\}^{\frac{2}{q}}\nonumber\\&\quad+8\e\left\{\E\left[\sup_{0\leq t\leq T}\left|\sum_{k=1}^{\infty}\int_0^t\big(\nabla\boldsymbol{\sigma}_k(\e s,\v_n^{\e}(s))\d\boldsymbol{W}_k(s),\nabla\v_n^{\e}(s)\big) \right|\right]^{q}\right\}^{\frac{2}{q}}\nonumber\\&\leq 2\|\u_0^n\|_{\V}^4+2\e^2\widehat{K}^2\left\{\E\left[\left(\int_0^T(1+\|\v_n^{\e}(s)\|_{\V}^2)\d s\right)^q\right]\right\}^{\frac{2}{q}}\nonumber\\&\quad+8Cq\e\left\{\E\left[\left(\sum_{k=1}^{\infty}\int_0^T\|\v_n^{\e}(s)\|_{\V}^2\|\nabla\boldsymbol{\sigma}_k(\e s,\v_n^{\e}(s))\|_{\H}^2\d s\right)^{\frac{q}{2}}\right]\right\}^{\frac{2}{q}}\nonumber\\&\leq 2\|\u_0^n\|_{\V}^4+4\e^2\widehat{K}^2T\left\{T+\int_0^T\left[\E\left(\sup_{0\leq r\leq s}\|\v_n^{\e}(r)\|_{\V}^{2q}\right)\right]^{\frac{2}{q}}\d s\right\}\nonumber\\&\quad+16 Cq\e\widehat{K}\left\{T+\int_0^T\left[\E\left(\sup_{0\leq r\leq s}\|\v_n^{\e}(r)\|_{\V}^{2q}\right)\right]^{\frac{2}{q}}\d s\right\}. 
		\end{align}
		An application of Gronwall's inequality yields 
		\begin{align*}
			\left\{\E\left[\sup_{0\leq t\leq T}\|\v_n^{\e}(t)\|_{\V}^{2q}\right]\right\}^{\frac{2}{q}}\leq \left\{ 2\|\u_0^n\|_{\V}^4+4\e^2\widehat{K}^2T^2+16 Cq\e\widehat{K}T\right\}e^{4\e^2\widehat{K}^2T^2+16 Cq\e\widehat{K}T},
		\end{align*}
		and the proof of \eqref{3p22} can be completed by using similar arguments as in the proof of Lemma \ref{lem3.7}. 
		
		For some $p\geq 2$, applying the infinite-dimensional formula to the process $\|\v_n^{\e}(\cdot)\|_{\V_p}^p,$ we obtain for all $t\in[0,T]$,	$\mathbb{P}$-a.s.,  (see \cite[Theorem A.1]{ZBSP})
		\begin{align}
			\|\v_n^{\e}(t)\|_{\V_p}^p&=\|\u_0^n\|_{\V_p}^p+\frac{p(p-1)\e}{2}\sum_{k=1}^{\infty}\int_0^t\langle|\nabla\v^{\e}_n(s)|^{p-2},|\nabla\boldsymbol{\sigma}_k(\e s,\v_n^{\e}(s))|^2\rangle\d s\nonumber\\&\quad+p\sqrt{\e}\sum_{k=1}^{\infty}\int_0^t\big\langle|\nabla\v^{\e}_n(s)|^{p-2}\nabla\v^{\e}_n(s),\nabla\boldsymbol{\sigma}_k(\e s,\v_n^{\e}(s))\d\boldsymbol{W}_k(s)\big\rangle,
		\end{align}
		For $q\geq 2$, using Hypothesis \ref{hyp1} (H.3) and \eqref{3.15}, we find 
		\begin{align}\label{3p28}
			&\left\{\E\left[\sup_{0\leq t\leq T}\|\v_n^{\e}(t)\|_{\V_p}^{pq}\right]\right\}^{\frac{2}{q}}\nonumber\\&\leq 2\|\u_0^n\|_{\V_p}^{2p}+p^2(p-1)^2\e^2\left\{\E\left[\left(\sum_{k=1}^{\infty}\int_0^T\|\nabla\v_n^{\e}(s)\|_{\wi\L^p}^{p-2}\|\nabla\boldsymbol{\sigma}_k(\e s,\v_n^{\e}(s))\|_{\wi\L^p}^2\d s\right)^q\right]\right\}^{\frac{2}{q}}\nonumber\\&\quad+2p^2\e\left\{\E\left[\sup_{0\leq t\leq T}\left|\sum_{k=1}^{\infty}\int_0^t\big\langle|\nabla\v^{\e}_n(s)|^{p-2}\nabla\v^{\e}_n(s),\nabla\boldsymbol{\sigma}_k(\e s,\v_n^{\e}(s))\d\boldsymbol{W}_k(s)\big\rangle\right|\right]^q\right\}^{\frac{2}{q}}\nonumber\\&\leq 2\|\u_0^n\|_{\V_p}^{2p}+p^2(p-1)^2\e^2\wi K^2\left\{\E\left[\left(\int_0^T\|\nabla\v^{\e}_n(s)\|_{\wi\L^p}^{p-2}\left(1+\|\nabla\v^{\e}_n(s)\|_{\wi\L^p}^2\right)\d s\right)^q\right]\right\}^{\frac{2}{q}}\nonumber\\&\quad+2Cp^2q\e\left\{\E\left[\left(\int_0^T\sum_{k=1}^{\infty}\|\nabla\v^{\e}_n(s)\|_{\wi\L^p}^{2(p-1)}\|\nabla\boldsymbol{\sigma}_k(\e s,\v_n^{\e}(s))\|_{\wi\L^p}^2\d s\right)^{\frac{q}{2}}\right]\right\}^{\frac{2}{q}}\nonumber\\&\leq  2\|\u_0^n\|_{\V_p}^{2p}+2p^2(p-1)^2\e^2\wi K^2\left\{T+\int_0^T\left[\E\left(\sup_{0\leq r\leq s}\|\nabla\v_n^{\e}(r)\|_{\wi\L^p}^{pq}\right)\right]^{\frac{2}{q}}\d s\right\}\nonumber\\&\quad+2Cp^2q\e\wi K\left\{\E\left[\left(\int_0^T\|\nabla\v^{\e}_n(s)\|_{\wi\L^p}^{2(p-1)}(1+\|\nabla\v^{\e}_n(s)\|_{\wi\L^p}^2)\d s\right)^{\frac{q}{2}}\right]\right\}^{\frac{2}{q}} \nonumber\\&\leq  2\|\u_0^n\|_{\V_p}^{2p}+2p^2(p-1)^2\e^2\wi K^2\left\{T+\int_0^T\left[\E\left(\sup_{0\leq r\leq s}\|\v_n^{\e}(r)\|_{\V_p}^{pq}\right)\right]^{\frac{2}{q}}\d s\right\}\nonumber\\&\quad+4Cp^2q\e\wi K\left\{T+\int_0^T\left[\E\left(\sup_{0\leq r\leq s}\|\v_n^{\e}(r)\|_{\V_p}^{pq}\right)\right]^{\frac{2}{q}}\d s\right\}.
		\end{align}
		An application of Gronwall's inequality in \eqref{3p28} yields 
		\begin{align}
			&\left\{\E\left[\sup_{0\leq t\leq T}\|\v_n^{\e}(t)\|_{\V_p}^{pq}\right]\right\}^{\frac{2}{q}}\nonumber\\&\leq \left\{2\|\u_0^n\|_{\V_p}^{2p}+2p^2(p-1)^2\e^2\wi K^2T+4Cp^2q\e\wi KT\right\}e^{2p^2(p-1)^2\e^2\wi K^2T+4Cp^2q\e\wi KT},
		\end{align}
		and the rest of the proof  of \eqref{3p23} can be completed by using similar arguments as in the proof of Lemma \ref{lem3.7}. 
		
		For some $r\geq 2$, applying the infinite-dimensional formula to the process $\|\v_n^{\e}(\cdot)\|_{\wi\L^r}^r,$  using Hypothesis \ref{hyp1} (H.3), and then applying similar arguments as in the previous case one can obtain the estimate \eqref{3p24}. 
	\end{proof}

	\begin{lemma}\label{lem3.9}
		For any $\delta>0$, we have 
		\begin{align}
			\lim_{n\to\infty}\sup_{0<\e\leq 1}\e\log\mathbb{P}\left\{\sup_{0\leq t\leq T}\|\v_n^{\e}(t)-\v^{\e}(t)\|_{\H}^2>\delta \right\}=-\infty. 
		\end{align}
	\end{lemma}
	\begin{proof}
		Applying infinite-dimensional It\^o's formula to the process $\|\v^{\e}_n(\cdot)-\v^{\e}(\cdot)\|_{\H}^2$, we find for all $t\in[0,T]$,	$\mathbb{P}$-a.s., 
		\begin{align*}
			&	\|\v^{\e}_n(t)-\v^{\e}(t)\|_{\H}^2\nonumber\\&=\|\u_0^n-\u_0\|_{\H}^2+\e\sum_{k=1}^{\infty}\int_0^t\|\boldsymbol{\sigma}_k(\e s,\v^{\e}_n(s))-\boldsymbol{\sigma}_k(\e s,\v^{\e}(s))\|_{\H}^2\d s\nonumber\\&\quad+2\sqrt{\e}\sum_{k=1}^{\infty}\int_0^t\big(\boldsymbol{\sigma}_k(\e s,\v_n^{\e}(s))-\boldsymbol{\sigma}_k(\e s,\v^{\e}(s))\d\boldsymbol{W}_k(s),\v^{\e}_n(s)-\v^{\e}(s)\big).
		\end{align*}
		Using calculations similar to \eqref{3.25} yield 
		\begin{align*}
			&\left\{\E\left[\sup_{0\leq t\leq T}\|\v_n^{\e}(t)-\v^{\e}(t)\|_{\V}^{2q}\right]\right\}^{\frac{2}{q}}\nonumber\\&\leq 2\|\u_0^n-\u_0\|_{\V}^4+(4\e^2L^2T+16 Cq\e L^2)\int_0^T\left[\E\left(\sup_{0\leq r\leq t}\|\v_n^{\e}(s)-\v_{\e}(s)\|_{\V}^{2q}\right)\right]^{\frac{2}{q}}\d s.
		\end{align*}
		Applying Gronwall's inequality, we get 
		\begin{align*}
			\left\{\E\left[\sup_{0\leq t\leq T}\|\v_n^{\e}(t)-\v^{\e}(t)\|_{\V}^{2q}\right]\right\}^{\frac{2}{q}}\leq 2\|\u_0^n-\u_0\|_{\V}^4e^{(4\e^2L^2T+16 Cq\e L^2)T}.
		\end{align*}
		Taking $q=\frac{2}{\e}$ and applying Markov's inequality, we obtain 
		\begin{align}\label{3.27}
			&\sup_{0<\e\leq 1}\e\log\mathbb{P}\left\{\sup_{0\leq t\leq T}\|\v^{\e}(t)-\v^{\e}_n(t)\|_{\H}^2>\delta \right\}\nonumber\\&\quad\leq\sup_{0<\e\leq 1}\e\log\Bigg\{\frac{\E\left[\sup\limits_{0\leq t\leq T}\|\v^{\e}(t)-\v^{\e}_n(t)\|_{\H}^{2q}\right]}{\delta^q}\Bigg\}\nonumber\\&\quad\leq (4L^2T+32C L^2)T+\log (2\|\u_0^n-\u_0\|_{\V}^4)-2\log\delta \nonumber\\&\quad\to-\infty,\ \text{ as }\ n\to\infty,
		\end{align}
		for any $\delta>0$, which completes the proof. 
	\end{proof}

	\begin{lemma}\label{lem3.10}
		For any $\delta>0$, we have 
		\begin{align}
			\lim_{n\to\infty}\sup_{0<\e\leq 1}\e\log\mathbb{P}\left\{\sup_{0\leq t\leq T}\|\u^{\e}_n(t)-\u^{\e}(t)\|_{\H}^2>\delta \right\}=-\infty. 
		\end{align}
	\end{lemma}
	\begin{proof}
		For $M>0$, let us define a sequence of stopping times as 
		\begin{align}
			\tau_{\e,M}:=\inf_{t\geq 0}\left\{t:\e\int_0^t\big(\|\u^{\e}(s)\|_{\V}^2+\|\u^{\e}(s)\|_{\V_p}^p\big)\d s>M\ \text{ or }\ \|\u^{\e}(t)\|_{\H}^2>M\right\}.
		\end{align}	
		It is clear that 
		\begin{align}\label{330}
			&	\mathbb{P}\left\{\sup_{0\leq t\leq T}\|\u^{\e}_n(t)-\u^{\e}(t)\|_{\H}^2>\delta,\ \left(|\u^{\e}|_{\V_p,\wi\L^{r}}^{\H}(T)\right)^{2}\leq M\right\}\nonumber\\&\leq 	\mathbb{P}\left\{\sup_{0\leq t\leq T}\|\u^{\e}_n(t)-\u^{\e}(t)\|_{\H}^2>\delta,\ \tau_{\e,M}\geq T\right\}\nonumber\\&\leq 	\mathbb{P}\left\{\sup_{0\leq t\leq T}\|\u^{\e}_n(t)-\u^{\e}(t)\|_{\H}^2>\delta\right\}. 
		\end{align}
		\vskip 0.2 cm
		\noindent \textbf{Case 1}: $p\geq\frac{d}{2}+1$ and $r\geq 2$.
		Let $k$ be a positive constant. For $p\geq\frac{d}{2}+1$ and $r\geq 2$,  applying the infinite-dimensional It\^o's formula to  $ e^{-k\e\int_0^{\cdot}\|\u^{\e}(s)\|_{\V_p}^{\frac{2p}{2p-d}}\d s}\|\u^{\e}_n(\cdot)-\u^{\e}(\cdot)\|_{\H}^2$, we find for all $t\in[0,T]$,	$\mathbb{P}$-a.s., 
		\begin{align}\label{3p30}
			&	e^{-k\e\int_0^{t\wedge\tau_{\e,M}}\|\u^{\e}(s)\|_{\V_p}^{\frac{2p}{2p-d}}\d s}\|\u^{\e}_n(t\wedge\tau_{\e,M})-\u^{\e}(t\wedge\tau_{\e,M})\|_{\H}^2\nonumber\\&\quad-2\mu\e\int_0^{t\wedge\tau_{\e,M}}e^{-k\e\int_0^{s\wedge\tau_{\e,M}}\|\u^{\e}(r)\|_{\V_p}^{\frac{2p}{2p-d}}\d r}
			\langle\mathcal{A}(\u^{\e}_n(s)-\u^{\e}(s)),\u^{\e}_n(s)-\u^{\e}(s)\rangle\d s\nonumber\\&\quad+2\beta\e\int_0^{t\wedge\tau_{\e,M}}e^{-k\e\int_0^{s\wedge\tau_{\e,M}}\|\u^{\e}(r)\|_{\V_p}^{\frac{2p}{2p-d}}\d r}\langle\mathcal{C}(\u^{\e}_n(s))-\mathcal{C}(\u^{\e}(s)),\u^{\e}_n(s)-\u^{\e}(s)\rangle\d s\nonumber\\&=\|\u_0^n-\u_0\|_{\H}^2-k\e\int_0^{t\wedge\tau_{\e,M}}e^{-k\e\int_0^{s\wedge\tau_{\e,M}}\|\u^{\e}(r)\|_{\V_p}^{\frac{2p}{2p-d}}\d r}\|\u^{\e}(s)\|_{\V_p}^{\frac{2p}{2p-d}}\|\u^{\e}_n(s)-\u^{\e}(s)\|_{\H}^2\d s\nonumber\\&\quad-2\e\int_0^{t\wedge\tau_{\e,M}}e^{-k\e\int_0^{s\wedge\tau_{\e,M}}\|\u^{\e}(r)\|_{\V_p}^{\frac{2p}{2p-d}}\d r}\langle\B(\u^{\e}_n(s))-\B(\u^{\e}(s)),\u^{\e}_n(s)-\u^{\e}(s)\rangle\d s\nonumber\\&\quad+\e\sum_{j=1}^{\infty}\int_0^{t\wedge\tau_{\e,M}}e^{-k\e\int_0^{s\wedge\tau_{\e,M}}\|\u^{\e}(r)\|_{\V_p}^{\frac{2p}{2p-d}}\d r}\|\boldsymbol{\sigma}_j(\e s,\u^{\e}_n(s))-\boldsymbol{\sigma}_j(\e s,\u^{\e}(s))\|_{\H}^2\d s\nonumber\\&\quad+2\sqrt{\e}\sum_{j=1}^{\infty}\int_0^{t\wedge\tau_{\e,M}}e^{-k\e\int_0^{s\wedge\tau_{\e,M}}\|\u^{\e}(r)\|_{\V_p}^{\frac{2p}{2p-d}}\d r}\nonumber\\&\qquad\qquad\times\big((\boldsymbol{\sigma}_j(\e s,\u^{\e}_n(s))-\boldsymbol{\sigma}_j(\e s,\u^{\e}(s)))\d\boldsymbol{W}_j(s),\u^{\e}_n(s)-\u^{\e}(s)\big).
		\end{align}
		From \eqref{ae}, we infer that 
		\begin{align}\label{3.31}
			-&2\mu\e\langle\mathcal{A}(\u^{\e}_n-\u^{\e}),\u^{\e}_n-\u^{\e}\rangle\geq 2\mu\e\|\nabla(\u^{\e}_n-\u^{\e})\|_{\H}^2,
		\end{align}
		and from \eqref{2.23}, we obtain 
		\begin{align}\label{3.32}
			2&\beta\e\langle\mathcal{C}(\u^{\e}_n)-\mathcal{C}(\u^{\e}),\u^{\e}_n-\u^{\e}\rangle\geq 0. 
		\end{align}
		For $p\geq\frac{d}{2}+1$ and $r\geq 2$, we estimate $-2\e\langle\B(\u^{\e}_n)-\B(\u^{\e}),\u^{\e}_n-\u^{\e}\rangle$ in a similar way as in \eqref{225} as 
		\begin{align}\label{3.33}
			-&2\e\langle\B(\u^{\e}_n)-\B(\u^{\e}),\u^{\e}_n-\u^{\e}\rangle\nonumber\\&\leq \mu\e\|\nabla(\u^{\e}_n-\u^{\e})\|_{\H}^2+\e\wi\eta\|\u^{\e}\|_{\V_p}^{\frac{2p}{2p-d}}\|\u^{\e}_n-\u^{\e}\|_{\H}^2,
		\end{align}
		where $\wi\eta$ is defined in \eqref{eta1}. Let us choose $k>\wi\eta$. Combining \eqref{3.31}-\eqref{3.33} and then substituting it in \eqref{3p30}, we find for all $t\in[0,T]$,	$\mathbb{P}$-a.s., 
		\begin{align}\label{3p34}
			&	e^{-k\e\int_0^{t\wedge\tau_{\e,M}}\|\u^{\e}(s)\|_{\V_p}^{\frac{2p}{2p-d}}\d s}\|\u^{\e}_n(t\wedge\tau_{\e,M})-\u^{\e}(t\wedge\tau_{\e,M})\|_{\H}^2\nonumber\\&\leq\|\u_0^n-\u_0\|_{\H}^2+\e L\int_0^{t\wedge\tau_{\e,M}}e^{-k\e\int_0^{s\wedge\tau_{\e,M}}\|\u^{\e}(r)\|_{\V_p}^{\frac{2p}{2p-d}}\d r}\|\u^{\e}_n(s)-\u^{\e}(s)\|_{\H}^2\d s\nonumber\\&\quad+2\sqrt{\e}\bigg|\sum_{j=1}^{\infty}\int_0^{t\wedge\tau_{\e,M}}e^{-k\e\int_0^{s\wedge\tau_{\e,M}}\|\u^{\e}(r)\|_{\V_p}^{\frac{2p}{2p-d}}\d r}\nonumber\\&\qquad\qquad\times\big((\boldsymbol{\sigma}_j(\e s,\u^{\e}_n(s))-\boldsymbol{\sigma}_j(\e s,\u^{\e}(s)))\d\boldsymbol{W}_j(s),\u^{\e}_n(s)-\u^{\e}(s)\big)\bigg|,
		\end{align}
		where we have used Hypothesis \ref{hyp} (H.2) also. 
		Using \eqref{3.15}, for $q\geq 2$, one can deduce that 
		\begin{align}\label{3p35}
			&\left\{\E\left[\sup_{0\leq t\leq T\wedge\tau_{\e,M}}\left(e^{-k\e\int_0^{t\wedge\tau_{\e,M}}\|\u^{\e}(s)\|_{\V_p}^{\frac{2p}{2p-d}}\d s}\|\u^{\e}_n(t)-\u^{\e}(t)\|_{\H}^2\right)\right]^q\right\}^{\frac{2}{q}}\nonumber\\&\leq 2\|\u_0^n-\u_0\|_{\H}^4+2\e^2L^2\int_0^T\left\{\E\left[\sup_{0\leq s\leq t\wedge\tau_{\e,M}}\left(e^{-k\e\int_0^{s}\|\u^{\e}(r)\|_{\V_p}^{\frac{2p}{2p-d}}\d r}\|\u^{\e}_n(s)-\u^{\e}(s)\|_{\H}^2\right)\right]^q\right\}^{\frac{2}{q}}\d t\nonumber\\&\quad+8C\e qL^2\int_0^T\left\{\E\left[\sup_{0\leq s\leq t\wedge\tau_{\e,M}}\left(e^{-2k\e\int_0^{s}\|\u^{\e}(r)\|_{\V_p}^{\frac{2p}{2p-d}}\d r}\|\u^{\e}_n(s)-\u^{\e}(s)\|_{\H}^4\right)\right]^{\frac{q}{2}}\right\}^{\frac{2}{q}}\d t\nonumber\\&\leq  2\|\u_0^n-\u_0\|_{\H}^4+2\e^2L^2\int_0^T\left\{\E\left[\sup_{0\leq s\leq t\wedge\tau_{\e,M}}\left(e^{-k\e\int_0^{s}\|\u^{\e}(r)\|_{\V_p}^{\frac{2p}{2p-d}}\d r}\|\u^{\e}_n(s)-\u^{\e}(s)\|_{\H}^2\right)\right]^q\right\}^{\frac{2}{q}}\d t\nonumber\\&\quad+8C\e qL^2\int_0^T\left\{\E\left[\sup_{0\leq s\leq t\wedge\tau_{\e,M}}\left(e^{-k\e\int_0^{s}\|\u^{\e}(r)\|_{\V_p}^{\frac{2p}{2p-d}}\d r}\|\u^{\e}_n(s)-\u^{\e}(s)\|_{\H}^2\right)\right]^q\right\}^{\frac{2}{q}}\d t.
		\end{align}
		An application of Gronwall's inequality in \eqref{3p35} gives 
		\begin{align}
			&\left\{\E\left[\sup_{0\leq t\leq T\wedge\tau_{\e,M}}\left(e^{-k\e\int_0^{t\wedge\tau_{\e,M}}\|\u^{\e}(s)\|_{\V_p}^{\frac{2p}{2p-d}}\d s}\|\u^{\e}_n(t)-\u^{\e}(t)\|_{\H}^2\right)^q\right]\right\}^{\frac{2}{q}}\nonumber\\&\leq 2\|\u_0^n-\u_0\|_{\H}^4e^{2\e^2L^2T+8C\e qL^2T}.
		\end{align}
		Therefore, we easily have 
		\begin{align}
			&\left\{\E\left[\sup_{0\leq t\leq T\wedge\tau_{\e,M}}\left(\|\u^{\e}_n(t)-\u^{\e}(t)\|_{\H}^2\right)^q\right]\right\}^{\frac{2}{q}}\nonumber\\&\leq \left\{\E\left[\sup_{0\leq t\leq T\wedge\tau_{\e,M}}\left(e^{-k\e\int_0^{t}\|\u^{\e}(s)\|_{\V_p}^{\frac{2p}{2p-d}}\d s}\|\u^{\e}_n(t)-\u^{\e}(t)\|_{\H}^2\right)^qe^{qk\e\int_0^{t\wedge\tau_{\e,M}}\|\u^{\e}(s)\|_{\V_p}^{\frac{2p}{2p-d}}\d s}\right]\right\}^{\frac{2}{q}}\nonumber\\&\leq e^{2k\e M^{\frac{2}{2p-d}}T^{\frac{2p-d-2}{2p-d}}}\left\{\E\left[\sup_{0\leq t\leq T\wedge\tau_{\e,M}}\left(e^{-k\e\int_0^{t}\|\u^{\e}(s)\|_{\V_p}^{\frac{2p}{2p-d}}\d s}\|\u^{\e}_n(t)-\u^{\e}(t)\|_{\H}^2\right)^q\right]\right\}^{\frac{2}{q}}\nonumber\\&\leq 2e^{2k M^{\frac{2}{2p-d}}T^{\frac{2p-d-2}{2p-d}}}\|\u_0^n-\u_0\|_{\H}^4e^{2\e^2L^2T+8C\e qL^2T},
		\end{align}
		for $p\geq\frac{d}{2}+1$. For any fixed $M$, taking $q=\frac{2}{\e}$, we get 
		\begin{align}\label{3.38}
			&	\sup_{0<\e\leq 1}\e\log\mathbb{P}\left\{\sup_{0\leq t\leq T\wedge\tau_{\e,M}}\|\u^{\e}_n(t)-\u^{\e}(t)\|_{\H}^2>\delta\right\}\nonumber\\&\leq \sup_{0<\e\leq 1}\e\log\frac{\E\left[\sup\limits_{0\leq t\leq T\wedge\tau_{\e,M}}\left(\|\u^{\e}_n(t)-\u^{\e}(t)\|_{\H}^{2q}\right)\right]}{\delta^{q}}\nonumber\\&\leq 2k M^{\frac{2}{2p-d}}T^{\frac{2p-d-2}{2p-d}}+2L^2T+16CL^2T-2\log\delta+\log(2\|\u_0^n-\u_0\|_{\H}^4)\nonumber\\&\to-\infty,\ \text{ as }\ n\to\infty. 
		\end{align}
		Using Lemma \ref{lem3.7}, we infer that, for any $R>0$, there exists a constant $M>0$ such that for any $\e\in(0,1]$, the following inequality holds: 
		\begin{align}\label{340}
			\mathbb{P}\left\{\left(|\u^{\e}|_{\V_p,\wi\L^{r}}^{\H}(T)\right)^{2}>M\right\}\leq e^{-\frac{R}{\e}}. 
		\end{align}
		For such an $M$, combining \eqref{330} and \eqref{3.38}, there exists a positive integer $N$, such that for any $n\geq N$, 
		\begin{align}\label{341}
			\sup_{0<\e\leq 1}\e\log\mathbb{P}\left\{\sup_{0\leq t\leq T\wedge\tau_{\e,M}}\|\u^{\e}_n(t)-\u^{\e}(t)\|_{\H}^2>\delta,\ \left(|\u^{\e}|_{\V_p,\wi\L^{r}}^{\H}(T)\right)^{2}>M\right\}\leq -R. 
		\end{align}
		Combining \eqref{340} and \eqref{341}, one can obtain the existence of a positive integer $N$, such that for any $n\geq N$, $\e\in(0,1]$, 
		\begin{align}
			\sup_{0<\e\leq 1}\e\log\mathbb{P}\left\{\sup_{0\leq t\leq T}\|\u^{\e}_n(t)-\u^{\e}(t)\|_{\H}^2>\delta\right\}\leq -2R.
		\end{align}
		The arbitrariness of $R$ completes the proof for $p\geq\frac{d}{2}+1$ and $r\geq 2$. 
		
		\vskip 0.2 cm
		\noindent \textbf{Case 2}: $p\geq 2$ and $r\geq 4$.  For the case $p\geq 2$ and $r>4$, we first apply It\^o's formula to the process $\|\u^{\e}_n(\cdot)-\u^{\e}(\cdot)\|_{\H}^2$ to find for all $t\in[0,T]$,	$\mathbb{P}$-a.s., 
		\begin{align}\label{3p43}
			&\|\u^{\e}_n(t)-\u^{\e}(t)\|_{\H}^2-2\mu\e\int_0^{t}\langle\mathcal{A}(\u^{\e}_n(s))-\mathcal{A}(\u^{\e}(s)),\u^{\e}_n(s)-\u^{\e}(s)\rangle\d s\nonumber\\&\quad+2\beta\e\int_0^{t}\langle\mathcal{C}(\u^{\e}_n(s))-\mathcal{C}(\u^{\e}(s)),\u^{\e}_n(s)-\u^{\e}(s)\rangle\d s\nonumber\\&=\|\u_0^n-\u_0\|_{\H}^2-2\e\int_0^{t}\langle\B(\u^{\e}_n(s))-\B(\u^{\e}(s)),\u^{\e}_n(s)-\u^{\e}(s)\rangle\d s\nonumber\\&\quad+\e\sum_{k=1}^{\infty}\int_0^{t}\|\boldsymbol{\sigma}_k(\e s,\u^{\e}_n(s))-\boldsymbol{\sigma}_k(\e s,\u^{\e}(s))\|_{\H}^2\d s\nonumber\\&\quad+ 2\sqrt{\e}\sum_{k=1}^{\infty}\int_0^{t}\big((\boldsymbol{\sigma}_k(\e s,\u^{\e}_n(s))-\boldsymbol{\sigma}_k(\e s,\u^{\e}(s)))\d\boldsymbol{W}_k(s),\u^{\e}_n(s)-\u^{\e}(s)\big).
		\end{align}
		One can estimate $2\beta\e\langle\mathcal{C}(\u^{\e}_n)-\mathcal{C}(\u^{\e}),\u^{\e}_n-\u^{\e}\rangle$ as (see \eqref{2.23})
		\begin{align}\label{3p44}
			2&\beta\e\langle\mathcal{C}(\u^{\e}_n)-\mathcal{C}(\u^{\e}),\u^{\e}_n-\u^{\e}\rangle\nonumber\\&\geq \beta\e\||\u^{\e}_n|^{\frac{r-2}{2}}(\u^{\e}_n-\u^{\e})\|_{\H}^2+\beta\e\||\u^{\e}|^{\frac{r-2}{2}}(\u^{\e}_n-\u^{\e})\|_{\H}^2.
		\end{align}
		A calculation similar to \eqref{2.30} gives 
		\begin{align}\label{3p45}
			-&2\e\langle\B(\u^{\e}_n)-\B(\u^{\e}),\u^{\e}_n-\u^{\e}\rangle\nonumber\\&\leq \mu\e\|\nabla(\u^{\e}_n-\u^{\e})\|_{\H}^2+\beta\e\||\u^{\e}|^{\frac{r-2}{2}}(\u^{\e}_n-\u^{\e})\|_{\H}^2+\e\eta\|\u^{\e}_n-\u^{\e}\|_{\H}^2, 
		\end{align}
		where $\eta$ is defined in \eqref{eta}. Combining \eqref{3.31}, \eqref{3p44}-\eqref{3p45} and then substituting it in \eqref{3p43}, we find for all $t\in[0,T]$,	$\mathbb{P}$-a.s., 
		\begin{align}\label{3p46}
			&	\|\u^{\e}_n(t)-\u^{\e}(t)\|_{\H}^2\nonumber\\&\leq \|\u_0^n-\u_0\|_{\H}^2+\e(\eta+L)\int_0^{t}\|\u^{\e}_n(s)-\u^{\e}(s)\|_{\H}^2\d s\nonumber\\&\quad+ 2\sqrt{\e}\left|\sum_{k=1}^{\infty}\int_0^{t}\big((\boldsymbol{\sigma}_k(\e s,\u^{\e}_n(s))-\boldsymbol{\sigma}_k(\e s,\u^{\e}(s)))\d\boldsymbol{W}_k(s),\u^{\e}_n(s)-\u^{\e}(s)\big)\right|,
		\end{align}
		where we have used Hypothesis \ref{hyp} (H.2)  also. A calculation similar to \eqref{3p35} yields 
		\begin{align}
			&\left\{\E\left[\sup_{0\leq t\leq T}\|\u^{\e}_n(t)-\u^{\e}(t)\|_{\H}^{2q}\right]\right\}^{\frac{2}{q}}\nonumber\\&\leq 2\|\u_0^n-\u_0\|_{\H}^4+2\e^2(\eta+L)^2\int_0^T\left\{\E\left[\sup_{0\leq s\leq t}\|\u^{\e}_n(s)-\u^{\e}(s)\|_{\H}^{2q}\right]\right\}^{\frac{2}{q}}\d t\nonumber\\&\qquad+8C\e qL^2\int_0^T\left\{\E\left[\sup_{0\leq s\leq t}\|\u^{\e}_n(s)-\u^{\e}(s)\|_{\H}^{2q}\right]\right\}^{\frac{2}{q}}\d t. 
		\end{align}
		An application of Gronwall's inequality gives 
		\begin{align}
			&\left\{\E\left[\sup_{0\leq t\leq T}\|\u^{\e}_n(t)-\u^{\e}(t)\|_{\H}^{2q}\right]\right\}^{\frac{2}{q}}\leq 2\|\u_0^n-\u_0\|_{\H}^4 e^{2\e^2(\eta+L)^2+8C\e qL^2}.
		\end{align}
		Arguing similarly as in \eqref{3.27}, one can complete the proof. 
		
		For $r=4$ and $2\beta\mu\geq 1$, one can estimate $-2\e\langle\B(\u^{\e}_n)-\B(\u^{\e}),\u^{\e}_n-\u^{\e}\rangle$ as (see \eqref{232})
		\begin{align}\label{3p48}
			-&2\e\langle\B(\u^{\e}_n)-\B(\u^{\e}),\u^{\e}_n-\u^{\e}\rangle\leq 2\mu\e\|\nabla(\u^{\e}_n-\u^{\e})\|_{\H}^2+\frac{\e}{2\mu}\||\u^{\e}|(\u^{\e}_n-\u^{\e})\|_{\H}^2,
		\end{align}
		and the proof can be completed in a similar way as in the case of $p\geq 2$ and $r>4$.
	\end{proof}
	\begin{remark}\label{rem5.9}
		For the case $p\geq 2$ and $r\geq 4$,	it can be easily seen that the stopping time arguments are not needed in the proof of Lemma \ref{lem3.10}. 
	\end{remark}

	\begin{lemma}\label{lem3.11}
		For any $\delta>0$, and every positive integer $n$,
		\begin{align}
			\lim_{\e\to 0}\e\log\mathbb{P}\left\{\sup_{0\leq t\leq T}\|\u^{\e}_n(t)-\v^{\e}_n(t)\|_{\H}^2>\delta\right\}=-\infty. 
		\end{align}
	\end{lemma}
	\begin{proof}
		For $M>0$, let us define the following stopping times: 
		\begin{align*}
			\tau^{n,1}_{\e,M}&:=\inf_{t\geq 0}\left\{t:\e\int_0^t[\|\u^{\e}_n(s)\|_{\V}^2+\|\u^{\e}_n(s)\|_{\V_p}^p]\d s>M\ \text{ or }\ \|\u^{\e}_n(t)\|_{\H}^2>M\right\},\\
			\tau^{n,2}_{\e,M}&:=\inf_{t\geq 0}\left\{t:\|\nabla\v^{\e}_n(t)\|_{\H}^2>M \ \text{ or }\ \|\nabla\v^{\e}_n(t)\|_{\wi\L^p}^p>M\ \text{ or }\ \|\v^{\e}(t)\|_{\wi\L^r}^r>M\right\},
		\end{align*}
		and $\tau^{n}_{\e,M}=\tau^{n,1}_{\e,M}\wedge  \tau^{n,2}_{\e,M}$. 	Applying It\^o's formula to $\|\u^{\e}_n(\cdot)-\v^{\e}_n(\cdot)\|_{\H}^2$, we get for all $t\in[0,T]$,	$\mathbb{P}$-a.s., 
		\begin{align}\label{3.30}
			&	\|\u^{\e}_n(t\wedge\tau^n_{\e,M})-\v^{\e}_n(t\wedge\tau^n_{\e,M})\|_{\H}^2-	2\mu\e\int_0^{t\wedge\tau^n_{\e,M}}\langle\mathcal{A}(\u^{\e}_n(s))-\mathcal{A}(\v^{\e}_n(s)),\u^{\e}_n(s)-\v^{\e}_n(s)\rangle\d s\nonumber\\&\quad+2\beta\e\int_0^{t\wedge\tau^n_{\e,M}}\langle\mathcal{C}(\u^{\e}_n(s))-\mathcal{C}(\v^{\e}_n(s)),\u^{\e}_n(s)-\v^{\e}_n(s)\rangle\d s\nonumber\\&=2\mu\e\int_0^{t\wedge\tau^n_{\e,M}}\langle\mathcal{A}(\v^{\e}_n(s)),\u^{\e}_n(s)-\v^{\e}_n(s)\rangle\d s-2\e\int_0^{t\wedge\tau^n_{\e,M}}\langle\B(\u^{\e}_n(s)),\u^{\e}_n(s)-\v^{\e}_n(s)\rangle\d s\nonumber\\&\quad-2\beta\e\int_0^{t\wedge\tau^n_{\e,M}}\langle\mathcal{C}(\v^{\e}_n(s)),\u^{\e}_n(s)-\v^{\e}_n(s)\rangle\d s+2\e\int_0^{t\wedge\tau^n_{\e,M}}\langle\f(\e s),\u^{\e}_n(s)-\v^{\e}_n(s)\rangle\d s\nonumber\\&\quad+\e\sum_{k=1}^{\infty}\int_0^{t\wedge\tau^n_{\e,M}}\|\boldsymbol{\sigma}_k(\e s,\u^{\e}_n(s))-\boldsymbol{\sigma}_k(\e s,\v^{\e}_n(s))\|_{\H}^2\d s\nonumber\\&\quad+2\sqrt{\e}\sum_{k=1}^{\infty}\int_0^{t\wedge\tau^n_{\e,M}}\big((\boldsymbol{\sigma}_k(\e s,\u^{\e}_n(s))-\boldsymbol{\sigma}_k(\e s,\v^{\e}_n(s)))\d\boldsymbol{W}_k(s),\u^{\e}_n(s)-\v^{\e}_n(s)\big).
		\end{align}
		Note that $$\langle\B(\u),\u-\v\rangle=\langle \B(\u-\v,\u),\u-\v\rangle-\langle\B(\v,\u-\v),\v\rangle,$$ for all $\u,\v\in\V$.  Using \eqref{2.23} and \eqref{2.31} in \eqref{3.30}, we obtain for all $t\in[0,T]$,	$\mathbb{P}$-a.s., 
		\begin{align}\label{331}
			&	\|\u^{\e}_n(t\wedge\tau^n_{\e,M})-\v^{\e}_n(t\wedge\tau^n_{\e,M})\|_{\H}^2+\mu\e\int_0^{t\wedge\tau^n_{\e,M}}\|\u^{\e}_n(s)-\v^{\e}_n(s)\|_{\V}^2\d s\nonumber\\&\quad+\frac{\mu\e}{2}\int_0^{t\wedge\tau^n_{\e,M}}\||\nabla\u^{\e}_n(s)|^{\frac{p-2}{2}}\nabla(\u^{\e}_n(s)-\v^{\e}_n(s))\|_{\H}^2\d s\nonumber\\&\quad+\frac{\mu\e}{2}\int_0^{t\wedge\tau^n_{\e,M}}\||\nabla\v^{\e}_n(s)|^{\frac{p-2}{2}}\nabla(\u^{\e}_n(s)-\v^{\e}_n(s))\|_{\H}^2\d s\nonumber\\&\quad+\beta\e\int_0^{t\wedge\tau^n_{\e,M}}\||\u^{\e}_n(s)|^{\frac{r-2}{2}}(\u^{\e}_n(s)-\v^{\e}_n(s))\|_{\H}^2\d s\nonumber\\&\quad+\beta\e\int_0^{t\wedge\tau^n_{\e,M}}\||\v^{\e}_n(s)|^{\frac{r-2}{2}}(\u^{\e}_n(s)-\v^{\e}_n(s))\|_{\H}^2\d s\nonumber\\&\leq -2\mu\e\int_0^{t\wedge\tau^n_{\e,M}}\left(((1+|\nabla\v^{\e}_n(s)|^2)^{\frac{p-2}{2}})\nabla\v^{\e}_n(s),\nabla(\u^{\e}_n(s)-\v^{\e}_n(s))\right)\d s\nonumber\\&\quad-2\e\int_0^{t\wedge\tau^n_{\e,M}}\langle\B(\u^{\e}_n(s)-\v^{\e}_n(s),\u^{\e}_n(s)),\u^{\e}_n(s)-\v^{\e}_n(s)\rangle\d s\nonumber\\&\quad+2\e\int_0^{t\wedge\tau^n_{\e,M}}\langle\B(\v^{\e}_n(s),\u^{\e}_n(s)-\v^{\e}_n(s)),\v^{\e}_n(s)\rangle\d s\nonumber\\&\quad-2\beta\e\int_0^{t\wedge\tau^n_{\e,M}}\langle\mathcal{C}(\v^{\e}_n(s)),\u^{\e}_n(s)-\v^{\e}_n(s)\rangle\d s+2\e\int_0^{t\wedge\tau^n_{\e,M}}\langle\f(\e s),\u^{\e}_n(s)-\v^{\e}_n(s)\rangle\d s\nonumber\\&\quad+\e\sum_{k=1}^{\infty}\int_0^{t\wedge\tau^n_{\e,M}}\|\boldsymbol{\sigma}_k(\e s,\u^{\e}_n(s))-\boldsymbol{\sigma}_k(\e s,\v^{\e}_n(s))\|_{\H}^2\d s\nonumber\\&\quad+2\sqrt{\e}\sum_{k=1}^{\infty}\int_0^{t\wedge\tau^n_{\e,M}}\big((\boldsymbol{\sigma}_k(\e s,\u^{\e}_n(s))-\boldsymbol{\sigma}_k(\e s,\v^{\e}_n(s)))\d\boldsymbol{W}_k(s),\u^{\e}_n(s)-\v^{\e}_n(s)\big).
		\end{align}
		We estimate $-2\mu\e\left(((1+|\nabla\v^{\e}_n|^2)^{\frac{p-2}{2}})\nabla\v^{\e}_n,\nabla(\u^{\e}_n-\v^{\e}_n)\right)$ as 
		\begin{align}\label{332}
			&-2\mu\e\left(((1+|\nabla\v^{\e}_n|^2)^{\frac{p-2}{2}})\nabla\v^{\e}_n,\nabla(\u^{\e}_n-\v^{\e}_n)\right)\nonumber\\&\leq 2\mu\e\|((1+|\nabla\v^{\e}_n|^2)^{\frac{p-2}{4}})\nabla\v^{\e}_n\|_{\H}\|((1+|\nabla\v^{\e}_n|^2)^{\frac{p-2}{4}})\nabla(\u^{\e}_n-\v^{\e}_n)\|_{\H}\nonumber\\&\leq \mu\e 2^{\frac{p}{2}}\left(\|\nabla\v^{\e}_n\|_{\H}^2+\|\nabla\v^{\e}_n\|_{\wi\L^p}^p\right)^{\frac{1}{2}}\left(\|\nabla(\u^{\e}_n-\v^{\e}_n)\|_{\H}^2+\||\nabla\v^{\e}_n|^{\frac{p-2}{2}}\nabla(\u^{\e}_n-\v^{\e}_n)\|_{\H}^2\right)^{\frac{1}{2}}\nonumber\\&\leq\frac{\mu\e}{4}\|\u^{\e}_n-\v^{\e}_n\|_{\V}^2+\frac{\mu\e}{4}\||\nabla\v^{\e}_n|^{\frac{p-2}{2}}\nabla(\u^{\e}_n-\v^{\e}_n)\|_{\H}^2+2^{p+1}\mu\e\left(\|\nabla\v^{\e}_n\|_{\H}^2+\|\nabla\v^{\e}_n\|_{\wi\L^p}^p\right).
		\end{align}
		Similarly, we estimate $-2\beta\e\langle\mathcal{C}(\v^{\e}_n),\u^{\e}_n-\v^{\e}_n\rangle$ and $2\e\langle\f,\u^{\e}_n-\v^{\e}_n\rangle$ as 
		\begin{align}
			-2\beta\e\langle\mathcal{C}(\v^{\e}_n),\u^{\e}_n-\v^{\e}_n\rangle&\leq 2\beta\e\|\v^{\e}_n\|_{\wi\L^r}^{\frac{r}{2}}\||\v^{\e}_n|^{\frac{r-2}{2}}(\u^{\e}_n-\v^{\e}_n)\|_{\H}\nonumber\\&\leq \beta\e\||\v^{\e}_n|^{\frac{r-2}{2}}(\u^{\e}_n-\v^{\e}_n)\|_{\H}^2+\beta\e\|\v^{\e}_n\|_{\wi\L^r}^{r},\\ 2\e\langle\f,\u^{\e}_n-\v^{\e}_n\rangle&\leq 2\e\|\f\|_{\V_p'}\|\u^{\e}_n-\v^{\e}_n\|_{\V_p}\leq\frac{\mu\e}{2^{p-1}}\|\u^{\e}_n-\v^{\e}_n\|_{\V_p}^p+\wi C_{p,\mu}\e\|\f\|_{\V_p'}^{p'},\label{334}
		\end{align}
		where $\wi C_{p,\mu}=2^{\frac{2p-1}{p-1}}\frac{(p-1)}{p}\left(\frac{1}{\mu p}\right)^{\frac{1}{p-1}}$.
		\vskip 0.2 cm
		\noindent\textbf{Case 1:} $p\geq \frac{d}{2}+1$ and $r\geq 2$. 
		For the case $p\geq \frac{d}{2}+1$ and $r\geq 2$, an estimate similar to \eqref{225} gives 
		\begin{align}\label{337}
			-&2\e\langle\B(\u^{\e}_n-\v^{\e}_n,\u^{\e}_n),\u^{\e}_n-\v^{\e}_n\rangle \leq \frac{\mu\e}{4}\|\u^{\e}_n-\v^{\e}_n\|_{\V}^2+\wi\eta_1\e\|\u^{\e}_n\|_{\V_p}^{\frac{2p}{2p-d}}\|\u^{\e}_n-\v^{\e}_n\|_{\H}^2,
		\end{align}
		where $\wi\eta_1=(2C)^{\frac{2p}{2p-d}}\left(\frac{2p-d}{2p}\right)\left(\frac{d}{\mu p}\right)^{\frac{d}{2p-d}}$. By Sobolev's embedding, we know that $\H^1(\mathcal{O})\subset\L^4(\mathcal{O})$, for $\mathcal{O}\subset\R^d$ with $2\leq d\leq 4$. 
		Using H\"older's, Young's and Sobolev's inequalities, we estimate $2\e\langle\B(\v^{\e}_n,\u^{\e}_n-\v^{\e}_n),\v^{\e}_n\rangle$ as
		\begin{align}\label{338}
			2\e	\langle\B(\v^{\e}_n,\u^{\e}_n-\v^{\e}_n),\v^{\e}_n\rangle&\leq 2\e\|\u^{\e}_n-\v^{\e}_n\|_{\V}\|\v^{\e}_n\|_{\wi\L^4}^2\nonumber\\&\leq\frac{\mu\e}{4}\|\u^{\e}_n-\v^{\e}_n\|_{\V}^2+\frac{C\e}{\mu}\|\v_n^{\e}\|_{\V}^4.
		\end{align} 
		Combining \eqref{332}-\eqref{338} and then substituting it in \eqref{331}, we find for all $t\in[0,T]$,	$\mathbb{P}$-a.s., 
		\begin{align}
			&	\|\u^{\e}_n(t\wedge\tau^n_{\e,M})-\v^{\e}_n(t\wedge\tau^n_{\e,M})\|_{\H}^2\nonumber\\&\leq 2^{p+1}\mu\e\int_0^{t\wedge\tau^n_{\e,M}}\left(\|\nabla\v^{\e}_n(s)\|_{\H}^2+\|\nabla\v^{\e}_n(s)\|_{\wi\L^p}^p\right)\d s+\beta\e\int_0^{t\wedge\tau^n_{\e,M}}\|\v_n^{\e}(s)\|_{\wi\L^r}^r\d s\nonumber\\&\quad+\wi C_{p,\mu}\int_0^{\e t}\|\f(s)\|_{\V_p'}^{p'}\d s+\frac{C\e}{\mu}\int_0^{t\wedge\tau^n_{\e,M}}\|\v^{\e}_n(s)\|_{\V}^4\d s\nonumber\\&\quad+\wi\eta_1\e\int_0^{t\wedge\tau^n_{\e,M}}\|\u^{\e}_n(s)\|_{\V_p}^{\frac{2p}{2p-d}}\|\u^{\e}_n(s)-\v^{\e}_n(s)\|_{\H}^2\d s+\e L\int_0^{t\wedge\tau^n_{\e,M}}\|\u^{\e}_n(s)-\v^{\e}_n(s)\|_{\H}^2\d s\nonumber\\&\quad+2\sqrt{\e}\left|\sum_{k=1}^{\infty}\int_0^{t\wedge\tau^n_{\e,M}}\big((\boldsymbol{\sigma}_k(\e s,\u^{\e}_n(s))-\boldsymbol{\sigma}_k(\e s,\v^{\e}_n(s)))\d\boldsymbol{W}_k(s),\u^{\e}_n(s)-\v^{\e}_n(s)\big)\right|,
		\end{align}
		where we have used Hypothesis \ref{hyp} (H.2) also. Applying Gronwall's inequality, we deduce that for all $t\in[0,T]$,	$\mathbb{P}$-a.s., 
		\begin{align}
			&	\|\u^{\e}_n(t\wedge\tau^n_{\e,M})-\v^{\e}_n(t\wedge\tau^n_{\e,M})\|_{\H}^2\nonumber\\&\leq\Bigg\{2^{p+1}\mu\e\int_0^{t\wedge\tau^n_{\e,M}}\left(\|\nabla\v^{\e}_n(s)\|_{\H}^2+\|\nabla\v^{\e}_n(s)\|_{\wi\L^p}^p\right)\d s+\beta\e\int_0^{t\wedge\tau^n_{\e,M}}\|\v_n^{\e}(s)\|_{\wi\L^r}^r\d s\nonumber\\&\qquad+\wi C_{p,\mu}\int_0^{\e t}\|\f(s)\|_{\V_p'}^{p'}\d s+\frac{C\e}{\mu}\int_0^{t\wedge\tau^n_{\e,M}}\|\v^{\e}_n(s)\|_{\V}^4\d s\nonumber\\&\qquad+2\sqrt{\e}\left|\sum_{k=1}^{\infty}\int_0^{t\wedge\tau^n_{\e,M}}\big((\boldsymbol{\sigma}_k(\e s,\u^{\e}_n(s))-\boldsymbol{\sigma}_k(\e s,\v^{\e}_n(s)))\d\boldsymbol{W}_k(s),\u^{\e}_n(s)-\v^{\e}_n(s)\big)\right|\Bigg\}\nonumber\\&\quad\times\exp\left\{\e L+\wi\eta_1\e t^{\frac{2p-d-2}{2p-d}}\left(\int_0^{t\wedge\tau^n_{\e,M}}\|\u^{\e}_n(s)\|_{\V_p}^{p}\d s\right)^{\frac{2}{2p-d}}\right\}\nonumber\\&\leq e^{\e L+\wi\eta_1\e^{\frac{2p-d-2}{2p-d}} t^{\frac{2(2p-d-1)}{2p-d}}M^{\frac{2}{2p-d}}}\Bigg\{2^{p+1}\mu\e M t+\beta\e M t+\wi C_{p,\mu}\int_0^{\e t}\|\f(s)\|_{\V_p'}^{p'}\d s+\frac{C\e M^2t}{\mu}\nonumber\\&\quad +2\sqrt{\e}\left|\sum_{k=1}^{\infty}\int_0^{t\wedge\tau^n_{\e,M}}\big((\boldsymbol{\sigma}_k(\e s,\u^{\e}_n(s))-\boldsymbol{\sigma}_k(\e s,\v^{\e}_n(s)))\d\boldsymbol{W}_k(s),\u^{\e}_n(s)-\v^{\e}_n(s)\big)\right|\Bigg\},
		\end{align}
		for $p\geq 1+\frac{d}{2}$, where we have used the definition of stopping time also. Using the similar techniques as in Lemma \ref{lem3.9}, we obtain 
		\begin{align}\label{3.41}
			&	\left[\E\left(\sup_{0\leq s\leq T\wedge\tau^n_{\e,M}}\|\u^{\e}_n(s)-\v^{\e}_n(s)\|_{\H}^{2q}\right)\right]^{\frac{2}{q}}\nonumber\\&\leq e^{2\e L+2\wi\eta_1\e^{\frac{2(2p-d-2)}{2p-d}} T^{\frac{4(2p-d-1)}{2p-d}}M^{\frac{4}{2p-d}}}\Bigg\{2^{2p+3}\mu^2\e^2M^2T^2+2\beta^2\e^2M^2T^2\nonumber\\&\qquad+2\wi C_{p,\mu}^2\left(\int_0^{\e T}\|\f(s)\|_{\V_p'}^{p'}\d s\right)^{2}+\frac{C^2\e^2M^4T^2}{\mu^2}\nonumber\\&\qquad+16 Cq\e L^2\int_0^T\left[\E\left(\sup_{0\leq r\leq s\wedge\tau^n_{\e,M}}\|\u^{\e}_n(r)-\v^{\e}_n(r)\|_{\H}^{2q}\right)\right]^{\frac{2}{q}}\d s\Bigg\}.
		\end{align}
		Let us define $C_{\e,M,L,T}=e^{2\e L+2\wi\eta_1\e^{\frac{2(2p-d-2)}{2p-d}} T^{\frac{4(2p-d-1)}{2p-d}}M^{\frac{4}{2p-d}}}16Cq\e L^2$. Then an application of Gronwall's inequality in \eqref{3.41} yields 
		\begin{align}\label{3.42}
			&	\left[\E\left(\sup_{0\leq s\leq T\wedge\tau^n_{\e,M}}\|\u^{\e}_n(s)-\v^{\e}_n(s)\|_{\H}^{2q}\right)\right]^{\frac{2}{q}}\nonumber\\&\leq e^{2\e L+2\wi\eta_1\e^{\frac{2(2p-d-2)}{2p-d}} T^{\frac{4(2p-d-1)}{2p-d}}M^{\frac{4}{2p-d}}}\Bigg\{2^{2p+3}\mu^2\e^2M^2T^2+2\beta^2\e^2M^2T^2\nonumber\\&\qquad+2\wi C_{p,\mu}^2\left(\int_0^{\e T}\|\f(s)\|_{\V_p'}^{p'}\d s\right)^{2}+\frac{C^2\e^2M^4T^2}{\mu^2}\Bigg\}e^{C_{\e,M,L,T}}.
		\end{align}
		From Lemmas \ref{lem3.7} and \ref{lem3.8}, we infer that, for any $R>0$, there exists an $M>0$ such that 
		\begin{align}\label{3.43}
			\sup_{0<\e\leq 1}\e\log\mathbb{P}\left\{\left(|\u^{\e}_n|_{\V_p,\wi\L^{r}}^{\H}(T)\right)^2>M\right\}&\leq -R,\\ 
			\sup_{0<\e\leq 1}\e\log\mathbb{P}\left\{\sup_{0\leq t\leq T}\|\v_n^{\e}(t)\|_{\V}^2>M\right\}&\leq -R,
			\\ 
			\sup_{0<\e\leq 1}\e\log\mathbb{P}\left\{\sup_{0\leq t\leq T}\|\v_n^{\e}(t)\|_{\V_p}^p>M\right\}&\leq -R,\\ 
			\sup_{0<\e\leq 1}\e\log\mathbb{P}\left\{\sup_{0\leq t\leq T}\|\v_n^{\e}(t)\|_{\wi\L^{r}}^{r}>M\right\}&\leq -R.\label{3.46}
		\end{align}
		For such a constant $M$, taking $q=\frac{2}{\e}$ in \eqref{3.42}, we find 
		\begin{align}
			\e&\log\mathbb{P}\bigg\{\sup_{0\leq t\leq T}\|\u^{\e}_n(t)-\v^{\e}_n(t)\|_{\H}^2>\delta,\ \left(|\u^{\e}_n|_{\V_p,\wi\L^{r}}^{\H}(T)\right)^2 \leq M, \ \sup_{0\leq t\leq T}\|\v_n^{\e}(t)\|_{\V}^2\leq M, \nonumber\\ &\qquad\qquad  \sup_{0\leq t\leq T}\|\v_n^{\e}(t)\|_{\V_p}^p\leq M,\  \sup_{0\leq t\leq T}\|\v_n^{\e}(t)\|_{\wi\L^{r}}^{r}\leq M \bigg\}\nonumber\\&\leq \e\log\mathbb{P}\bigg\{\sup_{0\leq t\leq T\wedge\tau^n_{\e,\M}}\|\u^{\e}_n(t)-\v^{\e}_n(t)\|_{\H}^2>\delta\bigg\}\nonumber\\&\leq \log \left\{\left[\E\left(\sup_{0\leq s\leq t\wedge\tau^n_{\e,M}}\|\u^{\e}_n(s)-\v^{\e}_n(s)\|_{\H}^{2q}\right)\right]^{\frac{2}{q}}\right\}-\log\delta^2\nonumber\\&\leq 2\e L+2\wi\eta_1\e^{\frac{2(2p-d-2)}{2p-d}} T^{\frac{4(2p-d-1)}{2p-d}}M^{\frac{4}{2p-d}}+C_{\e,M,L}-\log\delta^2\nonumber\\&\quad+\log\left\{2^{2p+3}\mu^2\e^2M^2T^2+2\beta^2\e^2M^2T^2+2\wi C_{p,\mu}^2\left(\int_0^{\e T}\|\f(s)\|_{\V_p'}^{p'}\d s\right)^{2}+\frac{C^2\e^2M^4T^2}{\mu^2}\right\}\nonumber\\&\to-\infty,\ \text{ as }\ \e\to 0. 
		\end{align}
		Therefore, there exists an $\e_0$ such that for any $\e$ satisfying $0<\e\leq\e_0$, 
		\begin{align}\label{3.48}
			\e&\log\mathbb{P}\bigg\{\sup_{0\leq t\leq T}\|\u^{\e}_n(t)-\v^{\e}_n(t)\|_{\H}^2>\delta,\ \left(|\u^{\e}_n|_{\V_p,\wi\L^{r}}^{\H}(T)\right)^2 \leq M, \ \sup_{0\leq t\leq T}\|\v_n^{\e}(t)\|_{\V}^2\leq M, \nonumber\\ &\qquad\qquad  \sup_{0\leq t\leq T}\|\v_n^{\e}(t)\|_{\V_p}^p\leq M,\  \sup_{0\leq t\leq T}\|\v_n^{\e}(t)\|_{\wi\L^{r}}^{r}\leq M \bigg\}\leq -R. 
		\end{align}
		Combining \eqref{3.43}-\eqref{3.46} and \eqref{3.48}, we infer that there exists a constant $\e_0$ such that for any $\e$ satisfying $0<\e\leq\e_0$, 
		\begin{align}\label{3.49}
			\e&\log\mathbb{P}\bigg\{\sup_{0\leq t\leq T}\|\u^{\e}_n(t)-\v^{\e}_n(t)\|_{\H}^2>\delta\bigg\}\leq -5 R. 
		\end{align}
		The arbitrariness of $R$ completes the proof. 
		
		\vskip 0.2 cm
		\noindent\textbf{Case 2:} $p\geq 2$ and $r\geq 4$. 
		For $p\geq 2$ and $r>4$, an estimate similar to \eqref{2.30} yields 
		\begin{align}
			-&2\e\langle\B(\u^{\e}_n-\v^{\e}_n,\u^{\e}_n),\u^{\e}_n-\v^{\e}_n\rangle \nonumber\\&\leq \frac{\mu\e}{4}\|\u^{\e}_n-\v^{\e}_n\|_{\V}^2+\beta\e\||\u^{\e}_n|^{\frac{r-2}{2}}(\u^{\e}_n-\v^{\e}_n)\|_{\H}^2+\e\eta_1\|\u^{\e}_n-\v^{\e}_n\|_{\H}^2,
		\end{align}
		where $\eta_1=\frac{(r-4)}{(r-2)}\left(\frac{8}{\beta\mu (r-2)}\right)^{\frac{2}{r-4}}$. For $p\geq 2,$ $\beta\mu>1$  and $r=4$, we estimate $-2\e\langle\B(\u^{\e}_n-\v^{\e}_n,\u^{\e}_n),\u^{\e}_n-\v^{\e}_n\rangle $ using H\"older's and Young's inequalities as 
		\begin{align}
			-&2\e\langle\B(\u^{\e}_n-\v^{\e}_n,\u^{\e}_n),\u^{\e}_n-\v^{\e}_n\rangle \leq {\theta\mu\e}\|\u^{\e}_n-\v^{\e}_n\|_{\V}^2+\frac{\e}{\theta\mu}\||\u^{\e}_n|(\u^{\e}_n-\v^{\e}_n)\|_{\H}^2,
		\end{align}
		for some $0<\theta<1$. The rest of the proof can be completed in a similar fashion as in the case of $p\geq \frac{d}{2}+1$ and $r\geq 2$.
	\end{proof}

	\begin{proof}[Proof of \eqref{3p8}]
		Making use of Lemmas \ref{lem3.9} and \ref{lem3.10}, we have for any $R > 0$, there exists an $N_0$ such that 
		\begin{align*}
			\mathbb{P}\left\{\sup_{0\leq t\leq T}\|\u^{\e}(t)-\u^{\e}_{N_0}(t)\|_{\H}^2>\delta \right\}\leq e^{-R/\e},\ \text{ for any } \ \e\in(0,1],
		\end{align*}
		and 
		\begin{align*}
			\mathbb{P}\left\{\sup_{0\leq t\leq T}\|\v^{\e}(t)-\v^{\e}_{N_0}(t)\|_{\H}^2>\delta \right\}\leq e^{-R/\e},\ \text{ for any } \ \e\in(0,1].
		\end{align*}
		From Lemma \ref{lem3.11}, we infer that for such $N_0$, there exists na $\e_0$ such that 
		\begin{align*}
			\mathbb{P}\left\{\sup_{0\leq t\leq T}\|\u^{\e}_{N_0}(t)-\v^{\e}_{N_0}(t)\|_{\H}^2>\delta \right\}\leq e^{-R/\e},\ \text{ for any } \ \e\in(0,\e_0].
		\end{align*}
		Therefore, for any $\e\in(0,\e_0]$, we deduce that 
		\begin{align*}
			\mathbb{P}\left\{\sup_{0\leq t\leq T}\|\u^{\e}(t)-\v^{\e}(t)\|_{\H}^2>\delta \right\}\leq 3e^{-R/\e}.
		\end{align*}
		The arbitrariness of $R$ implies that 
		\begin{align*}
			\lim_{\e\to 0}\e\log 	\mathbb{P}\left\{\sup_{0\leq t\leq T}\|\u^{\e}(t)-\v^{\e}(t)\|_{\H}^2>\delta \right\}=-\infty,
		\end{align*}
		which completes the proof. 
	\end{proof}

	\medskip\noindent
	{\bf Acknowledgments:}  A. Kumar  would like to thank the Austrian Science Foundation (FWF) (Stochastic Turing Patters, Project Number: P 24681-N) for financial assistance. M. T. Mohan would like to thank the Department of Science and Technology (DST) Science \& Engineering Research Board (SERB), India for a MATRICS grant (MTR/2021/000066).
	
	\medskip\noindent	{\bf  Declarations:} 
	
	\noindent 	{\bf  Ethical Approval:}   Not applicable 
	
	\noindent  {\bf   Competing interests: } The author declare no competing interests. 
	
	\noindent  {\bf  Author contributions statement: } All authors contributed equally. 
	
	\noindent 	{\bf   Funding: } Stochastic Turing Patters (P 24681-N), Austrian Science Found (FWF) (Ankit Kumar).\\
	MATRICS grant, DST-SERB, India (MTR/2021/000066). (Manil T. Mohan).

	\noindent 	{\bf   Availability of data and materials: } Not applicable.

\end{document}